\documentclass[a4paper,11pt, reqno]{amsart}

\usepackage{amsmath,amssymb,amsthm}
\usepackage{txfonts, mathrsfs}
\usepackage[utf8]{inputenc}
\usepackage{a4}
\usepackage[english]{babel}
\usepackage{enumerate}
\usepackage{graphicx}
\usepackage{bbm}
\usepackage{color}
\usepackage[hidelinks,
  bookmarks=true,
  bookmarksopen=true,
  bookmarksnumbered=true,
  pdfauthor={Giuliano Basso, Stefan Wenger, Robert Young},
  pdftitle={title},
  pdfsubject={Paper},
  pdfstartview=FitH,
  colorlinks=false,
  pdfencoding=unicode
]{hyperref}

\usepackage{enumitem}

\usepackage[colorinlistoftodos]{todonotes} 

\usepackage{comment}

\numberwithin{equation}{section}

\newtheorem{theorem}{Theorem}[section]
\newtheorem{proposition}[theorem]{Proposition}
\newtheorem{lemma}[theorem]{Lemma}

\newtheorem{corollary}[theorem]{Corollary}

\theoremstyle{definition}
\newtheorem{remark}[theorem]{Remark}
\newtheorem{definition}[theorem]{Definition}
\newtheoremstyle{customNumber}
     {}          
     {}          
     {\itshape}  
     {}          
     {\bfseries} 
     {.}         
     { }         
     {\thmname{#1}\thmnumber{ #2}\thmnote{ #3}}
\theoremstyle{customNumber}

\renewcommand{\phi}{\varphi}
\renewcommand{\rho}{\varrho}

\newcommand{\norm}[1]{\lVert#1\rVert}
\newcommand{\abs}[1]{\lvert#1\rvert}
\newcommand{\diam}{\operatorname{diam}}

\newcommand{\on}{\:\mbox{\rule{0.1ex}{1.2ex}\rule{1.1ex}{0.1ex}}\:}
\newcommand{\bb}[1]{\llbracket #1\rrbracket}

\DeclareMathOperator{\CAT}{CAT}
\DeclareMathOperator{\EI}{EI}
\DeclareMathOperator{\CI}{CI}
\DeclareMathOperator{\LC}{LC}

\DeclareMathOperator{\N}{\mathbb{N}}
\DeclareMathOperator{\R}{\mathbb{R}}
\DeclareMathOperator{\Z}{\mathbb{Z}}
\DeclareMathOperator{\Hh}{\mathbb{H}}
\DeclareMathOperator{\Lip}{Lip}
\DeclareMathOperator{\lip}{lip}

\DeclareMathOperator{\spt}{spt}
\DeclareMathOperator{\st}{st}
\DeclareMathOperator{\interior}{int}
\DeclareMathOperator{\mass}{\mathbf{M}}
\DeclareMathOperator{\bI}{\mathbf{I}}
\DeclareMathOperator{\bM}{\mathbf{M}}
\DeclareMathOperator{\bN}{\mathbf{N}}
\DeclareMathOperator{\id}{id}
\DeclareMathOperator{\Fillvol}{Fillvol}
\DeclareMathOperator{\Hull}{Hull}
\DeclareMathOperator{\poly}{\mathscr{P}}

\DeclareMathOperator{\flatnorm}{\mathscr{F}} 
 
\DeclareMathOperator{\Haus}{\mathscr{H}}

\title{Undistorted fillings in subsets of metric spaces}

\author{Giuliano Basso}

\address{Department of Mathematics\\ University of Fribourg\\ Chemin du Mus\'ee 23\\ 1700 Fribourg, Switzerland}
\email{giuliano.basso@unifr.ch}

\author{Stefan Wenger}

\address{Department of Mathematics\\ University of Fribourg\\ Chemin du Mus\'ee 23\\ 1700 Fribourg, Switzerland}
\email{stefan.wenger@unifr.ch}

\author{Robert Young}

\address{Courant Institute of Mathematical Sciences\\ New York University\\ 251 Mercer St.\\ New York, NY  10012, USA}
\email{ryoung@cims.nyu.edu}

\thanks{G.~B.~and S.~W.~were supported by Swiss National Science Foundation grant 182423. R.~Y.~was supported by National Science Foundation grant 2005609.}

\date{\today}

\begin{document}

\pdfbookmark[0]{Undistorted fillings in subsets of metric spaces}{titleLabel}

\begin{abstract}
 We prove that if a quasiconvex subset $X$ of a metric space $Y$ has finite Nagata dimension and is Lipschitz $k$-connected or admits Euclidean isoperimetric inequalities up to dimension $k$ for some $k$ then $X$ is isoperimetrically undistorted in $Y$ up to dimension $k+1$. This generalizes and strengthens a recent result of the third named author and has several consequences and applications. It yields for example that in spaces of finite Nagata dimension, Lipschitz connectedness implies Euclidean isoperimetric inequalities, and Euclidean isoperimetric inequalities imply coning inequalities. It furthermore allows us to prove an analog of the Federer-Fleming deformation theorem in spaces of finite Nagata dimension admitting Euclidean isoperimetric inequalities.
\end{abstract}

\maketitle

%
\section{Introduction}

\subsection{Overview}

Isoperimetric inequalities measure how difficult it is to fill (Lipschitz) cycles in a given space by (Lipschitz) chains of one dimension higher. They are important in many branches of mathematics and play a crucial role in particular in asymptotic geometry and geometric group theory, where they appear as Dehn functions and higher filling functions and are quasi-isometry invariants of the underlying space or group. 
In this article, we study the relationship between isoperimetric inequalities and extension properties, especially Lipschitz connectivity and coning inequalities. 

A space $X$ is Lipschitz $k$-connected if Lipschitz maps from $k$-spheres to $X$ can be extended to Lipschitz maps from $(k+1)$-balls. This is a key ingredient in constructing Lipschitz extensions in general; see \cite{MR146835}, \cite{MR2200122}. A $1$-dimensional Lipschitz cycle is a sum of closed Lipschitz curves, so if $X$ is Lipschitz $1$-connected, it can be filled by filling each curve by a Lipschitz disc. Higher-dimensional Lipschitz cycles, however, can have complicated topology and it may not be possible to decompose them as sums of Lipschitz spheres. For this reason, it is often much more difficult to prove higher dimensional isoperimetric inequalities, and it is open in general whether Lipschitz $k$-connectedness implies a $k$-dimensional isoperimetric inequality of Euclidean type.

Coning inequalities bound the filling volume of a cycle in terms of its mass and diameter. They were introduced by Gromov in his seminal article \cite{MR697984}, where it was shown that Riemannian manifolds with coning inequalities admit Euclidean isoperimetric inequalities. This result was later generalized to complete metric spaces by the second named author in \cite{MR2153909}. Coning inequalities have also played a crucial role in recent articles on higher rank hyperbolicity \cite{MR4121159}, Morse quasiflats \cite{KleinerStadler2019MoresQuasiflatsI}, \cite{KleinerStadler2020MoresQuasiflatsII}, and the equivalence of flat and weak convergence of currents \cite{MR2284563}. It is open in general whether spaces with $k$-dimensional Euclidean isoperimetric inequalities admit $k$-dimensional coning inequalities. This is clear in dimension $1$, where one can fill a Lipschitz closed curve with diameter $d$ and length $L$ by decomposing it into $L/d$ closed curves of length at most $3d$, but it is not clear whether higher-dimensional Lipschitz cycles can be decomposed in the same way.

In this paper, we will show that Lipschitz connectivity implies Euclidean isoperimetric inequalities and that Euclidean isoperimetric inequalities imply coning inequalities when the underlying space has finite Nagata dimension, thus establishing a partial converse to the results mentioned above. Nagata dimension can be thought of as a quantitative metric version of topological dimension and is closely related to Gromov's asymptotic dimension \cite{MR1253544}. These results are consequences of a more general theorem about isoperimetric subspace distortion. Isoperimetric distortion of subspaces was briefly addressed by Gromov in \cite{MR1253544} and has recently been studied in the articles \cite{MR2270455}, \cite{MR3268779}, \cite{LeuzingerYoung-distortionDimension-2017}, \cite{MR4250389} in connection with conjectures of Thurston, Gromov, and Bux-Wortman. Roughly speaking, a subspace $X$ of a metric space $Y$ is (isoperimetrically) undistorted up to dimension $k+1$ if $m$-cycles in $X$ with $0\leq m\leq k$ can be filled almost as efficiently by $(m+1)$-chains in the subspace $X$ as they can be filled in the ambient space $Y$. Our theorem about isoperimetric distortion asserts that if a quasiconvex metric space $X$ of finite Nagata dimension is Lipschitz $k$-connected or has Euclidean isoperimetric inequalities up to dimension $k$ then $X$ is undistorted up to dimension $k+1$ in any ambient space $Y$. This generalizes and strengthens a recent result of the third named author \cite{MR3268779}. Since any metric space embeds isometrically in a Banach space and Banach spaces admit coning inequalities and hence Euclidean isoperimetric inequalities the relationships between Lipschitz connectedness, Euclidean isoperimetric inequalities, and coning inequalities described above follow.   

Besides these structural results, our theorem on isoperimetric subspace distortion also allows for a Federer-Fleming type deformation theorem in quasiconvex metric spaces of finite Nagata dimension that admit Euclidean isoperimetric inequalities. The deformation theorem and its variants approximate chains or currents in Euclidean space or a simplicial complex by polyhedral chains in a cubical lattice or triangulation. Metric spaces of finite Nagata dimension need not admit a triangulation, but we show that any quasiconvex space \(X\) of finite Nagata dimension can be approximated by simplicial complexes. When $X$ has a Euclidean isoperimetric inequality, we can use these approximations to define substitutes for polyhedral chains in $X$. Moreover, we obtain a Federer-Fleming type deformation theorem in $X$ by combining the classical deformation theorem in the simplicial approximation and the isoperimetric subspace distortion theorem in $X$. As an application, we show that when $X$ has finite Nagata dimension and is Lipschitz connected, integral currents in $X$ can be approximated in total mass by Lipschitz chains.

\subsection{Undistorted fillings and first applications}
Given a complete metric space $X$ and $k\geq 0$, we denote by $\bI_k(X)$ the abelian group of $k$-dimensional metric integral currents in $X$ in the sense of Ambrosio-Kirchheim \cite{MR1794185}. See Sections~\ref{sec:prelims} and \ref{sec:currents} for the definitions and concepts used throughout this introduction. Every Lipschitz $k$-chain in $X$ induces a $k$-dimensional metric integral current in $X$ and we encourage the geometrically minded reader to think of integral currents as suitable limits of Lipschitz chains. The filling volume in $X$ of $T\in\bI_k(X)$ is defined by
\[
\Fillvol_X(T)=\inf\big\{ \mass(S) : S\in \bI_{k+1}(X) \text{ and } \partial S=T \big\},
\] 
where $\mass(S)$ denotes the mass of $S$ and $\partial S$ is its boundary. 

We say that $X$ has $(\EI_k)$ for some $k\geq 0$ if there exists $D>0$ such that for every $0\leq m\leq k$ and every cycle $T\in\bI_m(X)$  we have $$\Fillvol_X(T)\leq D\mass(T)^{\frac{m+1}{m}}.$$ This means that $X$ has Euclidean isoperimetric inequalities in dimensions $m=1,\dots, k$ whenever $k\geq 1$. Condition $(\EI_0)$ is always true and is only included to shorten the statements of our theorems.

We say that $X$ has $(\LC_k)$ or that $X$ is Lipschitz $k$-connected if there exists $c\geq 1$ such that for every $0\leq m\leq k$ each \(L\)-Lipschitz map from the Euclidean $m$-sphere to $X$ extends to a $c L$-Lipschitz map defined on the $(m+1)$-ball.
Examples of Lipschitz $k$-connected metric spaces are compact $k$-connected Riemannian manifolds, ${\rm CAT}(0)$-spaces and, more generally, spaces with a convex bicombing. This class also includes some Carnot groups such as the $(2k+1)$-dimensional Heisenberg group.

A closed subset \(X \subset Y\) of a complete metric space \(Y\) is said to be (isoperimetrically) \textit{undistorted in} \(Y\) \textit{up to dimension} \(k+1\), for some \(k\geq 0\), if there exists \(C>0\) such that for all cycles \(T\in \bI_m(X)\) with  \(0\leq m\leq k\) one has
\[
\Fillvol_X(T)\leq C \Fillvol_Y(T).
\]
Our definition of undistorted does not involve additive error terms and thus differs slightly from the definitions introduced in \cite{MR2270455}, \cite{MR3268779}, \cite{MR4250389}.
The following theorem is one of the main results in the present paper.

\begin{theorem}\label{thm:intro-undistorted-main}
 Let $Y$ be a complete metric space and $X\subset Y$ a closed quasiconvex subset of finite Nagata dimension. If $X$ has \((\LC_k)\) or \((\EI_k)\) for some $k\geq 0$, then $X$ is undistorted in $Y$ up to dimension $k+1$. The distortion constant only depends on the data of $X$.
\end{theorem}

For the definition of Nagata dimension see Section~\ref{subsec:Nagata} and \cite{MR2200122}. Examples of metric spaces of finite Nagata dimension include compact Riemannian manifolds, homogeneous or negatively pinched Hadamard manifolds, Carnot groups and, more generally, equiregular sub-Riemannian manifolds, doubling metric spaces and many more. Moreover, products, finite unions, and subsets of spaces of finite Nagata dimension have finite Nagata dimension.

In Section~\ref{Sec:variants-main-theorem} we will discuss variants and generalizations of Theorem~\ref{thm:intro-undistorted-main} in which we relax the condition that $X$ have finite Nagata dimension and in addition obtain control on the support of fillings in $X$.
Theorem~\ref{thm:intro-undistorted-main} generalizes and strengthens  \cite[Theorem 1.3]{MR3268779} by the third named author, which was used to prove isoperimetric inequalities for subsets of symmetric spaces. Notice that in \cite{MR3268779} the ambient space $Y$ is assumed to have finite Nagata dimension. In contrast, we do not pose any restrictions on $Y$. Since $X$ embeds isometrically into a Banach space and since every Banach space admits Euclidean isoperimetric inequalities and coning inequalities in all dimensions, see \cite{MR2153909}, we in particular obtain the following consequences.

\begin{corollary}\label{cor:lip-implies-EI}
 Let $X$ be a complete metric space of finite Nagata dimension. If $X$ has $(\LC_k)$ for some $k\geq 1$, then $X$ has $(\EI_k)$.
\end{corollary}

When $k=1$ the corollary holds without the assumption that $X$ have finite Nagata dimension and is much easier to prove. This is because integral $1$-cycles can be decomposed into the sum of closed Lipschitz curves, see \cite{MR2358060}. The corollary is furthermore known when $X$ has $(\LC_k)$ and is of Nagata dimension at most $k$ because in this case $X$ is an absolute Lipschitz retract by \cite{MR2200122}. Thus, the corollary is most interesting for spaces which are not Lipschitz connected up to the Nagata dimension. Typical examples with such a behavior are Carnot groups such as the higher Heisenberg groups; see the paragraph after the next corollary. We mention that the $(\LC_k)$ property is preserved under various constructions such as taking products, or passing to ultralimits or asymptotic cones. In comparison, it is not known whether the $(\EI_k)$ condition is preserved under these constructions.

In order to formulate the second consequence of the theorem, recall that a complete metric space $X$ is said to admit coning inequalities up to dimension $k$, or has $(\CI_k)$ for short, if there exists $C>0$ such that for every $0\leq m\leq k$ and every cycle $T\in\bI_m(X)$ of bounded support $\spt T$ we have $$\Fillvol_X(T)\leq C\diam(\spt T)\mass(T).$$ It is well-known that in a complete metric space $(\CI_k)$ implies $(\EI_k)$. This was proved by Gromov \cite{MR697984} for Riemannian manifolds and extended to metric spaces in \cite{MR2153909}. The following corollary gives a partial converse.

\begin{corollary}\label{cor:EI-implies-CI}
 Let $X$ be a complete quasiconvex metric space of finite Nagata dimension. If $X$ has $(\EI_k)$ for some $k\geq 1$, then $X$ has $(\CI_k)$.
\end{corollary}

Using a suitable variant of Theorem~\ref{thm:intro-undistorted-main} established in Section~\ref{Sec:variants-main-theorem} we will actually get a stronger version of the coning inequality in which we also obtain control over the diameter of the filling chain; see Corollary~\ref{cor:isop-lip-implies-cone}. Coning inequalities play an important role in the recent articles \cite{KleinerStadler2019MoresQuasiflatsI},  \cite{KleinerStadler2020MoresQuasiflatsII}, \cite{MR4121159}, and the stronger version of the coning inequality established in Corollary~\ref{cor:isop-lip-implies-cone} prominently appears in \cite{KleinerStadler2019MoresQuasiflatsI},  \cite{KleinerStadler2020MoresQuasiflatsII}, where it is called strong coning inequality.

We now briefly discuss some examples of spaces to which Corollaries~\ref{cor:lip-implies-EI} and \ref{cor:EI-implies-CI} apply. Compact $k$-connected Riemannian manifolds have \((\LC_k)\) and finite Nagata dimension by \cite{MR2200122} and therefore have $(\CI_k)$ and $(\EI_k)$. The $n$-th Heisenberg group $\Hh^n$ of topological dimension $2n+1$, equipped with a left-invariant sub-Riemannian or sub-Finsler metric $d_c$, has Nagata dimension $2n+1$ by \cite{MR3320519} and satisfies \((\LC_{n-1})\) by \cite{MR2729637}. It thus follows from the corollaries above that $(\Hh^n, d_c)$ has $(\CI_{n-1})$ and  $(\EI_{n-1})$. Previously, it was known that $\Hh^n$, when equipped with a left-invariant Riemannian metric, admits Euclidean isoperimetric inequalities up to dimension $n-1$ for Lipschitz cycles, see  \cite{MR3134033}. Together with \cite{MR2783380} this also yields that $(\Hh^n, d_c)$ has $(\EI_{n-1})$ for compactly supported integral currents.

We finally mention that Theorem~\ref{thm:intro-undistorted-main} can be reformulated in terms of the {\it absolute filling volume} of a cycle $T\in\bI_k(X)$ in a complete metric space $X$ defined by 
$$
\Fillvol_\infty(T)\coloneqq \inf\bigl\{ \Fillvol_Y(T): \text{$Y$ complete metric space containing $X$} \bigr\},
$$ compare with Gromov's filling volume of abstract Riemannian manifolds in \cite{MR697984}. Our theorem then says that if $X$ is a complete, quasiconvex metric space of finite Nagata dimension which has $(\LC_k)$ or $(\EI_k)$ for some $k\geq 0$ then every cycle $T\in\bI_k(X)$ satisfies $$\Fillvol_X(T)\leq C\Fillvol_\infty(T)$$ for some $C>0$ only depending on the data of $X$. 
It is easy to see that if $Y$ is an injective metric space containing $X$ then 
$\Fillvol_\infty(T) = \Fillvol_Y(T)$ for every cycle $T\in\bI_k(X)$. Particular examples of such spaces $Y$ are the Banach space $\ell_\infty(X)$ of bounded functions on $X$ equipped with the sup norm or the injective hull of $X$; see \cite{Isbell-injective-1964}.

\subsection{Deformation theorem in spaces of finite Nagata dimension}
The classical deformation theorem of Federer-Fleming, see \cite{MR123260} or \cite{MR0257325}, shows that integral $k$-currents in Euclidean $\R^n$ can be approximated by polyhedral $k$-chains in the $k$-skeleton of the standard cubical subdivision of $\R^n$. This result can be generalized to the setting of Riemannian manifolds with a geometric group action and, more generally, to metric spaces admitting a bilipschitz triangulation. The deformation theorem and its variants have been of fundamental importance in geometric measure theory and other fields, for example in the context of Dehn functions and higher filling functions;  see, for example, ~\cite{MR1161694}, \cite{MR1934388}, \cite{MR3268779}, \cite{MR3134033}, \cite{Abrams-et-al-Artin-groups-fillings-2013}. Metric spaces of finite Nagata dimension do not  in general admit a bilipschitz triangulation. Nevertheless, using a variant of Theorem~\ref{thm:intro-undistorted-main} and the methods developed in its proof, we can establish an analogue of the Federer-Fleming deformation theorem in spaces of finite Nagata dimension. 

We first introduce a substitute for a cubical subdivision in our setting. 
Given a metric simplicial complex $\Sigma$ and $m\geq 0$, let $\poly_m(\Sigma)$ be the abelian group of polyhedral $m$-chains in $\Sigma$, that is, finite sums of the form $\sum \theta_i \llbracket \sigma_i\rrbracket$, where $\theta_i\in\Z$ and $\llbracket \sigma_i\rrbracket$ is the current induced by the (oriented) $m$-simplex $\sigma_i\subset \Sigma$. Let $X$ be a complete metric space and let $k\in\N$ and $C, \varepsilon>0$. Suppose there exist a finite-dimensional metric simplicial complex $\Sigma$ and a map $\varphi\colon\Sigma^{(0)}\to X$, defined on the $0$-skeleton $\Sigma^{(0)}$ of $\Sigma$, with the following properties:
\begin{enumerate}
    \item the metric on $\Sigma$ is a length metric and each simplex in $\Sigma$ is a Euclidean simplex of side length $\varepsilon$;
    \item $\varphi$ is a quasi-isometry; more precisely, the image of $\varphi$ is $(C\varepsilon)$-dense in $X$ and
    $$d(z,w) - C\varepsilon \leq d(\varphi(z),\varphi(w)) \leq Cd(z,w)\quad\text{ for all $z,w\in\Sigma^{(0)}$.}$$
\end{enumerate}
Suppose also that there are homomorphisms $\Lambda_m\colon\poly_m(\Sigma)\to \bI_m(X)$, $m=0,\dots, k$, subject to: 
\begin{enumerate}
        \item[(3)] $\partial\circ\Lambda_{m+1}=\Lambda_m\circ\partial$ and $\Lambda_0(\llbracket z\rrbracket) = \llbracket \varphi(z)\rrbracket$ for every $z\in\Sigma^{(0)}$; 
          \item[(4)] $\Lambda_m(\bb{\sigma})$ has support in $B\bigl(\phi(\sigma^{(0)}), C\varepsilon\bigr)$ and $\mass(\Lambda_m(\llbracket\sigma\rrbracket))\leq C \varepsilon^m$
     for every $m$-simplex $\sigma$ in $\Sigma$.
   
\end{enumerate} 
Any triple $(\Sigma, \varphi,\Lambda_*)$, where $\Sigma$, $\varphi$ and $\Lambda_* = \{\Lambda_0,\dots, \Lambda_k\}$ satisfy the properties listed above, is called $(k,\varepsilon)$-polyhedral structure on $X$ with constant \(C\).
If such a polyhedral structure is given we set $\poly_m(X)\coloneqq \bigl\{\Lambda_m(Q): Q\in\poly_m(\Sigma)\bigr\}$ and call its elements {\it polyhedral $m$-chains in $X$}.
 
\begin{theorem}\label{thm:Def-thm-finite-Nagata-dim}
Let $X$ be a complete quasiconvex metric space of finite Nagata dimension which has $(\LC_k)$ or $ (\EI_k)$ for some $k\geq 1$. Then there exists $C>0$ such that for every $\varepsilon>0$ there is a $(k,\varepsilon)$-polyhedral structure on $X$ with constant $C$ such that the following holds. For every  $m=1,\dots, k$ and every $T\in\bI_m(X)$ there exist $P\in\poly_m(X)$, $R\in\bI_m(X)$, and $S\in\bI_{m+1}(X)$ with $T=P+R+\partial S$ and such that 
 \begin{alignat*}{2}
  \mass(P)&\leq C \mass(T),&\quad \mass(\partial P)&\leq C\mass(\partial T),\\
 \mass(S) &\leq \varepsilon\,C\mass(T) + \varepsilon^2\, C\mass(\partial T), & \quad   \mass(R)&\leq \varepsilon\,C\mass(\partial T),
 \end{alignat*}\
 as well as $\spt P$, $\spt S\subset B(\spt T, C\varepsilon)$ and $\spt \partial P$, $\spt R\subset B(\spt\partial T, C\varepsilon)$.
\end{theorem}

The bounds on the mass of $P$, $\partial P$, and $R$ in Theorem~\ref{thm:Def-thm-finite-Nagata-dim} are as in the deformation theorems of Simon \cite{MR756417} and White \cite{MR1738045}. In particular, this gives a slightly sharper bound on the mass of \(P\) than the original version of Federer and Fleming, which only gives a bound $\mass(P) \leq C[\mass(T) + \varepsilon \mass(\partial T)]$. The bound on the mass of $S$ includes an additional error term coming from a use of Theorem~\ref{thm:intro-undistorted-main}. This error term disappears however when $T$ is a cycle. If $X$ has $(\LC_k)$ then the map $\varphi$ in the polyhedral structure $(\Sigma, \varphi, \Lambda_*)$ in Theorem~\ref{thm:Def-thm-finite-Nagata-dim} is even defined on $\Sigma^{(k)}$ and the homomorphisms $\Lambda_m$ are simply given by the pushforward under $\varphi$. Hence, in this case, elements of $\poly_k(X)$ are (currents induced by) Lipschitz $k$-chains. It thus follows that if a metric space $X$ has finite Nagata dimension and has \((\LC_k)\) for some $k\geq 1$ then integral $k$-currents in $X$ can be approximated in the flat norm by Lipschitz chains; see Corollary~\ref{cor:flat-approximation}. The classical deformation theorem yields a stronger approximation result in Euclidean space which asserts that integral currents can be approximated by Lipschitz chains in total mass. Using the tools developed to prove Theorem~\ref{thm:Def-thm-finite-Nagata-dim} we can generalize this stronger approximation result as follows.

\begin{corollary}\label{cor:strong-approximation}
Let \(X\) be a complete metric space of finite Nagata dimension and let \(T\in \bI_k(X)\) for some \(k\geq 1\). If \(X\) has \((\LC_{k-1})\) then there exists a sequence of Lipschitz \(k\)-chains such that the induced integral currents \(T_i\) converge to \(T\) in total mass, that is, \(\mass(T-T_i)\to 0\) and \(\mass(\partial T-\partial T_i)\to 0\) as \(i\to \infty\).
\end{corollary}

For related approximation results in a different setting see for example \cite{DePauw-approx-2014}, \cite{Goldhirsch-approx-2021}. We do not know whether Corollary~\ref{cor:strong-approximation} holds without the assumption that $X$ have finite Nagata dimension.

The proof of Theorem~\ref{thm:Def-thm-finite-Nagata-dim} crucially relies on following approximation result which should be of independent interest and whose proof uses methods and arguments developed by Lang-Schlichenmaier \cite{MR2200122} in the context of Lipschitz extensions. 

\begin{theorem}\label{thm:factorization-side-length-r}
 Let \(X\), \(Y\) be metric spaces such that \(X\) is quasiconvex and of Nagata dimension \(\leq n\) and \(Y\) has $(\LC_{n-1})$ and contains \(X\). Then there exists a constant \(C\) depending only on the data of \(X\) and \(Y\) such that for every \(\varepsilon>0\) there is a metric simplicial complex \(\Sigma\) and \(C\)-Lipschitz maps \(\psi\colon X\to\Sigma\) and \(\varphi\colon\Sigma\to Y\) with the following properties:
 \begin{enumerate}
     \item\label{it:thm-1.6-1} \(\Sigma\) has dimension \(\leq n\), the metric on \(\Sigma\) is a length metric, and every simplex in \(\Sigma\) is a Euclidean simplex of side length \(\varepsilon\);
     \item\label{it:thm-1.6-2} $\varphi(\Sigma^{(0)})\subset X$, $\Hull(\psi(X)) = \Sigma$ and  \(d\bigl(x,\varphi(\psi(x))\bigr)\leq C\varepsilon\) for all \(x\in X\).
 \end{enumerate}
 If $X$ has $(\LC_{k-1})$ for some $k\geq 1$, then $\varphi(\Sigma^{(k)})\subset X$.
\end{theorem}

 If $X$ is compact then $\Sigma$ is a finite simplicial complex and, in particular, the metric on $\Sigma$ is a geodesic metric. The hull of a subset $A\subset\Sigma$ is the smallest subcomplex of $\Sigma$ containing $A$ and is denoted $\Hull(A)$.  Theorem~\ref{thm:factorization-side-length-r} is used as follows to construct a $(k,\varepsilon)$-polyhedral structure on a metric space $X$ as in Theorem~\ref{thm:Def-thm-finite-Nagata-dim}. Let $\Sigma$, $\psi$, $\varphi$ be as in Theorem~\ref{thm:factorization-side-length-r}, when applied to $X$ and $Y=\ell_\infty(X)$. It is not difficult to see that $\Sigma$ and the restriction $\varphi|_{\Sigma^{(0)}}\colon \Sigma^{(0)}\to X$ satisfy properties (1) and (2) in the definition of a polyhedral structure. Furthermore, using $\varphi|_{\Sigma^{(0)}}$ and the $(\LC_k)$ or $(\EI_k)$ condition on $X$ one easily builds homomorphisms $\Lambda_m\colon\poly_m(\Sigma)\to\bI_m(X)$ for $m=0,\dots, k$ satisfying properties (3) and (4) in the definition of a polyhedral structure. This is accomplished in Proposition~\ref{prop:homomorphisms-transporting}. We finally mention that Theorem~\ref{thm:factorization-side-length-r} has found another application in the recent article \cite{meier2021quasiconformal}, where it was used to prove the existence of a quasiconformal parametrization of $2$-dimensional metric surface under minimal conditions. 
 
 \subsection{Outlines of proof}\label{sec:proof-outline}
 The idea behind the proof of Theorem~\ref{thm:intro-undistorted-main} is inspired by the arguments used to prove \cite[Theorem 1.3]{MR3268779}. Let $X$ and $Y$ be as in the statement of the theorem.
 After possibly embedding $Y$ into a bigger space (for example a Banach space) we may assume that $Y$ has $(\LC_n)$, where $n$ is the Nagata dimension of $X$. 
 
 One of the main ingredients in the proof is a factorization theorem (see Theorem~\ref{thm:factorization}), which is established using arguments from \cite{MR2200122} and roughly says the following. There exists a Lipschitz map $f\colon Y\to Y$ such that $f$ restricts to the identity on $X$ and $f|_{Y\setminus X}$ factors through an $(n+1)$-dimensional simplicial complex $\Sigma$. More precisely, on $Y\setminus X$ we have $f=h\circ g$ for suitable maps $g\colon Y\setminus X\to \Sigma$ and $h\colon \Sigma\to Y$ satisfying the following properties. Let \(\varrho\colon Y\to \R\) be the distance function from the set \(X\) and let $r>0$. Then $g$ is \(Cr^{-1}\)-Lipschitz on the set $\{\varrho>r\}$ and $h$ is $Cr$-Lipschitz on every simplex $\sigma\subset \Sigma$ for which $g^{-1}(\st \sigma)\cap \{\varrho\leq r\}\not= \varnothing$, where $C$ is independent of $r$ and \(\st \sigma\) denotes the open star of \(\sigma\) in \(\Sigma\). Moreover, $h$ maps the $1$-skeleton of $\Sigma$ to $X$. The simplicial complex $\Sigma$ is equipped with the so-called $\ell_2$-metric and every simplex is isometric to a standard Euclidean simplex.
 
We now outline how this factorization theorem is used to prove Theorem~\ref{thm:intro-undistorted-main}. Let $T\in\bI_k(X)$ be a cycle in $X$ and $S\in\bI_{k+1}(Y)$ a filling of $T$ in $Y$. After possibly modifying $Y$ and $S$ slightly we may achieve that $S$ meets the set $X$ only on the boundary $T$ and not too much mass of $S$ is concentrated near $X$.  For suitable $r>0$ (small) we then send the restriction $S\on \{\varrho>r\}$ to an integral current $S_r'\in\bI_{k+1}(\Sigma)$ via the map $g$. Notice that the mass distortion of $g$ is proportional to $s^{-(k+1)}$ on the part of $S$ located at a distance around $s$. Next, we use a refined version of the classical Federer-Fleming deformation theorem in $\Sigma$ to ``push'' the integral current $S_r'$ to the $(k+1)$-skeleton $\Sigma^{(k+1)}$ of $\Sigma$ and obtain a polyhedral chain $P'_r\in\poly_{k+1}(\Sigma)$. This refined version of the deformation theorem, a detailed proof of which is given in the Appendix, includes local mass estimates and allows us to control the mass of $P_r'$ on every simplex. Finally, we map $P_r'$ back to $X$ via a chain homomorphism $\Lambda_{k+1}\colon \poly_{k+1}(\Sigma)\to \bI_{k+1}(X)$. Here, $\Lambda_1$ is the map induced by the restriction $h|_{\Sigma^{(1)}}\colon \Sigma^{(1)}\to X$. For $m=2,\dots, k$, we construct $\Lambda_{m+1}$ from $\Lambda_m$ by applying the $(\LC_m)$ or $(\EI_m)$ property of $X$. For a suitable sequence $r_i\to 0$, the currents $P_i\coloneqq \Lambda_{k+1}(P'_{r_i})\in\bI_{k+1}(X)$ converge to a filling $P\in\bI_{k+1}(X)$ of \(T\) with the desired bound on mass, thus concluding the proof of Theorem~\ref{thm:intro-undistorted-main}.

 The proof of Theorem~\ref{thm:Def-thm-finite-Nagata-dim} is similar but somewhat easier because all the involved maps are globally Lipschitz. Instead of using the factorization theorem sketched above we apply Theorem~\ref{thm:factorization-side-length-r} with $Y=\ell_\infty(X)$ to obtain a finite dimensional metric simplicial complex $\Sigma$ and Lipschitz maps $\psi\colon X\to \Sigma$ and $\varphi\colon \Sigma \to Y$ as in the theorem. As described after the theorem, this leads to a $(k,\varepsilon)$-polyhedral structure $(\Sigma, \varphi, \Lambda_*)$ on $X$. Now, let $T\in\bI_k(X)$ and consider the pushforward current $T'=\psi_\# T\in\bI_k(\Sigma)$. The classical Federer-Fleming deformation theorem in $\Sigma$ yields a decomposition $T' = P' + R' +\partial S'$ for currents $P'\in\poly_k(\Sigma)$ and $R'\in\bI_k(X)$, $S'\in\bI_{k+1}(\Sigma)$ with suitable bounds on mass. We then use the fact that $\varphi\circ\psi$ is close to the identity on $X$ together with further properties of $\Lambda_m$ to show that $T$ can be decomposed as $T= P + \hat{R} + \partial \hat{S}$, where $P=\Lambda_k(P')$ and where $\hat{R}\in\bI_k(Y)$ and $\hat{S}\in\bI_{k+1}(Y)$ are suitable currents. We finally use Theorem~\ref{thm:intro-undistorted-main} to ``push'' $\hat{R}$ and $\hat{S}$ from $Y$ to $X$ in order to obtain a decomposition as in Theorem~\ref{thm:Def-thm-finite-Nagata-dim}.

 \subsection{Structure of the paper} 
 In Section~\ref{sec:prelims} we introduce notation and definitions concerning concepts used in the article, including simplicial complexes, Nagata dimension, Lipschitz connectedness, and pointwise Lipschitz constants. Section~\ref{sec:currents} contains the necessary definitions from the Ambrosio-Kirchheim theory of metric currents. We furthermore prove Proposition~\ref{prop:homomorphisms-transporting} which allows us to construct suitable homomorphisms from the group of polyhedral chains in a simplicial complex to the group of integral currents in a metric space. This technical result will be needed in the proofs of Theorems~\ref{thm:intro-undistorted-main} and \ref{thm:Def-thm-finite-Nagata-dim}. The main result in Section~\ref{sec:factorization} is the factorization theorem sketched in Section~\ref{sec:proof-outline} above; see Theorem~\ref{thm:factorization}. As already mentioned, this is one of the main ingredients in the proof of Theorem~\ref{thm:intro-undistorted-main}. In Section~\ref{sec:isoperimetric-subspace-distortion} we discuss and prove variants and generalizations of Theorem~\ref{thm:intro-undistorted-main} and deduce Theorem~\ref{thm:intro-undistorted-main} from them. The main aim of  Section~\ref{sec:approx-simplicial-complexes} is to establish  Theorem~\ref{thm:factorization-side-length-r}, whose proof is similar to that of Theorem~\ref{thm:factorization}. Section~\ref{sec:def-thm-finite-Nagata} contains the proof of Theorem~\ref{thm:Def-thm-finite-Nagata-dim}. Finally, in the Appendix we give a detailed proof of a refined version of the Federer-Fleming deformation theorem in metric spaces admitting a bilipschitz triangulation.  Our refined version includes local mass estimates which are of particular importance in the proof of our isoperimetric subspace distortion theorem.

\section{Preliminaries}\label{sec:prelims}

\subsection{Notation and definitions} We recall standard notation and definitions from metric geometry. Suppose that \(A\), \(B\subset X\) are subsets of a metric space \(X\). We write 
\[
d(A, B)\coloneqq \inf\bigl\{ d(a, b) : a\in A,\, b \in B\bigr\} \quad \text{and} \quad \diam(B)\coloneqq \sup\{ d(x,x^\prime) : x, x^\prime\in B\}.
\] 
We use the convention that \(\inf \varnothing=+\infty\) and \(\sup \varnothing=-\infty\). For every \(\varepsilon >0\) we let \(U(A, \varepsilon)\) and \(B(A, \varepsilon)\) denote the open and closed \(\varepsilon\)-neighborhood of \(A\), respectively. 
We say that a subset \(Z\subset X\) is an \textit{\(\varepsilon\)-net} if \(d(z,z^\prime)>\varepsilon\) for all distinct \(z\), \(z^\prime\in Z\) and if the family of all balls \(B(z, \varepsilon)\) with \(z\in Z\) covers \(X\).

A metric space \(X\) is called \(c\)-\textit{quasiconvex}, \(c\geq 1\), if all \(x\), \(x'\in X\) can be connected by a curve \(\gamma\colon [0,1]\to X\) such that \(\ell(\gamma)\leq c d(x,x')\). 
For us a curve is always continuous and we use \(\ell(\gamma)\in [0, \infty]\) to denote the \textit{length} of a curve  \(\gamma\). 
Let \(k\geq 0\) and denote by \(\Haus^k\) the \(k\)-dimensional Hausdorff measure on \(X\) and let \(\mathscr{L}^k\) denote the Lebesgue measure on \(\R^k\). 
A subset \(A\subset X\) is said to be \textit{countably \(\Haus^k\)-rectifiable} if it is \(\Haus^k\)-measurable and there exist subsets \(A_i\subset \R^k\) and Lipschitz maps \(f_i\colon A_i\to X\) such that
\(\Haus^k\bigl( A \setminus \bigcup_{i=1}^\infty f_i(A_i)\bigr)=0\). A finite Borel measure \(\mu\) on \(X\) is said to be concentrated on a Borel set \(B\subset X\) if \(\mu(X\setminus B)=0\). The set \(\spt \mu\) consisting of all \(x \in X\) such that \(\mu(U(x, \varepsilon)) > 0\) for all \(\varepsilon>0\) is called the support of \(\mu\). If $X$ is separable then \(\mu\) is concentrated on \(\spt \mu\). This is true more generally if the cardinality of $X$ is an Ulam number; see \cite[2.1.6]{MR0257325}. As is done for example in \cite{MR1794185} we will assume throughout this paper that the cardinality of any set is an Ulam number.

\subsection{Simplicial complexes} 

For the concepts appearing below we refer e.g.\ to \cite{MR1325242}. Let $\Sigma$ be a simplicial complex. We denote by $\mathcal{F}_k(\Sigma)$ the collection of closed $k$-simplices in $\Sigma$ and by $\mathcal{F}(\Sigma)$ the collection of closed simplices of any dimension in $\Sigma$. We also write $\mathcal{F}_k$ and $\mathcal{F}$ if there is no danger of ambiguity. We say that \(\Sigma\) is \(n\)-dimensional if \(\mathcal{F}_n\neq \varnothing\) and \(\mathcal{F}_{n+1}=\varnothing\). The $k$-skeleton of $\Sigma$ is denoted by $\Sigma^{(k)}$. Notice that if \(\Sigma\) is \(n\)-dimensional, then \(\Sigma^{(m)}=\Sigma^{(n)}\) for every \(m\geq n\). The smallest subcomplex of $\Sigma$ containing a given subset $A\subset \Sigma$ is called the hull of $A$ and is denoted $\Hull(A)$. The open star of \(\sigma\in \mathcal{F}\) is defined by
\[\st \sigma\coloneqq \bigcup \Big\{ \interior \tau : \tau\in \mathcal{F} \text{ and }\sigma \subset \tau \Big\},
\]
where $\interior{\tau}$ is the interior of $\tau$. We will use the following two specific choices of a metric on a given simplicial complex. For a nonempty index set $I$ define a simplicial complex by 
\begin{equation*}
\Sigma(I)\coloneqq 
\Bigl\{ x\in \ell_2(I) : x_i\geq 0 \text{ and } \sum_{i\in I} x_i =1 \Bigr\}
\end{equation*}
and equip $\Sigma(I)\subset \ell_2(I)$ with the metric induced by the norm of \(\ell_2(I)\).
A given simplicial complex $\Sigma$ can be naturally realized as a subset of $\Sigma(I)$, where $I=\mathcal{F}_0$ is the collection of vertices in $\Sigma$. The induced metric on $\Sigma\subset \Sigma(I)$  will be referred to as the $\ell_2$-metric on $\Sigma$. The diameter of $\Sigma$ with respect to the $\ell_2$-metric is always at most $\sqrt{2}$. If $\Sigma$ is path-connected then the associated length metric on $\Sigma$ will be called the {\it length metric} on $\Sigma$. Notice that with respect to either metric each simplex of \(\Sigma\) is isometric to a standard Euclidean simplex. 
We close this subsection with the following lemma which collects various results that will be used later.

\begin{lemma}\label{lem:GeometryOfDeltaN}
Let $\Sigma$ be a finite-dimensional simplicial complex, equipped with the $\ell_2$-metric. Then the following holds:

\begin{enumerate}
    \item\label{it:epsilon_x} For every \(x\in \Sigma\), there is \(\varepsilon_x>0\) such that if \(y\in U(x, \varepsilon_x)\) then \(x\) and \(y\) lie in a common simplex; 
    \item If \(A\subset \Sigma\) is separable, then \(\Sigma^\prime\coloneqq\Hull(A)\) is a countable simplicial complex, that is, the set of vertices of \(\Sigma^\prime\) is countable; 
    \item If \(\Sigma\) is path-connected, then the length metric and the \(\ell_2\)-metric induce the same topology. 
 \end{enumerate} 
\end{lemma}
An analogous statement holds when the $\ell_2$-metric is replaced by the length metric.
\begin{proof}
For \(x\in \Sigma\) denote by \(\sigma(x)\) the unique simplex \(\sigma\) such that \(x\in \interior \sigma\). Given \(x=(x_i)_{i\in I}\), we put \(I(x)\coloneqq\{ i\in X : x_i\neq 0\}\) and \(\varepsilon_x\coloneqq \min\{ x_i : i\in I(x)\}\). Clearly, \(\varepsilon_x>0\) and if \(\abs{x-y} < \varepsilon_x\), then \(y_i\neq 0\) for all \(i\in I(x)\). Hence, \(\sigma(x)\subset \sigma(y)\) whenever \(\abs{x-y}< \varepsilon_x\). This gives \eqref{it:epsilon_x}. Next, we show (2). Let \(Z\) be a countable dense subset of \(A\). We claim that 
 \begin{equation}\label{eq:hull-countable}
 \Hull(A)=\bigcup_{z\in Z} \sigma(z),
 \end{equation}
 Fix \(x\in A\). As \(Z\) is a dense subset of \(A\), there exists \(z\in Z\) such that \(\abs{x-z}<\varepsilon_{x}\) and thus \(\sigma(x)\subset \sigma(z)\). This yields \eqref{eq:hull-countable} and (2) follows. 
 
 Finally, we prove (3). Suppose now that \(\Sigma\) is path-connected. By virtue of (1), the open balls \(U(x, \varepsilon) \), where \(\varepsilon \leq \varepsilon_x\),  with respect to the length metric and the \(\ell_2\)-metric coincide with each other. This shows in particular that both metrics induce the same topology on \(\Sigma\), as desired.
 \end{proof}

\subsection{Nagata Dimension}\label{subsec:Nagata}
In the following, we recall the definition of Nagata dimension and some of its basic properties. A covering \(\mathcal{B}=(B_i)_{i\in I}\) of a metric space \(X\) is said to be \textit{\(D\)-bounded}, for a real number \(D\geq 0\), if every set \(B\in \mathcal{B}\) has diameter less than or equal to \(D\). 
Let \(s>0\) be a real number and \(n\geq 0\) an integer. 
A covering \(\mathcal{B}\) has \textit{\(s\)-multiplicity at most} \(n\) if every subset \(B\subset X\) with \(\diam(B)\leq s\) meets at most \(n\) members of \(\mathcal{B}\). 

\begin{definition}\label{def:nagata}
We say that a metric space \(X\) has \textit{Nagata dimension \(\leq n\) with constant \(c\)} if for every real number \(s>0\), \(X\) 
admits a \(cs\)-bounded covering \(\mathcal{B}_s\) having \(s\)-multiplicity at most \(n+1\).
\end{definition}
Equivalent definitions can be found in \cite[Proposition 2.5]{MR2200122}. 
The infimum of those \(n\geq 0\) for which \(X\) has Nagata dimension \(\leq n\) is denoted by \(\dim_N(X)\) and is called the \textit{Nagata dimension} of \(X\). 
Nagata dimension was introduced by Assouad in \cite{MR651069}.
One has \(\dim(X) \leq \dim_N(X)\) for any metric space \(X\), where \(\dim(X)\) denotes the topological dimension of \(X\).
If \(X\) is a compact Riemannian manifold or a Carnot group equipped with the Carnot-Carathéodory distance, then \(\dim_N(X)=\dim(X)\). See \cite[p. 3635]{MR2200122} and \cite[Theorem 4.2]{MR3320519}. Furthermore, \(\dim_N X\times Y \leq \dim_N X+\dim_N Y\) and if \(X=A\cup B\), then \(\dim_N X = \sup\{ \dim_N A, \dim_N B\}\). For further results concerning Nagata dimension we refer to the articles \cite{BellDranishnikov-asymptotic-2008}, \cite{MR2520117}, \cite{MR2418302},  \cite{MR2200122},  \cite{MR3320519} and the references therein.

\subsection{Lipschitz connectedness}\label{sec:Lip-connected}
Let \(X\) and \(Y\) be metric spaces. A map \(f\colon X\to Y\) is called \textit{\(L\)-Lipschitz}, where \(L\geq 0\) is a real number, if \(d(f(x), f(x^\prime))\leq L\,d(x, x^\prime)\)
for all \(x,x^\prime \in X\).
The quantity 
\[
\Lip(f)\coloneqq \inf\big\lbrace L\geq 0 : f \text{ is \(L\)-Lipschitz}\big\rbrace
\] 
is called the \textit{Lipschitz constant} of \(f\).
We say that \(f\) is \textit{\(C\)-bilipschitz}, \(C\geq 1\), if 
 \[
 C^{-1} d(x,y) \leq d(f(x), f(y)) \leq C d(x,y)
 \]
 for all \(x\), \(y\in X\). In the following, we recall the definition of Lipschitz \(k\)-connected metric spaces and state several examples. Let \(B^{m+1}\coloneqq \bigl\{ x\in \R^{m+1} : \abs{x}\leq 1\bigr\}\), for \(m\geq 0\), denote
the closed \((m+1)\)-dimensional Euclidean unit ball and put \(S^m\coloneqq \partial B^{m+1}\). We equip \(S^m\) and \(B^{m+1}\) with the induced Euclidean metric. 

\begin{definition}A metric space $X$ is said to be Lipschitz $k$-connected if there exists $c\geq 1$ such that every \(L\)-Lipschitz map $f\colon S^m\to X$ with $0\leq m\leq k$ has a $c L$-Lipschitz extension $\bar{f}\colon B^{m+1}\to X$.  We abbreviate this by saying that $X$ has property $(\LC_k)$. 
\end{definition}

In particular, Banach spaces and \(\CAT(0)\) spaces or, more generally,  metric spaces admitting a convex bicombing have \((\LC_k)\) with constant \(3\); see \cite[Proposition 6.2.2.]{20.500.11850/148886}.
 Every \(k\)-connected compact Riemannian manifold has \((\LC_k)\). This follows from \cite[Theorem 5.1]{MR2200122} which states that if a compact metric space \(X\) is \(k\)-connected and Lipschitz \(k\)-connected in the small, then \(X\) has \((\LC_k)\). Moreover, certain horospheres in symmetric spaces of noncompact type of rank \(k\) or in the product of \(k\) proper \(\CAT(0)\) spaces have \((\LC_{k-2})\). This has been established in \cite{LeuzingerYoung-distortionDimension-2017} and \cite{Link-HiherOrder-2019}, respectively. 
The jet space Carnot groups \((J^{s}(\R^k), d_c)\) have \((\LC_{k-1})\); see \cite{MR2729637}. Examples of Carnot groups which can be realized as jet space Carnot groups include the \(n\)th Heisenberg group, the Engel group and the model filiform groups. We refer to \cite{Warhurst-jetspaces-2005} for more information on these groups.
There are many examples of metric groups which have \((\LC_1)\); see e.g. \cite{Cohen-Lipschitz-2019}, \cite{Magnani-conact-2010}.

\subsection{Pointwise Lipschitz constants}
Let \(f\colon X\to Y\) be a map between metric spaces $X$ and $Y$. The \textit{pointwise (lower) Lipschitz constant} of \(f\) at \(x\) is defined by
\[\lip f(x)\coloneqq \lim_{r\to 0+} \ell_r f(x), \quad \text{ where } \, \ell_r f(x)\coloneqq  \inf_{\substack{\\[0.1em]0<s<r}} \,\sup_{d(x,y)<s} \frac{d(f(x),f(y))}{s}.\]
If $f$ is Lipschitz then the function \(\lip f\) is real-valued and Borel; see \cite[Lemma 4.1.2]{MR2041901}.

\begin{lemma}\label{lem:littleLip}
Let \(f\colon X\to Y\) be a map such that \(\lip f(x)\leq C\) for all \(x\in X\). Then \(f\) is continuous and \(\ell(f\circ \gamma)\leq C \ell(\gamma) \)
for every curve \(\gamma\) in \(X\). In particular,  if \(X\) is \(\lambda\)-quasiconvex, then \(f\) is \(C\lambda\)-Lipschitz. 
\end{lemma}

\begin{proof}
Let \(x\in X\). Since \(\lip f(x)\leq C\), for every \(r>0\) there is \(s\in (0,r)\) such that
\[
\sup_{d(x,y)<s} d(f(x), f(y)) \leq 2Cs;
\]
hence, \(f\) is continuous at \(x\). In order to prove the second statement it is enough to show that 
\begin{equation}\label{eq:ineq-dist-length-lemma-lip}
    d\bigl(f(\gamma(0)), f(\gamma(l))\bigr)\leq C\ell(\gamma)
\end{equation}
for every rectifiable curve $\gamma\colon[0,l]\to X$. Fix a rectifiable curve $\gamma$, which we may assume to be parametrized by arc-length. Let $C'>C$ and consider the set
\[
A\coloneqq \bigl\{ t\in [0,l] : d\big(f\big(\gamma(0)\big), f\big(\gamma(t)\big)\big)\leq C' t  \bigr\}.
\]
Clearly, \(A\) is nonempty and closed, so \(t\coloneqq \sup A\) is contained in \(A\). We claim that \(t=l\). We argue by contradiction and assume that \(t<l\). Choose \(s>0\) sufficiently small such that \(t+s<l\) and
\[
\sup_{d(\gamma(t),y)<s} \frac{d(f(\gamma(t)),f(y))}{s} \leq C'.
\]
Since $d(\gamma(t),\gamma(t+s')) \leq \ell(\gamma|_{[t,t+s']}) = s'$ for all $0<s'<s$, we obtain from the above and the continuity of $f$ that $d\bigl(f(\gamma(t)), f(\gamma(t+s))\bigr)\leq C's$ and hence 
\[
d\bigl(f(\gamma(0)), f(\gamma(t+s))\bigr)
\leq C'(t+s).
\]
This shows that \(t+s\in A\), which is impossible.
We therefore have \(t=l\) and thereby \(d(f(\gamma(0)), f(\gamma(l)))\leq C' l = C'\ell(\gamma)\). Since $C'>C$ was arbitrary this shows \eqref{eq:ineq-dist-length-lemma-lip} and proves the second statement of the lemma. 
\end{proof}

\section{Currents in metric spaces}\label{sec:currents}
We recall relevant notions from Ambrosio-Kirchheim's theory of currents of finite mass in metric spaces \cite{MR1794185}. See \cite{MR2810849} for a variant of this theory.

\subsection{Currents of finite mass, push-forwards and restrictions}\label{sec:3.1-ref-nagata}

Let $X$ be a complete metric space, $k\geq 0$ an integer, and let \(\mathcal{D}^k(X)\) be the set of tuples \((\pi_0, \dots, \pi_k)\) of Lipschitz functions \(\pi_i\colon X\to\R\) with \(\pi_0\) bounded.

\begin{definition}\label{def:currents}
A multi-linear functional \(T\colon\mathcal{D}^k(X)\to X\) is called metric $k$-current (of finite mass) in $X$ if the following properties hold:
\begin{enumerate}
    \item[(i)]\label{it:currents-1} $T(\pi_0, \pi_1^{j}, \dots, \pi_k^{j})\to T(\pi_0, \pi_1, \dots, \pi_k)$
    whenever \(\sup_{i,j} \Lip\bigl(\pi_i^{j}\big)<\infty\) and \(\pi_i^{j}\) converges pointwise to \(\pi_i\).
    \item[(ii)]\label{it:currents-2} if \(\{ x\in X : \pi_0(x)\neq 0\}\) is contained in the union \(\bigcup_{i=1}^k B_i\) of Borel sets \(B_i\) and if \(\pi_i\) is constant on \(B_i\), then \(T(\pi_0, \dots, \pi_k)=0\).
    \item[(iii)] there exists a finite Borel measure \(\mu\) on \(X\) such that for all \((\pi_0, \dots, \pi_k)\in \mathcal{D}^{k}\bigl(X\bigr)\)
    \begin{equation}\label{eq:def-of-mass}
    \abs{T(\pi_0, \dots, \pi_k)}\leq \prod_{i=1}^k \Lip(\pi_i) \int_X \abs{\pi_0(x)} \, d\mu(x).
    \end{equation}
\end{enumerate}
\end{definition}

The minimal Borel measure \(\mu\) satisfying \eqref{eq:def-of-mass} is denoted \(\lVert T \rVert\). The closed set 
\[
\spt T\coloneqq \bigl\{x\in X : \lVert T \rVert\bigl( U(x,r) \bigr) >0 \text{ for all \(r>0\)}\bigr\}
\]
is called the support of $T$, and \(\mass(T)= \lVert T \rVert(X)\) is the mass of $T$. Let \(\bM_k(X)\) be the space of  metric \(k\)-currents in \(X\). When equipped with the mass norm $\mass(\cdot)$, this becomes a Banach space. If \(T\in \bM_k(X)\) and \(T\neq 0\), then \(\dim_N( \spt T)\geq k\); see \cite[Proposition 2.5]{zuest-currents-2011}. In particular, one has \(\bM_{k}(X)=\{0\}\) whenever \(k > \dim_N(X)\).

Let $Y$ be another complete metric space and let $f\colon X\to Y$ be a Lipschitz map. The push-forward  of \(T\in \bM_k(X)\) under $f$ is the element $f_\#T\in\bM_k(Y)$ defined by
\[
f_\# T(\pi_0,\dots, \pi_k)=T(\pi_0\circ f, \dots, \pi_k \circ f)
\]
for all \((\pi_0, \dots, \pi_k)\in \mathcal{D}^k(Y)\). We have \(\spt f_\# T\subset f(\spt T)\) and $\mass(f_\# T) \leq \Lip(f)^k \mass(T)$.
If \(T\in \bM_k(X)\) then the \textit{restriction} \(T\on B\) to a Borel set \(B\subset X\) is defined by
\[
T\on B\,(\pi_0, \dots, \pi_k)\coloneqq T(\mathbbm{1}_B\, \pi_0, \pi_1, \dots, \pi_k)
\]
for all \((\pi_0, \dots, \pi_k)\in \mathcal{D}^k(X)\), where \(\mathbbm{1}_B\) is the characteristic function of \(B\). This is well-defined since \(T\) admits a unique extension to a functional on \((k+1)\)-tuples for which the first argument is a bounded Borel function. One has \(T\on B \in \bM_{k}(X)\) and \(\lVert T\on B \rVert=\lVert T \rVert \on B\).
Moreover, we have \((f_\# T)\on B=f_\#\big( T\on f^{-1}(B)\big)\) whenever $f\colon X\to Y$ is Lipschitz and $B\subset Y$ is Borel. 

A sequence $(T_i)\subset\bM_k(X)$ is said to \textit{converge weakly} to $T\in\bM_k(X)$ if $(T_i)$ converges pointwise to $T$, thus for all \((\pi_0, \dots, \pi_k)\in \mathcal{D}^{k}\bigl(X\bigr)\) we have $T_i(\pi_0,\dots, \pi_k)\to T(\pi_0,\dots, \pi_k)$ as $i\to\infty$.

 \subsection{Boundaries, slicings and integral currents}
When \(k\geq 1\), the boundary of \(T\in \bM_k(X)\) is the function \(\partial T \colon \mathcal{D}^{k-1}(X)\to \R\) given by 
\[
\partial T(\pi_0, \dots, \pi_{k-1})\coloneqq T(1, \pi_0, \pi_1, \dots, \pi_{k-1}).
\]
Any \(T\in \bM_{k}(X)\), \(k\geq 1\), satisfying \(\partial T=0\) is called a cycle. If \(T\in \bM_0(X)\) then we say that \(T\) is a cycle provided that \(T(1)=0\).
Notice that \(\partial T\) always satisfies \((\text{i})\) and \((\text{ii})\) of Definition~\ref{def:currents}.
If \(\partial T\in \bM_{k-1}(X)\) then \(T\) is called \textit{normal} current. For such $T$ we have $\spt(\partial T) \subset \spt T$. We denote by \(\bN_k(X)\) the space of normal $k$-currents for $k\geq 1$, and we set $\bN_0(X)\coloneqq \bM_0(X)$. Notice that $\partial(\partial T) = 0$ for all $T\in\bN_k(X)$ with $k\geq 2$. 
If \(\theta\in L^1(\R^k)\) then the function \(\llbracket \theta \rrbracket \colon \mathcal{D}^{k}(\R^k)\to \R\) given by
\[
\llbracket \theta \rrbracket(\pi_0,\dots, \pi_k)\coloneqq \int_{\R^k} \theta\, \pi_0 \det\Bigl( \bigl[\partial_i \pi_j\bigr]_{i,j=1}^k\Bigr)\, d\mathscr{L}^k.
\]
defines an element of \(\bM_k(\R^k)\) and satisfies \(\lVert \, \llbracket \theta \rrbracket \, \rVert= \abs{\theta} \, \mathscr{L}^k\). If $\theta$ has bounded variation then $\llbracket \theta \rrbracket\in\bN_k(\R^k)$. In particular, if $B\subset\R^k$ is a bounded Borel set of finite perimeter then \(\bb{B}\coloneqq \bb{\mathbbm{1}_B}\) belongs to \(\bN_k(\R^k)\).

\begin{definition}
An element \(T\in \bN_k(X)\) with \(k\geq 1\) is called an \emph{integral current} if \(\lVert T \rVert\) is concentrated on a countably \(\Haus^k\)-rectifiable set and if for any Lipschitz map \(f \colon X\to \R^k\) and any open set \(U\subset X\) there exists \(\theta\in L^1(\R^k, \Z)\) such that \(f_\# (T\on U)= \llbracket \theta \rrbracket\).
\end{definition}

Furthermore, $T\in\bN_0(X)$ is called an integral current if there exist $x_1,\dots, x_m\in X$ and $\theta_1,\dots,\theta_m\in\Z$ such that $T = \sum_{i=1}^m \theta_i\bb{x_i}$. Here, for $x\in X$ the current $\bb{x}\in \bN_0(X)$ is defined by $\bb{x}(\pi) = \pi(x)$ for every bounded Lipschitz function \(\pi\colon X \to \R\). The family of integral $k$-currents in $X$ is an additive subgroup of $\bN_k(X)$ and is denoted by $\bI_k(X)$. If $f\colon X\to Y$ is Lipschitz and $T\in\bI_k(X)$ then $f_\# T\in\bI_k(Y)$. In particular, if $\Delta\subset\R^k$ is a $k$-simplex and $\varphi\colon \Delta\to X$ a Lipschitz map then $\varphi_\#\llbracket\Delta\rrbracket \in\bI_k(X)$; consequently, every Lipschitz chain in $X$ induces an integral current in $X$. The following lemma shows that the mass estimate $\mass(f_\# T) \leq \Lip(f)^k \mass(T)$ can be improved if \(T\) is an integral current.

\begin{lemma}\label{lem:rajala-wenger}
If \(f\colon X\to Y\) is a Lipschitz map then 
\begin{equation*}
\lVert f_\# T \rVert(B) \leq \int_{f^{-1}(B)} \bigl[\lip f(x)\bigr]^k \, d\Vert T \Vert(x)
\end{equation*}
for every integral current \(T\in \bI_k(X)\) and every Borel set \(B\subset Y\). 
\end{lemma}

\begin{proof}
The inequality is obviously true when $k=0$. If $k\geq 1$ then \cite[Lemma 3.10]{MR3125271} implies that
\[
\abs{f_\# T(\pi_0, \dots, \pi_k)} \leq \int_X \abs{(\pi_0\circ f)(x)} \,\prod_{i=1}^k \lip (\pi_i \circ f)(x) \, d\lVert T \rVert(x)
\]
for all \((\pi_0, \dots, \pi_k)\in \mathcal{D}^{k}(Y)\).
Since \(\lip (\pi_i\circ f)(x)\leq \Lip(\pi_i)\cdot\lip f(x)\) for all \(x\in X\), we have
\[
\abs{f_\# T(\pi_0, \dots, \pi_k)} \leq \prod_{i=1}^k \Lip(\pi_i) \int_X \abs{\pi_0\circ f(x)}\cdot \bigl[\lip f(x)\bigr]^k \, d\Vert T \Vert(x). 
\]
Hence, the claim follows from the definition of $\lVert f_\#T\rVert$. 
\end{proof}

If \(T\in \bN_k(X)\) with \(k\geq 1\) and \(u\colon X\to \R\) is Lipschitz then for almost every \(r\in \R\),
\[
\langle T, u, r \rangle\coloneqq \partial\bigl( T\on \{ u \leq r\} \bigr)-(\partial T)\on \{u \leq r\}.
\]
defines an element of $\bN_{k-1}(X)$, which is called a slice of $T$. By the slicing theorem, the measure \(\lVert \langle T, u, r\rangle \rVert\) is supported on \(\spt T \cap \{ u = r\}\), and 
\[
\int_{a}^b \mass( \langle T, u, r\rangle ) \, dr \leq \Lip(u) \, \lVert T \rVert ( \{ a < u <b \}) 
\]
whenever \(a <b\). Moreover, if \(T\in \bI_k(X)\) then \(\langle T, u, r\rangle \in \bI_{k-1}(X)\) for almost all \(r\in \R\).

\subsection{Homotopy formula and coning inequality}\label{subsec:homotopy-formula}

Let \(\varepsilon >0\) and let \([0,\varepsilon]\times X\) be equipped with the Euclidean product metric. For a function $f\colon[0,\varepsilon]\times X\to\R$ and $t\in[0,\varepsilon]$ we let $f_t\colon X\to \R$ be defined by $f_t(x)\coloneqq f(t,x)$. If \(T\in \bI_k(X)\) for some $k\geq 0$ then the function on \(\mathcal{D}^{k}\bigl([0,\varepsilon] \times X\bigr)\) given by
\[
\llbracket t \rrbracket \times T (\pi_0, \dots, \pi_k)\coloneqq T(\pi_{0t}, \dots, \pi_{kt})
\]
defines an element of  \(\bI_k([0,\varepsilon]\times X)\). 
Moreover, the functional \(\bb{0,\varepsilon}\times T\) assigning
\[
(\pi_0, \dots, \pi_{k+1})\mapsto \sum_{i=1}^{k+1} (-1)^{i+1} \int_{0}^\varepsilon T\Bigl( \pi_{0t} \frac{\partial \pi_{it}}{\partial t}, \pi_{1t}, \dots, \pi_{(i-1)t}, \pi_{(i+1)t}, \dots, \pi_{(k+1)t} \Bigr)\, dt
\]
is a multi-linear functional on \(\mathcal{D}^{k+1}\bigl([0,\varepsilon] \times X\bigr)\) and has the following property.

\begin{proposition}\label{prop:homotopy}
If \(T\in \bI_k(X)\) then \(\bb{0,\varepsilon}\times T \in \bI_{k+1}\bigl([0, \varepsilon] \times X\bigr)\) and
\begin{equation}\label{eq:ProductFormula}
\partial \bigl(\bb{0,\varepsilon}\times T \bigr)+\bb{0,\varepsilon}\times \partial T= \llbracket \varepsilon \rrbracket \times T-\llbracket 0 \rrbracket \times T.
\end{equation}
For $k=0$ the second term on the left-hand side is zero.
\end{proposition}

This is analogous to \cite[Theorem 2.9]{MR2153909} and \cite{MR1794185}. Notice that the assumption made in \cite[Theorem 2.9]{MR2153909} that \(T\) has bounded support is not needed.

Given a map $h\colon[0,\varepsilon]\times X\to Y$ and $t\in[0,\varepsilon]$, $x\in X$, we let $h_t\colon X\to Y$ and $h_x\colon [0,\varepsilon]\to Y$ be the maps given by $h_t(z)\coloneqq h(t,z)$ and $h_x(s)\coloneqq h(s,x)$.

\begin{lemma}\label{lem:mass-of-pushforward-of-product}
If \(h\colon [0,\varepsilon]\times X\to Y\) is Lipschitz and \(T\in\bI_k(X)\), then
\[
\lVert\, h_\#\bigl(\bb{0,\varepsilon}\times T\bigr)\,\rVert(B) \leq (k+1)\int_{h^{-1}(B)} \lip h_x(t)\cdot\bigl[\lip h_t(x)\bigr]^{k} \, d(\mathscr{L}^1 \times \lVert T \rVert)(t, x) 
\]
for every Borel set \(B\subset Y\).
\end{lemma}

\begin{proof}
Let \((\pi_0, \dots, \pi_{k+1})\in \mathcal{D}^{k+1}(Y)\). We have
\begin{align*}
\bigl\lvert &h_\#\bigl(\bb{0,\varepsilon}\times T\bigr)(\pi_0, \dots, \pi_{k+1})\bigr\rvert \\
&\leq \sum_{i=1}^{k+1} \Bigl\lvert\int_{0}^\varepsilon T\Bigl( \pi_0 \circ h_t \, \frac{\partial (\pi_i\circ h_t)}{\partial t}, \pi_{1}\circ h_t, \dots, \pi_{(i-1)}\circ h_t,\, \pi_{(i+1)}\circ h_t, \dots, \pi_{(k+1)}\circ h_t \Bigr)\, dt\Bigr\rvert \\
&\leq \sum_{i=1}^{k+1} \int_{0}^\varepsilon \int_X \,\Bigl\lvert \pi_0 \circ h_t \, \frac{\partial (\pi_i\circ h_t)}{\partial t} \Big\rvert\, \prod_{j=1, j\neq i}^{k+1} \lip (\pi_j \circ h_t)(x) \, d\lVert T \rVert(x) \,dt,
\end{align*}
where the second inequality follows from the definition of $\lVert T\rVert$ when $k=0$ and from \cite[Lemma 3.10]{MR3125271} when $k\geq 1$.
Since
\[
\bigl\lvert \frac{\partial (\pi_i\circ h_t(x))}{\partial t} \bigr\rvert= \lip(\pi_i\circ h_x)(t)
\]
for almost all \(t\in [0,\varepsilon]\) and since \(\lip (\pi_j \circ h_t)(x)\leq \Lip(\pi_j)\lip h_t(x)\) we obtain
\begin{align*}
\bigl\lvert& h_\#\bigl(\bb{0,\varepsilon}\times T\bigr)(\pi_0, \dots, \pi_{k+1})\bigr\rvert\\
&\leq (k+1) \prod_{i=1}^{k+1} \Lip(\pi_i) \int_0^1 \int_X \abs{ (\pi_0\circ h)(t,x)}\, \lip h_x(t)\cdot\bigl[\lip h_t(x)\bigr]^{k} \, d\lVert T \rVert \,dt.
\end{align*}
Hence, the claim follows from the definition of $\lVert h_\#\bigl(\bb{0,\varepsilon}\times T\bigr)\rVert$.
\end{proof}

With the homotopy formula from Proposition~\ref{prop:homotopy} and the mass estimate from Lemma~\ref{lem:mass-of-pushforward-of-product} at hand it is not hard to see that Banach spaces admit coning inequalities for any \(k\geq 1\).

\begin{corollary}
 Banach spaces have $(\CI_k)$ with constant \(C=1\) for every $k\geq 0$.
\end{corollary}

\begin{proof}
 Let $X$ be a Banach space and $T\in\bI_k(X)$ a cycle of bounded support for some $k\geq 0$. Fix $x_0\in \spt T$. The map $h\colon[0,1]\times X\to X$ given by $h(t,x)= (1-t)x_0 + tx$ is Lipschitz and satisfies $\lip h_x(t) \leq \diam(\spt T)$ and $\lip h_t(x) = t$ for all $t\in [0,1]$ and all $x\in \spt T$. Thus, by Proposition~\ref{prop:homotopy} and Lemma~\ref{lem:mass-of-pushforward-of-product} the current defined by $S\coloneqq h_\#\bigl(\bb{0,1}\times T\bigr)$ belongs to $\bI_{k+1}(X)$ and satisfies $\partial S=T$ as well as $$\mass(S)\leq (k+1)\diam(\spt T)\int_{[0,1]\times X}t^k\,d(\mathscr{L}^1 \times \lVert T \rVert) = \diam(\spt T)\mass(T),$$
 as desired. Since this holds for every $k\geq 0$ the proof is complete.
\end{proof}

We conclude this subsection with the following lemma which will be used in the proof of Proposition~\ref{prop:homomorphisms-transporting} and also later in the article.

\begin{lemma}\label{lem:wenger-bounded-distance-to-support}
Let \(X\) be a complete metric space and \(T\in \bI_k(X)\) a cycle, where $k\geq 1$. If $X$ has \((\EI_k)\), then for every $\varepsilon>0$ there exists \(S\in \bI_{k+1}(X)\) with $\partial S=T$ and such that 
\begin{equation}\label{eq:good-filling-bound-on-spt-EI}
\mass(S)\leq \Fillvol_X(T) +\varepsilon\quad\text{and}\quad \spt S \subset B\bigl( \spt T, D \mass(T)^{\frac{1}{k}}\bigr)
\end{equation}
for some $D$ only depending on $k$ and on the constant in $(\EI_k)$. All minimal fillings of $T$ satisfy \eqref{eq:good-filling-bound-on-spt-EI}. 
\end{lemma}

The lemma follows directly from \cite[Lemma 3.4]{MR2153909} and its proof. By choosing $\varepsilon>0$ small enough we may assume that the filling $S$ satisfies the isoperimetric inequality $\mass(S)\leq D\mass(T)^{\frac{k+1}{k}}$.

\subsection{Mapping polyhedral currents to spaces with \texorpdfstring{$(\LC_k)$}{TEXT} or \texorpdfstring{$(\EI_k)$}{TEXT}}
The aim of this subsection is to prove the technical Proposition~\ref{prop:homomorphisms-transporting} below which will be important in the proofs of Theorems~\ref{thm:intro-undistorted-main} and \ref{thm:Def-thm-finite-Nagata-dim}. 

Let \(\Sigma\) be a simplicial complex, equipped with the length metric or the $\ell_2$-metric. We define the set \(\poly_k(\Sigma)\subset \bI_k(\Sigma)\) of polyhedral \(k\)-chains in $\Sigma$ as follows. One has \(P\in \poly_k(\Sigma)\) if and only if there exist finitely many \(\sigma_i\in \mathcal{F}_k\), \(\theta_i\in \Z\), and isometries \(\phi_i\colon \Delta \to \sigma_i\),
where $\Delta$ denotes the standard Euclidean $k$-simplex,
 such that \[
P=\sum \theta_i \,\phi_{i\,\#} \bb{\Delta}.
\]
Since \(\Delta\) can be embedded isometrically into \(\R^k\), the integral current \(\bb{\Delta}\) is well-defined. 
By construction, \(\poly_k(\Sigma)\subset \bI_k(\Sigma)\) is an additive subgroup with generating set \(\bigl\{ \bb{\sigma} : \sigma\in \mathcal{F}_k\bigr\}\). We use the convention that \(\bb{\sigma}\) denotes any of the currents \(\phi_{\#} \bb{\Delta}\), where \(\phi\colon \Delta \to \sigma\) is an isometry. Notice that \(\bb{\sigma}\) is uniquely determined up to a sign.

Given a Lipschitz map $h: \Sigma \to Y$ such that $h(\Sigma^{(0)})\subset X$, the following proposition constructs a chain homomorphism $\Lambda_* : \poly_*(\Sigma) \to \bI_*(X)$ and a chain homotopy $\Gamma_* : \poly_*(\Sigma) \to \bI_{*+1}(X)$ from $\Lambda_*$ to $h_\#$. This chain homotopy connects $\Lambda_*$ to $h_\#$ in the sense that if $T\in \poly_m(\Sigma)$ is a $m$-cycle, then $\Gamma_k(T)$ is an $(m+1)$-chain such that $\partial \Gamma_m(T) = \Lambda_m(T) - h_\#T$.

\begin{proposition}\label{prop:homomorphisms-transporting}
 Let $Y$ be complete metric space, \(X\subset Y\) a closed quasiconvex subset and \(\Sigma\) a simplicial complex equipped with the length metric or \(\ell_2\)-metric, both rescaled by a factor \(\varepsilon >0\). Suppose that $h\colon \Sigma\to Y$ is a map such that $h(\Sigma^{(0)})\subset X$ and $h|_\sigma$ is Lipschitz for every $\sigma\in \mathcal{F}$. Let $k\geq 0$ and suppose $Y$ has $(\EI_{k+1})$ and $X$ has \((\LC_k)\) or $(\EI_k)$. Then there exist homomorphisms $\Lambda_m\colon\poly_m(\Sigma)\to\bI_m(X)$ and $\Gamma_m\colon\poly_m(\Sigma)\to\bI_{m+1}(Y)$ for $m=0,\dots, k+1$ with the following properties:
 \begin{enumerate}
     \item $\Lambda_0 = h_\#$ and $\Gamma_0 = 0$,
     \item $\Lambda_m\circ\partial = \partial \circ \Lambda_{m+1}$ for $m=0,\dots, k$,
     \item $\partial\circ\Gamma_m = \Lambda_m - h_\# - \Gamma_{m-1}\circ \partial$ for $m=1,\dots, k+1$,
     \item For each $m=0,\dots, k+1$ and each \(\sigma\in \mathcal{F}_m\),  $$\mass\bigl(\Lambda_m(\bb{\sigma})\bigr)\leq C\cdot \bigl[\varepsilon \Lip(h|_\sigma)\bigr]^m$$ and $$\mass\bigl(\Gamma_m(\bb{\sigma})\bigr)\leq C\cdot\bigl[\varepsilon\Lip(h|_\sigma)\bigr]^{m+1},$$ 
     \item For each $m=0,\dots, k+1$ and \(\sigma\in \mathcal{F}_m\) the currents $\Lambda_m(\bb{\sigma})$ and $\Gamma_m(\bb{\sigma})$ have support in $B\bigl(h(\sigma^{(0)}),  C\varepsilon\Lip(h|_\sigma)\bigr)$.
 \end{enumerate}
 The constant $C$ depends only on $k$ and the data of $X$ and $Y$.
\end{proposition}
We use the convention that \(0^0=1\). 
\begin{proof}
By a simple scaling argument, it suffices to consider the case when \(\varepsilon=1\). We first construct the homomorphisms $\Lambda_m$. If $X$ has \((\LC_k)\), then there exists a map $\bar{h}\colon \Sigma^{(k+1)}\to X$ which agrees with $h$ on $\Sigma^{(0)}$ and which is $C\Lip(h|_\sigma)$-Lipschitz on every simplex \(\sigma\subset \Sigma^{(k+1)}\). In this case we define $\Lambda_m(\bb{\sigma})\coloneqq \bar{h}_\#\bb{\sigma}$ for every oriented $m$-simplex $\sigma\subset\Sigma$ with $0\leq m\leq k+1$ and extend $\Lambda_m$ linearly to $\poly_m(\Sigma)$. This yields a homomomorphism $\Lambda_m$ with the desired properties.

 Next, we construct $\Lambda_m$ in the case when $X$ has $(\EI_k)$. Set $\Lambda_0\coloneqq h_\#$. If  $\sigma$ is an oriented $1$-simplex in $\Sigma$ with vertices $e_-$ and $e_+$ we let $\Lambda_1(\bb{\sigma})$ be the integral current induced by a Lipschitz curve $\gamma$ in $X$ from $h(e_-)$ to $h(e_+)$ of length $\ell(\gamma)\leq C_0\, d(h(e_-),h(e_+))$, where $C_0$ is the quasiconvexity constant of $X$. In particular, we have $\partial\Lambda_1 (\bb{\sigma})= \Lambda_0(\bb{\partial \sigma})$ and 
 $$
 \mass(\Lambda_1(\bb{\sigma})\leq \ell(\gamma)\leq C_1 \Lip(h|_\sigma)
 $$ 
 and $\spt(\Lambda_1(\bb{\sigma}) \subset B\bigl(h(\sigma^{(0)}), C_1\Lip(h|_{\sigma^{(0)}})\bigr)$, where \(C_1\coloneqq \sqrt{2} C_0\). 
 Doing this for every oriented $1$-simplex $\sigma$ and extending linearly we obtain the desired homomorphism $\Lambda_1\colon\poly_1(\Sigma)\to\bI_1(X)$. If $k=0$ this finishes the construction of the homomorphisms $\Lambda_m$. If $k\geq 1$, then suppose we have defined $\Lambda_m$ with properties (2), (4), (5) for some $1\leq m\leq k$. Let $\sigma$ be an oriented $(m+1)$-simplex in $\Sigma$. It follows from (2) that $\Lambda_m(\partial\bb{\sigma})$ is a cycle in $\bI_m(X)$. Thus, by Lemma~\ref{lem:wenger-bounded-distance-to-support}, there exists $S\in\bI_{m+1}(X)$ such that $\partial S = \Lambda_m(\partial\bb{\sigma})$ and 
 $$
 \mass(S)\leq D \mass(\Lambda_m(\partial\bb{\sigma}))^{\frac{m+1}{m}}\leq D\Bigl[(m+2)C_m\Lip(h|_{\sigma})^m\Bigr]^{\frac{m+1}{m}}\leq C_{m+1} \Lip(h|_{\sigma})^{m+1}
 $$
 and 
 $$
 \spt(S) \subset B\Bigl(\spt(\Lambda_m(\partial\bb{\sigma})), D \mass\bigl(\Lambda_m(\partial \bb{\sigma})\bigr)^{\frac{1}{m}}\Bigr)\subset B\bigl(h(\sigma^{(0)}), C_{m+1}\Lip(h|_\sigma)\bigr),
 $$ 
 where \(C_{m+1}\coloneqq C_m+D\bigl( (m+2) C_m \bigr)^{\frac{m+1}{m}}\). 
 We set $\Lambda_{m+1}(\bb{\sigma})\coloneqq S$ and by defining $\Lambda_{m+1}(\bb{\sigma})$ like this for every oriented $(m+1)$-simplex and extending linearly we obtain a homomorphism $\Lambda_{m+1}\colon\poly_{m+1}(\Sigma)\to\bI_{m+1}(X)$ satisfying properties (2), (4), (5).
  
 We proceed analogously in order to construct $\Gamma_m$. Set $\Gamma_0\coloneqq 0$ and suppose we have already constructed $\Gamma_{m-1}$ with the desired properties for some $1\leq m\leq k+1$ . Let $\sigma$ be an oriented $m$-simplex in $\Sigma$. The current $P\coloneqq \Lambda_m(\bb{\sigma}) - h_\#\bb{\sigma} -\Gamma_{m-1}(\bb{\sigma})$ is a cycle and satisfies 
 $$
 \mass(P) \leq \mass(\Lambda_m(\bb{\sigma})) + \mass(h_\#\bb{\sigma}) + \mass(\Gamma_{m-1}(\bb{\sigma}))\leq C_m^{\prime\prime}\Lip(h|_\sigma)^m,
 $$ 
 where \(C_m^{\prime\prime}\coloneqq C_m+1+C_{m-1}^{\prime}\). Therefore, by  Lemma~\ref{lem:wenger-bounded-distance-to-support}, there exists $S\in\bI_{m+1}(Y)$ with $\partial S = P$ and 
 $$
 \mass(S) \leq D'\mass(P)^{\frac{m+1}{m}} \leq D' C^{\prime\prime}_m{^{\frac{m+1}{m}}}\Lip(h|_\sigma)^{m+1}\leq C^{\prime}_m \Lip(h|_\sigma)^{m+1}
 $$ 
 and 
 $$\spt(S)\subset B\bigl(\spt(P), D^\prime \mass(P)^{\frac{1}{m}}\bigr)\subset B\bigl(h(\sigma^{(0)}), C^{\prime}_m\Lip(h|_\sigma)\bigr),
 $$ 
 where \(C^{\prime}_m\coloneqq C_m^{\prime\prime}+D^\prime C^{\prime\prime}_m{^{\frac{m+1}{m}}}\). We put \(\Gamma_m(\bb{\sigma})\coloneqq S\).
 Doing this for all $\sigma$ and extending linearly yields the desired homomorphism $\Gamma_m$. It is clear that it satisfies all the desired properties. We set \(C\coloneqq \max\{ C_m \cdot C_m^{\prime} : m=0, \dots, k+1\}\). This completes the proof. 
\end{proof}

\section{Factorization through simplicial complexes}\label{sec:factorization}

In this section, we prove the factorization theorem sketched in Section~\ref{sec:proof-outline}. The theorem is obtained by a close inspection of the proofs of Lang and Schlichenmaier's results in \cite{MR2200122}.

\begin{theorem}\label{thm:factorization} 
Let \(Z\subset Y\) be complete metric spaces such that \(Z\) has Nagata dimension \(n\) and \(Y\) is Lipschitz \(n\)-connected for some integer \(n\geq 0\). Then there exist \(C\geq 1\), an \((n+1)\)-dimensional simplicial complex \(\Sigma\) equipped with the $\ell_2$-metric, and
maps \(g\colon Y\setminus Z \to \Sigma\)  and \(h\colon \Sigma\to Y\) with \(h\bigl(\Sigma^{(0)}\bigr)\subset Z\) such that the following holds:

\begin{enumerate}
\item\label{it:PropOfF}  \(f\colon Y \to Y\), defined by \(f=h\circ g\) on \(Y\setminus Z\) and \(f=\id_Z\) on \(Z\), is \(C\)-Lipschitz;

\item\label{it:PropOfG} \(g\) is \(C r^{-1} \)-Lipschitz on \(\bigr\{y\in Y : d(y, Z) > r\bigl\}\) for all \(r>0\); 

\item\label{it:PropOfH} \(h\) is \(C r_\sigma\)-Lipschitz on every \(\sigma\in \mathcal{F}\), where \(r_\sigma\coloneqq \inf\big\{ d(y, Z) : y\in g^{-1}(\st \sigma)\big\}\);
\item\label{it:R-sigma}  for every \(\sigma\in \mathcal{F}\) one has \(R_\sigma \leq C r_\sigma\), where \(R_\sigma\coloneqq \sup\big\{ d(y, Z) : y\in g^{-1}(\st \sigma)\big\}\).

\end{enumerate}
Moreover, if \(B\subset Y\) is a bounded subset intersecting \(Z\), then there exists a simplicial complex \(\Sigma^\prime\subset \Sigma\) such that \(\Sigma^\prime\) is \((n+1)\)-dimensional and \(C\)-quasiconvex, \(g(B\setminus Z)\subset \Sigma^\prime\) and \(h\) is \(C\diam(B)\)-Lipschitz on \(\Sigma^\prime\). 
\end{theorem}

The constant \(C\) depends only on the data of \(Z\) and \(Y\).
Theorem~\ref{thm:factorization} is a crucial component of the proof of Theorem~\ref{thm:intro-undistorted-main}. 
If $X\subset Y$  is a closed subset which contains $Z$ and is Lipschitz \(k\)-connected for some $k\geq 0$ then we may choose $h$ to furthermore satisfy $h\bigl(\Sigma^{(k+1)}\bigr)\subset X$. As a result, if \(Z\) is Lipschitz \(n\)-connected, then it follows that \(f\) is a \(C\)-Lipschitz retraction onto \(Z\). 

\begin{proof}
We may suppose that \(Y\setminus Z\neq \varnothing\). Following \cite[Theorems 1.6 and 5.2]{MR2200122}, we find an infinite index set \(I\), a covering \((B_i)_{i\in I}\) of \(Y\setminus Z\) by subsets of \(Y\setminus Z\), and constants \(\delta\in(0,1)\) and \(\alpha >0\) depending only on the data of \(Z\) such that the following holds. For all \(i\in I\), \(\diam(B_i)\leq \alpha d(B_i, Z)\) and for every \(y\in Y\setminus Z\) there are at most \(n+2\) indices \(i\in I\) such that \(\tau_i(y)>0\). Here, \(\tau_i\colon Y\setminus Z \to [0, +\infty)\) is defined by 
\[
\tau_i(y)=\max\bigl\{ \delta d(B_i, Z)-d(y, B_i), 0 \bigr\}.
\]
We set \(\Sigma\coloneqq \Sigma^{(n+1)}(I)\) and  \(\bar{\tau}(y)\coloneqq \sum_{i\in I} \tau_i(y)\). By the above, \(\Sigma\) is \((n+1)\)-dimensional and \(g\colon Y\setminus Z \to \Sigma\) assigning
\(y\mapsto \bar{\tau}(y)^{-1}\bigl(\tau_{i}(y)\bigr)_{i\in I}\)
is well-defined. 
In the following, we prove items \eqref{it:PropOfF} -- \eqref{it:R-sigma}. To begin, we show that \(g\colon Y\setminus Z\to \Sigma\) has the desired properties. Let \(y\), \(y^\prime\in Y \setminus Z\) and denote by \(K\subset I\) the set of all \(i\in I\) such that \(\tau_i(y)>0\) or \(\tau_i(y^\prime)>0\). Notice that \(K\) contains at most \(2(n+2)\) elements.
We estimate
\begin{align*}
\sum_{i\in K} \Big\lvert\frac{\tau_i(y)}{\bar{\tau}(y)}-\frac{\tau_i(y')}{\bar{\tau}(y')}\Bigr\rvert&= \sum_{i\in K} \frac{1}{\bar{\tau}(y)\bar{\tau}(y')} \Big\lvert\sum_{j\in K} \Bigl(\tau_i(y)\tau_j(y')-\tau_i(y')\tau_j(y)\Bigr)\Big\rvert \\
&\leq \frac{2}{\bar{\tau}(y)\bar{\tau}(y')} \sum_{i\in K} \sum_{j\in K} \tau_{j}(y') \abs{\tau_{i}(y)-\tau_i(y')} \\
&\leq \frac{4(n+2)}{\bar{\tau}(y)} d(y,y'),
\end{align*}
where in the last inequality we used that \(\tau_i\) is \(1\)-Lipschitz.
Since \((B_i)_{i\in I}\) covers \(Y\setminus Z\), there exists \(i\in I\) such that \(y\in B_i\). As a result,
\[
d(y,Z)\leq \diam(B_i)+d(B_i, Z)\leq (1+\alpha) d(B_i, Z)
\]
and \(\bar{\tau}(y)\geq \delta d(B_i, Z)\). Thus, \((1+\alpha) \bar{\tau}(y)\geq (1+\alpha) \delta d(B_i, Z) \geq \delta d(y, Z)\). By the above,
\begin{equation*}
\abs{g(y)-g(y')} \leq \frac{4(n+2)(1+\alpha)}{\delta} \frac{1}{d(y, Z)}d(y,y') .
\end{equation*}
This shows that \(g\) is \(C_1r^{-1}\)-Lipschitz on \(\bigl\{ y\in Y : d(y, Z) >r\bigr\}\). Next, we show \eqref{it:PropOfH} and \eqref{it:R-sigma}. For each \(i\in I\) select \(y_i\in B_i\) and \(x_i\in Z\) for which \(d(y_i, x_i) \leq 2 d(B_i, Z)\). Let \(h_{0}\colon \Sigma^{(0)}\to Z\) denote the map that sends \(e_i\) to \(x_i\) for each \(i\in I\). Here, \(e_i\in \ell_2(I)\) is defined by \((e_i)_j=1\) if \(j=i\) and \((e_i)_j=0\) otherwise. Notice that \(\Sigma^{(0)}=\bigl\{ e_i : i\in I\bigr\}\). Using that \(Y\) is Lipschitz \(n\)-connected, we obtain a constant \(C_2>0\) and an extension \(h\colon \Sigma\to Y\) of \(h_{0}\) such that \(h\) is \(C_2\Lip\bigl(h|_{\sigma^{(0)}}\bigr)\)-Lipschitz on each \(\sigma\in \mathcal{F}\). Let \(\sigma\in \mathcal{F}\) such that \(g^{-1}(\st \sigma)\neq \varnothing\). 
Notice that 
\[
r_\sigma=\inf\big\{ d\bigl(Z, g^{-1}(\interior \sigma')\bigr) : \sigma^\prime\in \mathcal{F} \text{ and } \sigma \subset \sigma^\prime\big\}.
\]
Thus, to show that \(h\) is \(C_3 r_\sigma\)-Lipschitz on \(\sigma\) it suffices to show that \(h\) is \(C_3r\)-Lipschitz on each  \(\sigma'\in \mathcal{F}\) with \(\sigma\subset \sigma'\), where \(r\coloneqq d\bigl(Z, g^{-1}(\interior \sigma')\bigr)\). Fix \(\sigma'\in \mathcal{F}\) and let \(e_i\) and \(e_j\) be two vertices of \(\sigma'\). We may suppose that \(d(Z, B_j)\leq d(Z, B_i)\) and there exists some \(y\in Y \setminus Z\) such that 
\(g(y)\in \interior{\sigma'}\). We estimate
\begin{align*}
d(h(e_i), h(e_j))&\leq d(x_i, y)+d(y, x_j) \\
&\leq 2 d(Z, B_i)+ d(y_i, y)+d(y_j,y)+ 2 d(Z, B_j).
\end{align*}
As \(g(y)\in \interior \sigma'\), we infer \(\tau_i(y)>0\) and thus \(d(y, B_i)\leq \delta d(Z, B_i)\). Therefore,
\(d(y_i, y)\leq \diam(B_i)+d(y, B_i)\leq \diam(B_i)+\delta d(Z, B_i)\), 
and we arrive at
\[
d(h(e_i), h(e_j))\leq 2(2+\delta+\alpha) d(Z, B_i).
\]
Since \((1-\delta) d(B_i, Z) \leq d(y, Z)\), we obtain
\[
d(h(e_i), h(e_j))\leq\sqrt{2}\, \biggl(\frac{2+\delta+\alpha}{1-\delta}\biggr)\, d(y,Z) \,\abs{e_i-e_j}.
\]
As \(y\in Y\setminus Z\) with \(g(y)\in \interior(\sigma')\) was arbitrary, this implies that \(h\) is \(C_3r\)-Lipschitz on \(\sigma'\) for \(r\coloneqq d\bigl(Z, g^{-1}(\interior \sigma')\bigr)\), as desired.
This concludes the proof of \eqref{it:PropOfH}. If \(e_i\) is a vertex of \(\sigma\), then we have
\[
d(y, Z)\leq  2d(Z, B_i)+d(y_i, y) \leq (2+\delta+\alpha) d(Z, B_i) \leq  \biggl(\frac{2+\delta+\alpha}{1-\delta}\biggr)\, r_\sigma,
\]
and so \(R_\sigma \leq C_3 r_\sigma\) , as \(y\in g^{-1}(\st \sigma)\) was arbitrary. This yields \eqref{it:R-sigma}.

Next, we show \eqref{it:PropOfF}. Notice that \(Y\) is \(c\)-quasiconvex for some \(c\geq 1\). We want to apply Lemma \ref{lem:littleLip} to show that \(f\) is \(C_4c\)-Lipschitz.  To this end, we prove that \(\lip f(y)\leq C_4\) for all \(y\in Y\). For any \(y\in Y\setminus Z\) there exists \(i\in I\) such that \(y\in B_i\) and thus using \eqref{it:PropOfH} we obtain \(d(f(y), x_i)=d(h(g(y)), h(e_i))\leq C_3\, d(y, Z) \sqrt{2}\), and so
\begin{equation}\label{eq:distanceToX}
d(f(y), y)\leq d(f(y), x_i)+d(x_i, y_i)+d(y_i, y)\leq \bigl(C_3\hspace{-0.2em}\sqrt{2} +2+\alpha\bigr)d(y, Z).
\end{equation}
By combining \(d(f(x), f(y))\leq d(x,y)+d(y, f(y))\), where \(x\in Z\),  with \eqref{eq:distanceToX}, we find that \(\lip f(x) \leq (C_3\hspace{-0.2em}\sqrt{2}+3+\alpha)\) for all \(x\in Z\). 
Now, fix \(y\in Y \setminus Z\) and choose \(\varepsilon>0\) such that \(d(y', Z)\leq 2 d(y, Z)\) for all \(y'\in U_\varepsilon(y)\), and whenever \(y'\in U_\varepsilon(y)\) is contained in \(g^{-1}(\interior(\sigma))\) for some \(\sigma\in \mathcal{F}\), then \(y\in g^{-1}(\sigma)\). Because of \eqref{it:PropOfG} and \eqref{it:PropOfH}, for all \(y'\in U_\varepsilon(y)\),
\[
d(f(y), f(y'))\leq C_3 \,d(y',Z) \, d(g(y), g(y'))\leq 2 C_3 C_1\, d(y, y'),
\]
and so \(\lip f(y) \leq 2 C_3 C_1\). Hence, we have shown that \(\lip f(y) \leq C_4\) for all \(y\in Y\). By Lemma \ref{lem:littleLip}, \(f\) is \(C_4 c\)-Lipschitz, as desired. 

To finish the proof, we show the statements of the moreover part. Let \(B\subset Y\) be a bounded subset intersecting \(Z\) and suppose that \(J\subset I\) is the subset of those indices \(i\in I\) for which there exists \(y\in B\) such that \(\tau_i(y)>0\). By enlarging \(B\) (if necessary) we may assume that \(J\) has infinitely many elements. 
We set \(\Sigma^\prime\coloneqq \Sigma^{(n+1)}(J)\). Clearly, \(\Sigma^\prime\subset \Sigma\) and \(\Sigma^\prime\) is \((n+1)\)-dimensional. By definition of \(g\), we have \(g(B\setminus Z)\subset \Sigma^\prime\). Next, we claim that \(\Sigma^\prime\) is \(C_5\)-quasiconvex for some constant \(C_5\) depending only on \(n\). Let \(\sigma_1\), \(\sigma_2\in \mathcal{F}(\Sigma^\prime)\). Clearly, there is \(\sigma\in\mathcal{F}\bigl(\Sigma(J)\bigr)\) of dimension at most \(2n+3\) such that \(\sigma_1\), \(\sigma_2\subset \sigma^{(n+1)}\). As \(\sigma\) is isometric to a Euclidean standard simplex, it follows that \(\sigma^{(n+1)}\) \(C_5\)-quasiconvex, and so, by using that \(\sigma^{(n+1)}\subset \Sigma^\prime\), we find that \(\Sigma^\prime\) is \(C_5\)-quasiconvex as well. Next, we show that \(h\) is \(C_6 \diam(B)\)-Lipschitz on \(\Sigma^\prime\). For each \(r\in J\) choose \(z_r\in B\) such that \(\tau_r(z_r)>0\). We estimate
\begin{align*}
d(h(e_{r}), h(e_{s})) &\leq d(x_r, y_r)+d(y_r, z_r)+d(z_r,z_s)+d(z_s, y_s)+d(y_s, x_s)\\
&\leq \diam(B)+2(2+\alpha+\delta)\max\bigl\{ d(B_r, Z), d(B_s, Z) \bigr\}, 
\end{align*}
for all \(r\), \(s\in J\). Fix \(y_0\in B\cap Z\). Since \((1-\delta) d(B_s, Z) \leq d(z_s, Z) \) for all \(s\in J\), it follows that
\[
d(h(e_{r}), h(e_{s})) \leq \diam(B)+2\, \biggl(\frac{2+\delta+\alpha}{1-\delta}\biggr)\diam(B);
\]
thus, by construction of \(h\), for each \(\sigma\in \mathcal{F}(\Sigma^\prime)\), the map \(h\) is \(C_6 \diam(B)\)-Lipschitz on \(\sigma\). Now, an argument as in the proof of \eqref{it:PropOfF} above yields \(\lip h(x)\leq C_6 \diam(B)\) for all \(x\in \Sigma^\prime\). Hence, as \(\Sigma^\prime\) is \(C_5\)-quasiconvex, Lemma \ref{lem:littleLip} tells us that \(h\) is \(C_5 C_6 \diam(B)\)-Lipschitz on \(\Sigma^\prime\), as desired. We put \(C\coloneqq \max\bigl\{C_i : i=1, \dots, 6\bigr\}\). 
\end{proof}

If $h$ is as in Theorem~\ref{thm:factorization} above, then $h$ is $Cr$-Lipschitz on every  $\sigma\in \mathcal{F}$ for which $g^{-1}(\st\sigma) \cap  \{ d(y, Z) \leq r\} \neq\varnothing$. Hence, the following proposition tells us that $\lip h|_{\Sigma'}(x)\leq Cr$ for every $x\in \Sigma'$, where $\Sigma'\coloneqq \Hull\big(g(\{y\in Y: 0< d(y, Z) \leq r\})\big)$.

\begin{proposition}\label{prop:h-general-little-lip}
Let $h\colon \Sigma\to Y$ be a map from a simplicial complex $\Sigma$ equipped with the \(\ell_2\)-metric to a metric space $Y$. Let $L>0$ and suppose $A\subset\Sigma$ is such that $h$ is $L$-Lipschitz on every $\sigma\in \mathcal{F}$ for which $\st \sigma \cap A\neq\varnothing$. Then the restriction of $h$ to $\Sigma'\coloneqq \Hull(A)$ satisfies $\lip h|_{\Sigma'}(x)\leq L$ for every $x\in \Sigma'$.
\end{proposition}

\begin{proof}
 For $y\in\Sigma$ denote by $\sigma(y)$ be the unique simplex in $\Sigma$ with \(y\in \interior \sigma(y)\). Notice that if $y\in\Sigma'$ then \(h|_{\sigma(y)}\) is \(L\)-Lipschitz because $\sigma(y)\subset\Sigma'$ and hence $\st \sigma(y)\cap A\neq\varnothing$.

Now, let \(x\in \Sigma'\). By Lemma~\ref{lem:GeometryOfDeltaN} there is a real number \(s_0>0\) such that \(x\in \sigma(y)\) for all \(y\in \Sigma'\) with \(\abs{x-y} <s_0\). Therefore, if $y\in\Sigma'$ satisfies \(0<\abs{x-y} <s<s_0\) then, by the above,
\[
\frac{d(h(x),h(y))}{s} \leq \frac{d(h(x),h(y))}{\abs{x-y}} \leq L.
\]
This implies that \(\lip h|_{\Sigma'}(x) \leq L\) for all \(x\in \Sigma'\), as was to be shown.
\end{proof}

\begin{corollary}\label{cor:triangulation-bilip-length}
Suppose that \(h\colon \Sigma \to X\) is a homeomorphism between a simplicial complex \(\Sigma\) equipped with the \(\ell_2\)-metric and a metric space \(X\).
 If \(h\) is \(D\)-bilipschitz on every \(\sigma\in \mathcal{F}\) and \(X\) is \(c\)-quasiconvex, then \(h\colon \Sigma\to X\) is \(cD\)-bilipschitz when \(\Sigma\) is equipped with the length metric.
\end{corollary}

\begin{proof}

Proposition~\ref{prop:h-general-little-lip} tells us that \(\lip h(z)\leq D\) for all \(z\in \Sigma\). Let \(z\), \(z'\in \Sigma\) and let \(\gamma\) be a curve in \(\Sigma\) connecting \(z\) and \(z'\). By Lemma~\ref{lem:littleLip}, it follows that 
\[
d(h(z), h(z'))\leq \ell(h\circ\gamma)\leq D \ell(\gamma).
\]
Consequently, \(h\) is \(D\)-Lipschitz with respect to the length metric \(d_i\) on \(\Sigma\). 
To finish the proof it remains to show that \(g\colon X\to (\Sigma, d_i)\) given by \(x\mapsto h^{-1}(x)\) is \(cD\)-Lipschitz.

For every \(z\in \Sigma\) denote by \(\sigma(z)\) the unique simplex \(\sigma\in \mathcal{F}\) such that \(z\in \interior \sigma\). Given \(x\in X\), we abbreviate \(\sigma(x)\coloneqq h\bigl(\sigma(z)\bigr)\), where \(z=g(x)\).
As \(h\) is a homeomorphism, for every \(x\in X\) there exists \(\varepsilon_x>0\) such that  \(x\in \sigma(x')\) whenever \(x'\in U(x,\varepsilon_x)\). See Lemma~\ref{lem:GeometryOfDeltaN}. Notice that \(g\) is \(D\)-Lipschitz on \(\sigma(x)\) for every \(x\in X\). Hence, exactly the same argument as in the proof of Proposition~\ref{prop:h-general-little-lip} shows that \(\lip g(x) \leq D\) for all \(x\in X\). Now, by invoking Lemma~\ref{lem:littleLip} we get that \(g\) is \(cD\)-Lipschitz, as \(X\) is \(c\)-quasiconvex. This completes the proof. 
\end{proof}

\section{Proof of Theorem~\ref{thm:intro-undistorted-main} and generalizations}\label{sec:isoperimetric-subspace-distortion}

In this section we prove Theorem~\ref{thm:intro-undistorted-main} from the introduction about isoperimetric subspace distortion. We also establish strengthenings and generalizations of this theorem which will be used in Section~\ref{sec:def-thm-finite-Nagata} to establish Theorem~\ref{thm:Def-thm-finite-Nagata-dim}, a deformation theorem in spaces of finite Nagata dimension.

\subsection{Undistorted fillings with controlled support}\label{Sec:variants-main-theorem}
In this subsection, we discuss two theorems on isoperimetric subspace distortion and deduce Theorem~\ref{thm:intro-undistorted-main} from them. The main difference between the following theorem and Theorem~\ref{thm:intro-undistorted-main} is that we do not assume that the whole subspace $X$ has finite Nagata dimension and that in addition we obtain control on the filling chain in the subspace $X$.

\begin{theorem}\label{thm:technical-version-main-thm}
 Let $Y$ be a complete metric space and let $Z\subset X\subset Y$ be closed subsets such that $Z$ is bounded and has finite Nagata dimension, and $X$ is quasiconvex and has \((\LC_k)\) or $(\EI_k)$ for some $k\geq 0$.  If $S\in\bI_{k+1}(Y)$ satisfies
 $\spt(\partial S)\subset Z$ and $\spt S \subset B(Z, \eta)$ for some $\eta>0$, then there exists $\bar{S}\in\bI_{k+1}(Y)$ which has support in $X$ and satisfies $\partial \bar{S} = \partial S$ as well as 
 \[
 \mass(\bar{S})\leq C\mass(S) \quad \text{ and } \quad \spt\bar{S} \subset B(Z, C\eta),
 \]
 where $C$ depends only on the data of $X$ and the Nagata dimension and constant of $Z$. Moreover, if $Y$ is a Banach space then there exists $W\in\bI_{k+2}(Y)$ with $\partial W = \bar{S} - S$ and such that $\mass(W)\leq \eta C\mass(S)$ and $\spt W\subset B(Z,C\eta)$.
\end{theorem}

Notice that the constant $C$ is independent of the diameter of $Z$. The hypothesis that $Z$ be bounded can be dropped if one assumes that the whole subspace $X$ has finite Nagata dimension, see Theorem~\ref{thm:undistorted-unbounded-set} below. Theorem~\ref{thm:technical-version-main-thm} will be proved in Subsection~\ref{subsec:proof-thm-tech-version-main}. As a consequence of the theorem we obtain the following strengthening of Corollary~\ref{cor:EI-implies-CI}.

\begin{corollary}\label{cor:isop-lip-implies-cone}
 Let $X$ be a complete quasiconvex metric space of finite Nagata dimension. Suppose that $X$ has \((\LC_k)\) or $(\EI_k)$ for some $k\geq 0$. Then every cycle $T\in\bI_m(X)$ of bounded support with $m=0,\dots, k$ admits a filling $S\in\bI_{m+1}(X)$ satisfying $$\mass(S)\leq C\diam(\spt T)\cdot \mass(T)\quad\text{ and }\quad\diam(\spt S)\leq C\diam(\spt T)$$ for some constant $C>0$ depending only on the data of $X$. In particular, $X$ has $(\CI_k)$.
\end{corollary}

The stronger version of the coning inequality asserted in the corollary is called {\it strong coning inequality} in \cite{KleinerStadler2019MoresQuasiflatsI} and \cite{KleinerStadler2020MoresQuasiflatsII} and plays an important role in these papers.

\begin{proof}
Let $m\in\{0,\dots, k\}$ be an integer and suppose that $T\in\bI_m(X)$ is a cycle of bounded support. View $X$ as a subset of $Y\coloneqq \ell_\infty(X)$ via a Kuratowski embedding. Fix $x_0\in \spt T$ and define $h\colon [0,1]\times Y\to Y$ by $h(t,y)\coloneqq (1-t)x_0 + ty$. The integral current $S\coloneqq h_\#(\bb{0,1}\times T)$ satisfies $\partial S = T$ and $\mass(S)\leq \diam(\spt T)\mass(T)$ as well as 
$$
\spt S\subset h\bigl([0,1]\times\spt T\bigr)\subset B\bigl(\spt T, \diam(\spt T)\bigr);
$$ 
see Section~\ref{subsec:homotopy-formula}. Applying Theorem~\ref{thm:technical-version-main-thm} with $Z=\spt T$, we obtain $\bar{S}\in\bI_{m+1}(X)$ with $\partial \bar{S}=T$ and  such that $$\mass(\bar{S}) \leq C\mass(S) \leq C\diam(\spt T)\mass(T)$$ as well as $$\spt\bar{S}\subset B\bigl(\spt T, C\diam(\spt T)\bigr)$$ for some constant $C$ only depending on the data of $X$. In particular, it follows that $\diam(\spt \bar{S})\leq (2C+1)\diam(\spt T)$. 
\end{proof}

Every complete metric space with $(\CI_k)$ also has $(\EI_k)$, as was shown in \cite{MR2153909}. Hence, by Corollary~\ref{cor:isop-lip-implies-cone}, if \(X\) is quasiconvex, has finite Nagata dimension and satisfies \((\LC_k)\) for some \(k\geq 1\), then \(X\) has \((\EI_k)\). In the following, we use this fact in combination with Theorem~\ref{thm:technical-version-main-thm}, to show that when \(X\) has finite Nagata dimension, then the conclusion of Theorem~\ref{thm:technical-version-main-thm} is also valid for currents \(S\in \bI_{k+1}(Y)\) for which the support of \(\partial S\) is unbounded.

\begin{theorem}\label{thm:undistorted-unbounded-set}
 Let $Y$ be a complete metric space and let $X\subset Y$ be a closed quasiconvex subset of finite Nagata dimension. Suppose furthermore that $X$ has \((\LC_k)\) or $(\EI_k)$ for some $k\geq 0$. Then for all subsets $Z\subset X$ and all $S\in\bI_{k+1}(Y)$ with $\spt \partial S\subset Z$ and $\spt S\subset B(Z,\eta)$ for some $\eta>0$, there exists $\bar{S}\in\bI_{k+1}(Y)$  which has support in $X$ and satisfies $\partial \bar{S} = \partial S$ as well as 
 \[
 \mass(\bar{S})\leq C\mass(S)\quad \text{ and } \quad \spt\bar{S} \subset B(Z, C\eta),
 \]
 where $C$ depends only on the data of $X$. Moreover, if $Y$ is a Banach space then there exists $W\in\bI_{k+2}(Y)$ with $\partial W=\bar{S}-S$ and such that $\mass(W)\leq \eta C\mass(S)$ and $\spt W\subset B(Z,C\eta)$.
\end{theorem}

\begin{proof}
We first treat the case $k\geq 1$. The space $X$ has $(\CI_k)$ by Corollary~\ref{cor:isop-lip-implies-cone} and hence also $(\EI_k)$ by \cite{MR2153909}. By possibly embedding $Y$ isometrically into $\ell_\infty(Y)$ we may further assume that $Y$ has $(\EI_k)$. Lemma~\ref{lem:wenger-bounded-distance-to-support} thus implies that for every $\eta>0$ the following holds. For every $\varepsilon>0$ sufficiently small and every cycle $U\in\bI_m(Y)$ with $1\leq m\leq k$ and $\mass(U)\leq\varepsilon$ there is a filling $V\in\bI_{m+1}(Y)$ with $\mass(V)\leq \varepsilon$ and $\spt V\subset B(\spt U,\eta)$; if $\spt U\subset X$ then $\spt V\subset X$.

Let $Z$, $S$, and $\eta$ be as in the statement of the theorem and set $T\coloneqq \partial S$. Fix $y_0\in Y$ and let $\rho\colon Y\to\R$ be the distance function from $y_0$. Notice that $\langle T,\rho,r\rangle$ and $\langle S,\rho,r\rangle$ are integral currents for almost every $r>0$ and
$$\partial(T\on \{\rho\leq r\}) = \langle T,\rho,r\rangle = - \partial \langle S,\rho,r\rangle.$$ 
Let $\varepsilon\in(0,1)$ be sufficiently small. There exists $r>0$ arbitrarily large such that the integral currents $\langle T,\rho,r\rangle$ and $\langle S,\rho,r\rangle$ are supported on $\spt T\cap\{\rho = r\}$ and $\spt S\cap \{\rho=r\}$, respectively, and satisfy $\mass(\langle T,\rho,r\rangle)\leq\varepsilon$ and $\mass(\langle S,\rho,r\rangle)\leq\varepsilon$. If $r>0$ was chosen large enough we may also assume that $\mass(T\on\{\rho>r\})\leq\varepsilon$ and $\mass(S\on\{\varrho>r\})\leq \varepsilon$. 

In the following, we consider the cases \(k=1\) and \(k\geq 2\) separately. Assume first that $k\geq 2$. For $s,t>0$ we abbreviate $A(s,t)\coloneqq B(Z,s)\cap B(y_0,t)$. By the observation made at the beginning of the proof there exists a filling $T'\in\bI_k(Y)$ of the cycle $\langle T,\rho,r\rangle$ with \(\mass(T')\leq \varepsilon\) and 
$\spt T'\subset X\cap A(\eta,r+\eta)$. Define cycles by $$T_1\coloneqq T\on\{\rho>r\} + T'\quad\text{and}\quad T_2\coloneqq T\on\{\rho\leq r\} - T'$$ and notice that $T=T_1+T_2$. We will find a filling of $T$ by constructing suitable fillings of $T_1$ and $T_2$. Since $T_1$ has mass at most $2\varepsilon$ and support in $X\cap B(Z,\eta)$ it has a filling $\bar{S}_1$ of mass at most $2\varepsilon$ and support in $X\cap B(Z,2\eta)$, again by the observation at the beginning of the proof. Next, we construct a filling of $T_2$. Notice that $T_2$ has support in the set $Z'\coloneqq X\cap A(\eta, r+5\eta)$. The cycle $ T' + \langle S,\rho,r\rangle$  is contained in \(\bI_k(Y)\), has mass  less than or equal to $ 2\varepsilon$ and is supported in $A(\eta, r+\eta)$, so it has a filling $S'\in\bI_{k+1}(Y)$ with $\mass(S')\leq 2\varepsilon$ and support in $A(2\eta,r+2\eta)$. Thus, $S_2\coloneqq S\on\{\rho\leq r\}-S'$ is a filling of $T_2$ satisfying $\spt S_2\subset A(2\eta, r+2\eta)\subset B(Z', 2\eta)$ and $\mass(S_2)\leq \mass(S) + 2\varepsilon$. Hence, by Theorem~\ref{thm:technical-version-main-thm}, there exists a filling $\bar{S}_2\in\bI_{k+1}(Y)$ of $T_2$ with support in $X$ such that $$\mass(\bar{S}_2)\leq C\mass(S_2)\leq C\mass(S) + 2C\varepsilon$$ and $\spt \bar{S}_2\subset B(Z',2C\eta)\subset B(Z,3C\eta)$ for some constant $C\geq 1$ only depending on the data of $X$. It follows that $\bar{S}\coloneqq \bar{S}_1 + \bar{S}_2$ is a filling of $T$ with support in $X$ and satisfying $\spt \bar{S}\subset B(Z,3C\eta)$ and $$\mass(\bar{S})\leq \mass(\bar{S}_1) + \mass(\bar{S}_2)\leq C\mass(S) + 2(C+1)\varepsilon\leq 2C\mass(S),$$ where the last inequality holds provided $\varepsilon>0$ was chosen small enough.

In order to prove the last statement in the theorem, suppose $Y$ is a Banach space. Set $R_1\coloneqq \bar{S}_1 - S' - S\on\{\varrho>r\}$ and notice that $R_1$ is an integral cycle of mass at most $5\varepsilon$ and with support in $B(Z,2\eta)$. Therefore, if $\varepsilon>0$ was chosen sufficiently small, then $R_1$ has a filling $W_1\in\bI_{k+2}(Y)$ with $\mass(W_1)\leq 5\varepsilon$ and $\spt W_1\subset B(Z,3\eta)$, compare with the observation at the beginning of the proof. Moreover, by Theorem~\ref{thm:technical-version-main-thm}, there exists $W_2\in \bI_{k+2}(Y)$ with $\partial W_2 = \bar{S}_2 - S_2$ and $\spt W_2\subset B(Z, 3C\eta)$ and such that $\mass(W_2)\leq 3\eta C\mass(S_2)$. Consequently, if $\varepsilon>0$ was chosen small enough, the integral current $W\coloneqq W_1+W_2$ satisfies $\partial W = \bar{S}-S$, has support in $B(Z,C\eta)$ and mass bounded by $\eta C\mass(S)$ for some $C$ only depending on the data of $X$. This concludes the proof of the case $k\geq 2$. 
 
If \(k=1\), then \(\mass( \langle T,\rho,r\rangle)\) is a non-negative integer satisfying \(\mass( \langle T,\rho,r\rangle)< \varepsilon\). Hence, \(\langle T,\rho,r\rangle=0\) and it follows that \(T_1= T\on\{\rho> r\}\), \(T_2=T\on \{ \rho \leq r\}\) and \(\langle S,\rho,r\rangle\) are integral cycles. Now, the same reasoning as in the proof of the case \(k\geq 2\) yields a filling \(\bar{S}\) of $T$ with the desired properties and, if $Y$ is a Banach space, a filling $W$ of $\bar{S}- S$ with the desired properties. This completes the proof when $k\geq 1$.

 We finally treat the remaining case $k=0$. Let $\rho$ be as above. Similarly to the case $k=1$ there exist arbitrarily large $r>0$ such that $\langle S,\rho,r\rangle=0$. Thus, there clearly exists $r$ such that the integral current $S'\coloneqq S\on\{\rho\leq r\}$ satisfies $\partial S'=\partial S$ and has bounded support. Therefore, applying Theorem~\ref{thm:technical-version-main-thm} to $S'$ we obtain a suitable filling of $\partial S$ in $X$. The moreover part can  easily be proved as above.
\end{proof}

We can finally prove Theorem~\ref{thm:intro-undistorted-main}.

\begin{proof}[Proof of Theorem~\ref{thm:intro-undistorted-main}]
Let $X$ and $Y$ be as in the statement of the theorem and let $T\in\bI_m(X)$ be a non-trivial cycle for some \(m=0, \dots, k\). By replacing $Y$ by $\ell_\infty(Y)$ we may assume that $Y$ is quasiconvex and has $(\EI_k)$. In particular, $\Fillvol_Y(T)$ is finite. To apply Theorem~\ref{thm:undistorted-unbounded-set}, we must find a filling $S\in\bI_{m+1}(Y)$ of $T$ with $\mass(S)\leq 2\Fillvol_Y(T)$ such that $\spt S\subset B(\spt T, \eta)$ for some finite $\eta>0$. When $m\geq 1$, Lemma~\ref{lem:wenger-bounded-distance-to-support} implies that such an $S$ exists. When $m=0$ then for any filling $S$ of $T$ with $\mass(S)\leq 2\Fillvol_Y(T)$ one constructs a filling of $T$  with smaller mass and bounded support by restricting $S$ to a suitable ball as at the end of the proof of Theorem~\ref{thm:undistorted-unbounded-set}. In either case, we apply Theorem~\ref{thm:undistorted-unbounded-set} to $S$ to show that there exists a filling $\bar{S}\in\bI_{m+1}(Y)$ of $T$ with support in $X$ such that $\mass(\bar{S})\leq C\mass(S)\leq 2C\Fillvol_Y(T)$ for some constant $C$ depending only on the data of $X$. This completes the proof.
\end{proof}

\subsection{Proof of Theorem~\ref{thm:technical-version-main-thm}}\label{subsec:proof-thm-tech-version-main}

Let $Z$, $X$, $Y$, $S$, $\eta$  be as in the statement of Theorem~\ref{thm:technical-version-main-thm}. We may of course assume that $\partial S\not= 0$ and thus $k\leq \dim_N(Z)$; see Section~\ref{sec:3.1-ref-nagata}.
By possibly embedding $Y$ isometrically into $\ell_\infty(Y)$ we may assume that $Y$ itself is a Banach space. 
In particular, $Y$ has \((\LC_n)\) and $(\EI_n)$ for all $n\in\N$, with constants only depending on $n$ (see \cite{MR2153909}). 
Let $\rho\colon Y\to\R$ be the distance function from the set $Z$.

\begin{lemma}\label{lem:bigger-space-to-fill}
 We may assume that
 \begin{equation}\label{eq:lem-vanishing-mass}
     \|S\|\big(\{\rho\leq r\}\big)\leq r\,(k+1)\mass(\partial S)
\quad \text{ and } \quad
\mass\bigl(\partial(S\on \{\rho>r\})\bigr)\leq \mass(\partial S)
 \end{equation}
 for almost every $r>0$ small enough.
\end{lemma}

\begin{proof} 
 It is clearly enough to find a metric space $\hat{Y}$ containing $Y$ and a current $\hat{S}\in\bI_{k+1}(\hat{Y})$ satisfying the properties stated in the lemma and with $\partial \hat{S} = \partial S$, $\mass(\hat{S})\leq 2\mass(S)$ as well as $\spt\hat{S}\subset B(Z,2\eta)$. 
 
 Let $0<\varepsilon<\eta$ be so small that $\varepsilon(k+1)\mass(\partial S)\leq \mass(S)$.
 We define \(\hat{Y} \coloneqq [0,\varepsilon]\times Y\) and identify $Y$ with the subset $\{0\}\times Y\subset \hat{Y}$. We set
 \begin{equation}\label{eq:DefinitionOfS}
\hat{S}\coloneqq \llbracket \varepsilon\rrbracket \times S-\bb{0, \varepsilon} \times \partial S.
\end{equation}
 Clearly, \(\hat{S}\in \bI_{k+1}(\hat{Y})\) and \(\partial \hat{S}=\llbracket 0 \rrbracket \times \partial S\). It follows from Lemma~\ref{lem:mass-of-pushforward-of-product} that $$\mass(\hat{S})\leq \mass(S) + \varepsilon(k+1)\mass(\partial S) \leq 2\mass(S).$$ Since \(\spt\hat{S}\subset [0,\varepsilon]\times \spt S\), we infer $\spt\hat{S} \subset B(Z, 2\eta)$.  We now show that \(\hat{S}\) satisfies \eqref{eq:lem-vanishing-mass}. Let \(\hat{\rho}\colon \hat{Y}\to \R\) denote the distance function from \(Z\). We claim that 
 \begin{equation}\label{eq:slices-1}
 \hat{S}\on \{\hat{\rho} \leq r\}=-\bb{0,r}\times\partial S
 \end{equation}
 for every \(r\in (0, \varepsilon)\). To simplify the notation, we set \(Q_r\coloneqq \bb{0, r}\times \partial S\). To establish \eqref{eq:slices-1}, it suffices to show that \(Q_\varepsilon \on \{\hat{\rho}\leq r\}=Q_r\), as \(\spt \bigl(\bb{\varepsilon}\times S\bigr)  \subset \{\varepsilon\}\times \spt S\) and thus \(\bigl(\bb{\varepsilon}\times S\bigr)\on \{\hat{\rho}\leq r\}=0\) for all \(r\in (0,\varepsilon)\). Choose a sequence \(\ell_j\colon \R\to [0,1]\) of Lipschitz functions such that \(\ell_j\to \mathbbm{1}_{[0,r]}\) in \(L^1(\R)\) as \(j\to \infty\).
 Letting \(\hat{\rho}_j\coloneqq \ell_j\circ \hat{\rho}\) we observe that \(\hat{\rho}_j\to \mathbbm{1}_{\{\hat{\rho}\leq r\}}\) in \(L^1\bigl(\hat{Y}, \lVert Q_\varepsilon \rVert\bigr)\) as \(j\to \infty\). Moreover, using that \(\hat{\rho}_{jt}=\ell_j(t)\) on \(Z\), we find that 
 \(Q_\varepsilon( \hat{\rho}_j\,\pi_0, \pi_1, \dots, \pi_{k+1})\) is equal to 
 \[
 \sum_{i=1}^{k+1}(-1)^{i+1}\int_{0}^\varepsilon  \ell_j(t) \cdot \partial S\Bigl(\pi_{0t} \frac{\partial \pi_{it}}{\partial t}, \pi_{1t}, \dots, \pi_{(i-1)t}, \pi_{(i+1)t}, \dots, \pi_{(k+1)t} \Bigr)\, dt
 \]
 for all \((\pi_0, \dots, \pi_{k+1})\in \mathcal{D}^{k+1}(\hat{Y})\).
Hence, by letting \(j\to \infty\), we obtain \(Q_\varepsilon \on \{\hat{\rho}\leq r\}=Q_r\) and \eqref{eq:slices-1} follows.  Because of \eqref{eq:slices-1}, we have
\begin{align*}
\lVert \hat{S} \rVert \bigl(\{ \hat{\rho} \leq r\}\bigr)&=\mass(\bb{0,r}\times \partial S)\leq r\, (k+1) \mass(\partial S),\\
\mass\bigl(\partial(\hat{S}\on \{\hat{\rho}>r\})\bigr)&= \mass\bigl(\bb{r}\times \partial S)=\mass(\partial S)
\end{align*}
for every \(r\in (0, \varepsilon)\). This completes the proof.
\end{proof}

Next, denote by $n$ the Nagata dimension of $Z$. By applying Theorem~\ref{thm:factorization} with $B=B(Z,\eta)$, we obtain an \((n+1)\)-dimensional simplicial complex \(\Sigma\) equipped with the \(\ell_2\)-metric and maps $f\colon Y\to Y$, $g\colon Y\setminus Z\to \Sigma$, and $h\colon \Sigma\to Y$ with $h(\Sigma^{(0)})\subset Z$ satisfying the properties listed in the statement of that theorem. In particular, \(\Sigma\) is quasiconvex, $g(B\setminus Z)\subset \Sigma$ and $h$ is Lipschitz on $\Sigma$. For the remainder of this section we denote by \(C\) the constant of Theorem~\ref{thm:factorization}. Notice that \(C\) does not depend on the diameter of \(Z\), but the Lipschitz constant of $h$ might.

For almost every $r>0$ the current $S_r\coloneqq S\on \{\rho>r\}$ belongs to $\bI_{k+1}(Y)$ and $\spt(\partial S_r)\subset \{\rho=r\}$. 
Fix a decreasing sequence \((r_i)\) of positive real numbers converging to zero such that \(S_{\hspace{-0.1em}r_i}\in \bI_{k+1}(Y)\) and \eqref{eq:lem-vanishing-mass} holds with $r=r_i$ for all \(i\geq 1\). Notice that $S_{\hspace{-0.1em}r_i}$ is supported on $\{\rho\geq r_i\}$ and $g$ is $Cr_{i}^{-1}$-Lipschitz on this set, so the current $S_i'\coloneqq g_\#S_{\hspace{-0.1em}r_i}$ is well-defined and an element of $\bI_{k+1}(\Sigma)$. Let \(P_i^\prime \in\poly_{k+1}(\Sigma^{(k+1)})\), \(R_i^\prime\in \bI_{k+1}(\Sigma)\), \(Q_i^\prime\in \bI_{k+2}(\Sigma)\) with  
\begin{equation}\label{eq:FFSplitting}
S_i^\prime=P_i^\prime+R_i^\prime+\partial Q_i^\prime
\end{equation}
be the integral currents obtained by applying the deformation theorem (Theorem \ref{thm:Federer-Fleming}) to \(S_i^\prime\).
Further, let $\Lambda_m$ and $\Gamma_m$ be the homomorphisms from Proposition~\ref{prop:homomorphisms-transporting} and define
\begin{equation}\label{eq:definition-of-Pi}
P_i\coloneqq \Lambda_{k+1}(P_i')\in\bI_{k+1}(X).
\end{equation}
In what follows, we will show that the sequence $(P_i)$ converges to a filling of $T$ which has the desired properties.

To keep track of the constants, we denote by \(K\) the maximum of the constants appearing in Theorem~\ref{thm:Federer-Fleming} and  Proposition~\ref{prop:homomorphisms-transporting}. By construction, \(K\) depends only on the data of \(X\) and on the Nagata dimension and constant of $Z$. The next two lemmas show that the \(P_i\)'s defined in \eqref{eq:definition-of-Pi} have uniformly bounded mass and boundary mass and the sequence \((\partial P_i)\) converges weakly to \(\partial S\). 

\begin{lemma}\label{lem:bound-P_r}
For all sufficiently large  \(i\geq 1\) we have $\spt P_i\subset B(Z,\bar{C}\eta)$ as well as
 $$
 \mass(P_i)\leq \bar{C}\mass(S)\quad\text{and}\quad \mass(\partial P_i)\leq \bar{C}\mass(\partial S),
 $$
 where \(\bar{C}\) is a constant depending only on \(n\), \(C\) and \(K\). 
\end{lemma}

\begin{proof}
 For each \(\sigma\in \mathcal{F}_{k+1}\) for which $P_i'\on\sigma\neq 0$ we have
 \begin{equation*}
         0\neq \mass(P_i'\on \sigma) \leq K\mass(S_i'\on\st \sigma)= K\mass\bigl(g_\#(S_{\hspace{-0.1em}r_i}\on g^{-1}(\st \sigma))\bigr).
 \end{equation*}
 In particular, $g^{-1}(\st \sigma)\cap B(Z,\eta)$ is nonempty, and so $r_\sigma\coloneqq d(Z, g^{-1}(\st \sigma))$ satisfies $r_\sigma\leq \eta$.  It furthermore follows from Theorem~\ref{thm:factorization}(4) that $r_\sigma>0$.
 Since $g^{-1}(\st \sigma)\subset \{\rho\geq r_\sigma\}$ and $g$ is $Cr_\sigma^{-1}$-Lipschitz on $\{\rho\geq r_\sigma\}$ it follows with the above that
 \[
 \mass(P_i'\on \sigma) \leq K\,C^{k+1} r_\sigma^{-(k+1)} \mass\bigl(S_{\hspace{-0.1em}r_i}\on g^{-1}(\st \sigma)\bigr),
 \]
Since $h|_\sigma$ is $Cr_\sigma$-Lipschitz, we conclude 
 \begin{align*}
  \mass\bigl(\Lambda_{k+1}(P_i'\on\sigma)\bigr) &\leq K C^{k+1}\, r_\sigma^{k+1}\mass(P_i'\on\sigma)\\
  &\leq K^2 C^{2(k+1)}\mass\bigl(S_{\hspace{-0.1em}r_i}\on g^{-1}(\st \sigma)\bigr).
  \end{align*}
  Summing over all \(\sigma\in \mathcal{F}_{k+1}\) and observing that each open simplex in $\Sigma$ is in the open star of at most $C_{n,k}$ different $(k+1)$-simplices we obtain
 \begin{align*}
      \mass(P_i)\leq C_{n,k} K^2 C^{2(k+1)}\mass(S_{\hspace{-0.1em}r_i}),
 \end{align*}
 and hence $\mass(P_i)\leq \bar{C}\mass(S)$ for all $i$ large enough, where \(\bar{C}\coloneqq 2C_{n,k} K^2 C^{2(k+1)}\). The same argument as above shows that for each \(\sigma\in \mathcal{F}_k\) we have
 \[
 \mass\bigl(\Lambda_k\bigl((\partial P_i')\on\sigma\bigr)\bigr)\leq K^2 C^{2k} \mass\bigl((\partial S_{\hspace{-0.1em}r_i})\on g^{-1}(\st \sigma)\bigr).
 \]
 From this one concludes as above that $$\mass(\partial P_i) \leq C_{n,k} K^2 C^{2k} \mass(\partial S_{\hspace{-0.1em}r_i})\leq \bar{C}\mass(\partial S).$$
Finally, for every $\sigma$ for which $P'_i\on\sigma\not=0$ we have $\Lip(h|_{\sigma})\leq Cr_\sigma\leq C\eta$ and therefore $$\spt(\Lambda_{k+1}(P'_i\on \sigma))\subset B(h(\sigma^{(0)}),K C\eta)\subset B(Z, \bar{C}\eta).$$ This shows that $\spt P_i\subset B(Z,\bar{C}\eta)$ and completes the proof.
\end{proof}

\begin{lemma}\label{lem:fillvol-T-bdry-P_r}
We have $\Fillvol_Y(\partial S -\partial P_i)\to 0$ as \(i\to \infty\). In particular, the sequence $(\partial P_i)$ converges weakly to $\partial S$.
\end{lemma}

\begin{proof}
 Since $\partial P'_i = \partial S_i' - \partial R_i'$, we find $$h_\#(\partial P_i') = \partial(h_\#S_i') - \partial(h_\#R_i') = \partial(f_\#S_{\hspace{-0.1em}r_i}) - \partial(h_\#R_i')$$ and thus, by Proposition~\ref{prop:homomorphisms-transporting}, 
 \begin{multline*}
         \partial S-\partial P_i = \partial S- \Lambda_k(\partial P_i')
         = \partial S - h_\#(\partial P_i') - \partial \Gamma_k(\partial P_i')\\
         = \partial(f_\#(S-S_{\hspace{-0.1em}r_i})) + \partial(h_\#R_i') - \partial \Gamma_k(\partial P_i') =: \partial A_i.
 \end{multline*}
 It follows that 
 $$
 \Fillvol_Y(\partial S-\partial P_i) \leq \mass(A_i) \le \mass(f_\#(S-S_{\hspace{-0.1em}r_i})) + \mass(h_\#R_i') + \mass(\Gamma_k(\partial P_i')).
 $$ 
 We claim that there is a constant \(\bar{C}>0\) such that for every $i$ each term on the right hand side in the inequality above is bounded by $\bar{C} r_i\mass(\partial S)$. 
 Clearly, we have $$
 \mass(f_\#(S-S_{\hspace{-0.1em}r_i}))\leq C^{k+1} \mass(S-S_{\hspace{-0.1em}r_i}) = C^{k+1} \|S\|(\{\rho\leq r_i\}) \leq C^{k+1}(k+1)\,r_i\mass(\partial S),
 $$
 where we have used \eqref{eq:lem-vanishing-mass} in the last inequality.
 Next, set $M_i\coloneqq \spt R_i'$ and observe that 
 \begin{equation}\label{eq:hull-containment}
     M_i\subset\Hull(\spt\partial S_i')\subset\Hull\bigl(g(\spt\partial S_{\hspace{-0.1em}r_i})\bigr)\subset \Hull\bigl(g(\{\rho= r_i\})\bigr).
 \end{equation}
Thanks to Theorem~\ref{thm:factorization}\eqref{it:PropOfH}, \(h\) is \(Cr_i\)-Lipschitz on every \(\sigma\in \mathcal{F}\) for which the set \( \st \sigma \cap g\bigl(\{\rho =r_i\}\bigr)\) is nonempty. 
It thus follows from \eqref{eq:hull-containment} and Proposition~\ref{prop:h-general-little-lip} that $\lip(h|_{M_i})(x)\leq Cr_i$ for all $x\in M_i$, 
and so, by Lemma~\ref{lem:rajala-wenger}, we get
$$
\mass(h_\#R_i')\leq C^{k+1}r_i^{k+1}\mass(R_i')\leq KC^{k+1}r_i^{k+1}\mass(\partial S_i')
$$
and hence $\mass(h_\#R_i')\leq KC^{2k+1}r_i\mass(\partial S).$
Finally, one calculates exactly as in the proof of Lemma~\ref{lem:bound-P_r} that for each \(\sigma\in \mathcal{F}_k\) 
$$
\mass\bigl(\Gamma_k((\partial P_i')\on\sigma)\bigr)\leq K^2C^{2k}\, r_i\mass\bigl((\partial S_{\hspace{-0.1em}r_i})\on g^{-1}(\st \sigma)\bigr)
$$
and 
$$
\mass\bigl(\Gamma_k(\partial P_i')\bigr)\leq C_{n,k}K^2C^{2k} r_i \mass(\partial S).
$$
Thus, $\mass(A_i)\to 0$ as $i$ goes to infinity. This proves the claim and completes the proof.
\end{proof}

 The proof of all except the last statement of Theorem~\ref{thm:technical-version-main-thm} can now easily be concluded if we can show that $(P_i)$ has a weakly converging subsequence, which we will show in the following lemma. Indeed, if a subsequence of $(P_i)$ converges weakly to $\bar{S}$, then $\bar{S}$ belongs to $\bI_{k+1}(X)$ and satisfies 
  \[
 \mass(\bar{S})\leq \bar{C}\mass(S) \quad \text{ and } \quad \spt\bar{S} \subset B(Z, \bar{C}\eta),
 \]
as follows from Lemma~\ref{lem:bound-P_r}, the closure theorem \cite[Theorem~8.5]{MR1794185}, and the lower semi-continuity of mass. Moreover, we have $\partial \bar{S} = \partial S$ thanks to Lemma~\ref{lem:fillvol-T-bdry-P_r}, and thus $\bar{S}$ satisfies the properties stated in Theorem~\ref{thm:technical-version-main-thm}.

\begin{lemma}\label{lem:cauchy-seq}
The sequence \((P_{i})\) has a weakly converging subsequence.
\end{lemma}

 Notice that when $X$ is proper the lemma follows directly from Lemma~\ref{lem:bound-P_r} and the compactness theorem \cite[Theorem~5.2]{MR1794185}. 

\begin{proof}
 Fix $r>0$, let \(\mathcal{A}_{r}\) be the collection of all \(\sigma\in \mathcal{F}_{k+1}\) satisfying \(r_\sigma > r\), and denote by $A_r$ the union of all simplices in $\mathcal{A}_r$. We first claim that the set \(\bigr\{ P_i^\prime \on A_r : i\geq 1 \bigl\}\) contains only finitely many elements. To see this, let $\sigma\in\mathcal{A}_r$ and notice that $P^\prime_i\on\sigma = \theta_i^\sigma\llbracket \sigma\rrbracket$ for some $\theta_i^\sigma\in\Z$. By \eqref{eq:MassOfFace} we have 
 \begin{equation}\label{eq:lem-Cauchy-P_i}
  \mathscr{L}(\Delta)\cdot|\theta_i^\sigma| = \mass(P^\prime_i\on\sigma)\leq K\, \mass\bigl(S^\prime_{i}\on \st \sigma\bigr)=: KM_i^\sigma,
 \end{equation}
 where \(\mathscr{L}(\Delta)\) denotes the Lebesgue measure of the Euclidean \((k+1)\)-simplex \(\Delta\). By construction, there exists $i_0\geq 1$ depending on $r$ such that for all $\sigma\in\mathcal{A}_r$ and all $i,j\geq i_0$ we have $S'_i\on \st\sigma = S'_j\on\st\sigma$. In particular, $M_i^\sigma = M_{i_0}^\sigma =:M^\sigma$ for all $i\geq i_0$. Furthermore, since $M^\sigma\neq 0$ for at most countably many $\sigma\in\mathcal{A}_r$, we see that $$\sum_{\sigma\in\mathcal{A}_r}M^\sigma=\sum_{\sigma\in\mathcal{A}_r}\mass\bigl(S^\prime_{i_0}\on \st \sigma\bigr)\leq C_{n,k}\sum_{\tau\in\mathcal{F}\hspace{-0.2em},\, r_\tau>r}\mass\bigl(S^\prime_{i_0}\on \interior\tau\bigr)\leq C_{n,k}\mass(S^\prime_{i_0})<\infty,$$ where we have used that each open simplex in $\Sigma$ is in the open star of at most $C_{n,k}$ different $(k+1)$-simplices. Thus, the collection $\mathcal{C}\coloneqq \bigl\{\sigma\in\mathcal{A}_r: KM^\sigma\geq \mathscr{L}(\Delta)\bigr\}$ is finite and letting $M\coloneqq\sup_{\sigma'\in \mathcal{A}_r}M^{\sigma'}<\infty$,  by virtue of \eqref{eq:lem-Cauchy-P_i} we find that for all $i\geq i_0$, $\theta_i^\sigma=0$ if $\sigma\not\in\mathcal{C}$ and $|\theta_i^\sigma|\leq KM$ for all $\sigma\in\mathcal{C}$. This implies our first claim.

Next, we claim that for almost every $r>0$ small enough and every $i$ sufficiently large we have $$\sum_{\sigma\in\mathcal{F}_{k+1}\setminus\mathcal{A}_r}\mass(\Lambda_{k+1}(P^\prime_i\on\sigma)) \leq \bar{C}r\mass(\partial S),$$ where $\bar{C}$ only depends on $n$, $C$, and $K$. Indeed, the proof of Lemma~\ref{lem:bound-P_r} shows that 
$$
\mass\bigl(\Lambda_{k+1}(P_i'\on\sigma)\bigr)
  \leq K^2 C^{2(k+1)}\mass\bigl(S_{\hspace{-0.1em}r_i}\on g^{-1}(\st \sigma)\bigr).
$$ 
Summing over all $\sigma\in\mathcal{F}_{k+1}\setminus\mathcal{A}_r$ and noting that $g^{-1}(\st\sigma)\subset\{\varrho\leq Cr\}$, because of Theorem~\ref{thm:factorization}\eqref{it:R-sigma}, yields 
  \begin{equation*}
          \sum_{\sigma\in\mathcal{F}_{k+1}\setminus\mathcal{A}_r}\mass(\Lambda_{k+1}(P^\prime_i\on\sigma)) \leq C_{n,k}K^2C^{2(k+1)}\lVert S_{r_i}\rVert(\{\varrho\leq Cr\}),
  \end{equation*}
and together with Lemma~\ref{lem:bigger-space-to-fill} we obtain the second claim.

Finally, it follows from the first claim and a standard diagonal sequence argument, that \((P_{i}^\prime)\) has a subsequence \((P^\prime_{i(m)})\) with the following property. For every \(r>0\) we have
\[
P^\prime_{i(m_1)}-P^\prime_{i(m_2)}=\sum_{\sigma\in \mathcal{F}_{k+1}\setminus\mathcal{A}_r} \bigl(P^\prime_{i(m_1)}-P^\prime_{i(m_2)}\bigr)\on \sigma
 \]
for all \(m_1\), \(m_2\) sufficiently large (depending on $r$).
This together with the second claim shows that for every $r>0$ sufficiently small, the mass of \(P_{i(m_1)}-P_{i(m_2)}\) is bounded by
\begin{equation*}
    \begin{split}
        \sum_{\sigma\in \mathcal{F}_{k+1}\setminus\mathcal{A}_r} \Bigl[\mass(\Lambda_{k+1}(P^\prime_{i(m_1)}\on\sigma)) + \mass(\Lambda_{k+1}(P^\prime_{i(m_2)}\on \sigma))\Bigr]\leq 2\bar{C}r\mass(\partial S)
    \end{split}
\end{equation*}
whenever $m_1,m_2$ are sufficiently large.
Consequently, \((P_{i(m)})\) is a Cauchy sequence in \(\bM_{k+1}(X)\) and, in particular, is a weakly converging subsequence of $(P_i)$.
\end{proof}

The following lemma proves the last statement in Theorem~\ref{thm:technical-version-main-thm}.

\begin{lemma}
 If $Y$ is a Banach space then there exists $W\in\bI_{k+2}(Y)$ with $\partial W = \bar{S} - S$ and such that $$\mass(W)\leq \eta\bar{C}\mass(S)\quad \text{ and } \quad \spt W\subset B(Z,\bar{C}\eta)$$ for some $\bar{C}$ depending only on $n$, $C$ and $K$.
\end{lemma}

\begin{proof}
 Let $H\colon [0,1]\times Y\to Y$ be the straight line homotopy defined by $H(t,y) = (1-t)y + tf(y)$.  Since $\spt\partial S\subset Z$ and $H(t,z) = z$ for every $z\in Z$ it follows that $H_\#(\bb{0,1} \times \partial S) =0$ and thus, by Proposition~\ref{prop:homotopy}, the integral current $V\coloneqq H_\#(\bb{0,1}\times S)$ satisfies $\partial V = f_\#S - S$. We have $d(y,f(y))\leq (C+1)\eta$ for every $y\in B(Z,\eta)$ and thus, by  Lemma~\ref{lem:mass-of-pushforward-of-product},  
 $$
 \mass(V)\leq (k+2)C^{k+1}(C+1)\eta\mass(S)
 $$ 
 as well as 
 $$
 \spt V\subset H([0,1]\times\spt S)\subset B(Z, (C+1)C\eta).
 $$
 Fix $i$ sufficiently large, to be determined later. By \eqref{eq:FFSplitting} and Proposition~\ref{prop:homomorphisms-transporting} we have 
 $$P_i = f_\#S_{r_i} + \Gamma_k(\partial P'_i) - h_\#R'_i + \partial\Gamma_{k+1}(P'_i) - \partial h_\#Q'_i = f_\#S - A_i + \partial\Gamma_{k+1}(P'_i) - \partial h_\#Q'_i,$$ 
 where $A_i$ is as in the proof of Lemma~\ref{lem:fillvol-T-bdry-P_r}. Hence, 
 \begin{equation}\label{eq:Sbar-S}
 \bar{S} - S = (\bar{S} - P_i) - A_i + \partial\Gamma_{k+1}(P'_i) - \partial h_\#Q'_i + \partial V.
 \end{equation}
 Using the same arguments as in the proof of Lemma~\ref{lem:bound-P_r} one easily shows that $\Gamma_{k+1}(P'_i)$ and $h_\#Q'_i$ are supported in $B(Z,\bar{C}\eta)$ and satisfy $\mass(h_\#Q'_i)\leq \eta\bar{C}\mass(S)$ as well as 
 $
 \mass(\Gamma_{k+1}(P'_i))\leq \bar{C}\eta\mass(S)
 $
 for some $\bar{C}$ depending on $n$, $C$ and $K$.

 It follows from the proof of Lemma~\ref{lem:cauchy-seq} that, after passing to a subsequence, $P_i$ converges to $\bar{S}$ in mass, so we can take $\mass(\bar{S} - P_i)$ arbitrarily small. By the proof of Lemma~\ref{lem:fillvol-T-bdry-P_r}, $\mass(A_i)$ converges to zero as $i$ tends to infinity. 
 Moreover, \eqref{eq:Sbar-S} implies that \(\bar{S} - P_i - A_i\) is a cycle whose support is contained in $B(Z,\bar{C}\eta)$. Hence, by Lemma~\ref{lem:wenger-bounded-distance-to-support} and the remark following it, we can choose an arbitrarily large $i$ such that $\bar{S} - P_i - A_i$ has a filling $U\in\bI_{k+2}(Y)$ satisfying $\mass(U)\leq \eta\bar{C}\mass(S)$ and $\spt U\subset B(Z,2\bar{C}\eta)$. Consequently, the integral current $W\coloneqq U+\Gamma_{k+1}(P'_i) -h_\#Q'_i + V$ satisfies $\partial W=\bar{S} - S$ as well as $\mass(W) \leq \eta\bar{C}\mass(S)$ and $\spt W\subset B(Z,\bar{C}\eta)$ for some $\bar{C}$ depending on $n$, $C$ and $K$. 
\end{proof}

\section{Approximating finite-dimensional spaces by simplicical complexes}\label{sec:approx-simplicial-complexes}

The aim of this section is to prove Theorem~\ref{thm:factorization-side-length-r} from the introduction. To this end, we first establish the following approximation result:

\begin{proposition}\label{prop:factorization2}
 Let \(X\subset Y\) be metric spaces such that \(X\) has Nagata dimension \(\leq n\) and \(Y\) is Lipschitz \((n-1)\)-connected for some \(n\geq 0\). Then there is a constant \(C\) such that for every \(\varepsilon>0\) there exist a simplicial complex \(\Sigma\) equipped with the \(\ell_2\)-metric and maps \(\psi\colon X\to \Sigma\) and \(\phi\colon \Sigma\to Y\) with the following properties:
 \begin{enumerate}
 
\item\label{it:containment} \(\phi(\Sigma^{(0)})\subset X\), \(\Hull\bigl(\psi(X)\bigr)=\Sigma\) and \(\Sigma\) has dimension \(\leq n\);
 
\item\label{it:psi-1} \(\psi\) is \(C \varepsilon^{-1} \)-Lipschitz and \(\phi\) is \(C \varepsilon\)-Lipschitz on every \(\sigma\in \mathcal{F}\);

\item\label{it:fr} \(\phi\circ \psi\) is \(C\)-Lipschitz and \(d\bigl(x, \phi(\psi(x))\bigr)\leq C\varepsilon\) for all \(x\in X\).
\end{enumerate}
The constant \(C\) depends only on the data of \(X\) and \(Y\).
\end{proposition}
Notice that in the case \(n=0\), we do not impose any  assumption on the Lipschitz connectedness of \(Y\). If \(X\subset Y\) is a closed subset such that \(X\) is Lipschitz \(k\)-connected for some \(k\leq n\), then we can choose \(\phi\) such that, in addition, \(\phi(\Sigma^{(k+1)})\subset X\). 
The proof of Proposition~\ref{prop:factorization2} is similar to the proof of the factorization theorem in Section~\ref{sec:factorization}.

\begin{proof}[Proof of Proposition~\ref{prop:factorization2}]
Let \(c\) denote the Nagata constant of X. Fix \(\varepsilon>0\) and let \(\mathcal{B}=(B_i)_{i\in I}\) be a \(c\varepsilon\)-bounded covering of \(X\) with \(\varepsilon\)-multiplicity \(n+1\).
We begin by constructing a map $\psi\colon X\to\Sigma(I)$.
For each \(i\in I\) we define \(\tau_i\colon X\to \R
\) via \(x\mapsto \max\bigl\{\frac{\varepsilon}{2}-d(x, B_i),0\bigr\}.\)
If \(\tau_i(x)>0\), then there is \(b_i\in B_i\) with \(d(x, b_i)\leq \frac{\varepsilon}{2}\); consequently, for any \(x\in X\) there are at most \((n+1)\) indices \(i\in I\) such that \(\tau_i(x)>0\), as \(\mathcal{B}\) has \(\varepsilon\)-multiplicity \(n+1\). We set \(\bar{\tau}(x)=\sum_{i\in I} \tau_i(x)\). Clearly, \(\bar{\tau}(x)>0\) for every \(x\in X\). By the above,  \(\psi\colon X\to \Sigma(I)\) given by  \(x\mapsto \bar{\tau}(x)^{-1}\bigl(\tau_i(x)\bigr)_{i\in I}\) is well-defined and \(\Sigma\coloneqq \Hull\bigl(\psi(X)\bigr)\) is an \(n\)-dimensional simplicial complex. In what follows, we equip \(\Sigma\) with the \(\ell_2\)-metric and view \(\psi\) as a map from \(X\) to \(\Sigma\). One calculates exactly as in the proof of Theorem~\ref{thm:factorization} that for all \(x\), \(x'\in X\), 
\[
\abs{\psi(x)-\psi(x')}\leq \frac{4(n+1)}{\bar{\tau}(x)} d(x,x').
\]
Since \(\mathcal{B}\) is a covering of \(X\), it follows that \(\bar{\tau}(x)\geq \frac{\varepsilon}{2}\) for all \(x\in X\), so by the estimate above, \(\psi\) is \(C_0 \varepsilon^{-1}\)-Lipschitz, as desired.

Next, we construct \(\phi\colon \Sigma\to Y\). As \(Y\) is Lipschitz \((n-1)\)-connected, the same argument as in the proof of Theorem~\ref{thm:factorization} yields a map \(\phi\colon \Sigma\to Y\) such that \(h(e_i)\in B_i\) and \(\Lip( \phi|_\sigma)\leq C_1 \Lip(\phi|_{\sigma^{(0)}})\) for every \(\sigma\in \mathcal{F}(\Sigma)\).
Here, \(C_1\) denotes a constant depending only the data of \(X\) and \(Y\). Fix \(\sigma\in \mathcal{F}\) and let \(e_i\), \(e_j\) be two vertices of \(\sigma\). Clearly, there is \(x\in \psi^{-1}(\st \sigma)\) such that \(\tau_i(x)>0\) and \(\tau_j(x)>0\).  We obtain
\[
d(\phi(e_i), \phi(e_j))\leq d(\phi(e_i), x)+d(x, \phi(e_j))\leq 2\Big( c\varepsilon+\frac{\varepsilon}{2}\Big),
\]
since \(\mathcal{B}\) is \(c\varepsilon\)-bounded. Hence, by the above, \(\phi\) is \(C_2 \varepsilon\)-Lipschitz on \(\sigma\), where \(C_2\coloneqq \sqrt{2}(c+1)C_1\). We proceed by showing \eqref{it:fr}. Let \(\delta\coloneqq \phi \circ \psi\). Fix \(x\in X\) and let \(\sigma\in \mathcal{F}\) be the unique simplex such that \(\psi(x)\in \interior \sigma\). Since there exists some \(i\in I\) with \(\tau_i(x)>0\), we find that
\begin{equation}\label{eq:f-dist-to-id}
d(\delta(x), x)\leq d(\delta(x), \phi(e_i))+d(\phi(e_i), x) \leq C_3\varepsilon,
\end{equation}
as \(\phi\) is \(C_2 \varepsilon\)-Lipschitz on \(\sigma\) and \(d(\phi(e_i), x)\leq c\varepsilon+\frac{\varepsilon}{2}\). 
In particular, if \(d(x,x')\geq \frac{\varepsilon}{8}\), then \eqref{eq:f-dist-to-id} yields \(d(\delta(x), \delta(x'))\leq C_4 d(x,x')\), where \(C_4\coloneqq 1+ 16 C_3\). 
Now, suppose that \(x\), \(x'\in X\) satisfy \(d(x,x')\leq \frac{\varepsilon}{8}\), and let \(\sigma\), \(\sigma'\in \mathcal{F}\) denote the unique simplices such that \(\psi(x)\in \interior \sigma\) and \(\psi(x')\in \interior \sigma'\), respectively. If \(x\in B_i\) for some \(i\in I\), then \(\tau_i(x')>0\); thus,
\(\sigma \cap \sigma' \neq \varnothing\).
There is \(p\in \sigma \cap \sigma^\prime\) such that
\(\abs{\psi(x)-p}+\abs{p-\psi(x')}\leq 2n \abs{\psi(x)-\psi(x')}\), and we can estimate
\[
d(\delta(x), \delta(x'))\leq d(\delta(x), \phi(p))+d(\phi(p), \delta(x')) \leq C_2 \varepsilon\,\Big( \abs{\psi(x)-p}+\abs{p-\psi(x')}\Big). 
\]
Hence, \(d(\delta(x), \delta(x'))\leq 2C_0 C_2 n\, d(x,x')\), for \(\psi\) is \(C_0\varepsilon^{-1}\)-Lipschitz. By setting, \(C_5\coloneqq 2 C_0 C_2 C_4 n\), we conclude that \(\delta=\phi\circ \psi\) is \(C_5\)-Lipschitz, as desired. We put \(C\coloneqq \max\bigl\{ C_i : i=0, \dots, 5\bigr\}\), which does not depend on \(\varepsilon\). 
This finishes the proof of Proposition~\ref{prop:factorization2}.
\end{proof}

We conclude this section with the proof of Theorem~\ref{thm:factorization-side-length-r}.

\begin{proof}[Proof of Theorem~\ref{thm:factorization-side-length-r}]
By Proposition~\ref{prop:factorization2} there is \(C>0\) such that for every \(\varepsilon>0\) there exist a simplicial complex \(\Sigma\) of dimension \(\leq n\) equipped with the \(\ell_2\)-metric and maps \(\psi\colon X\to \Sigma\) and \(\phi\colon \Sigma\to Y\) such that statements \eqref{it:containment} -- \eqref{it:fr} of Proposition~\ref{prop:factorization2} hold. Fix \(\varepsilon >0\). Notice that $\psi(X)$ is rectifiably connected and every $z\in\Sigma$ lies in a simplex of $\Sigma$ which intersects $\psi(X)$, so $\Sigma$ is also rectifiably connected. Let $d_\varepsilon$ be the length metric on $\Sigma$, scaled by the factor $\varepsilon\cdot 2^{-1/2}$. Then $d_\varepsilon$ is a length metric and $(\Sigma, d_\varepsilon)$ satisfies (1). In the following, we show that $\psi\colon X\to(\Sigma, d_\varepsilon)$ and $\varphi\colon (\Sigma, d_\varepsilon)\to Y$ have the desired properties.

Let $x,x'\in X$ and let $\alpha$ be a curve in $X$ from $x$ to $x'$ satisfying $\ell(\alpha)\leq c d(x,x')$, where $c$ is the quasiconvexity constant of $X$. Then 
 $$
 \sqrt{2} \, d_\varepsilon(\psi(x),\psi(x'))\leq \varepsilon \,\ell(\psi\circ \alpha)\leq C \ell(\alpha)\leq Cc \,d(x,x').
 $$
 This shows that $\psi\colon X\to(\Sigma, d_\varepsilon)$ is $2^{-1/2}\cdot Cc$-Lipschitz. Next, let $z,z'\in\Sigma$. It follows from Proposition~\ref{prop:h-general-little-lip} and Lemma~\ref{lem:littleLip} that for every curve $\gamma$ in $\Sigma$ connecting $z$ and $z'$ we have $d(\phi(z),\phi(z')) \leq C\varepsilon \ell(\gamma)$. By taking the infimum over all such curves $\gamma$, we conclude
 $$
 d(\phi(z),\phi(z'))\leq  C\hspace{-0.2em}\sqrt{2} \,d_\varepsilon(z,z').
 $$ 
 In particular, the map $\varphi\colon (\Sigma, d_\varepsilon)\to Y$ is $C\hspace{-0.2em}\sqrt{2}$-Lipschitz and it is clear that $\varphi$ and $\psi$ satisfy condition (2) and the remaining statements in the theorem. 
\end{proof}

\section{Deformation theorem in spaces of finite Nagata dimension}\label{sec:def-thm-finite-Nagata}

The aim of this section is to prove Theorem~\ref{thm:Def-thm-finite-Nagata-dim}. The proof will rely on Theorem~\ref{thm:factorization-side-length-r} from the introduction and the deformation theorem for bilipschitz triangulated metric spaces (see Theorem~\ref{thm:Federer-Fleming}). We follow the proof strategy outlined in Section~\ref{sec:proof-outline}.

Let \(X\) be as in Theorem~\ref{thm:Def-thm-finite-Nagata-dim}. We define $Y\coloneqq \ell_\infty(X)$ and view $X$ as a subset of $Y$. By Theorem~\ref{thm:factorization-side-length-r} there is a constant \(C>0\) depending only on the data of $X$ such that for every  $\varepsilon>0$ there exist a metric simplicial complex $\Sigma$ and \(C\)-Lipschitz maps $\psi\colon X\to \Sigma$ and $\varphi\colon \Sigma\to Y$ such that \eqref{it:thm-1.6-1} and \eqref{it:thm-1.6-2} of Theorem~\ref{thm:factorization-side-length-r} hold. Fix \(\varepsilon>0\). Let $\Lambda_m$ and $\Gamma_m$ be the homomorphisms from Proposition~\ref{prop:homomorphisms-transporting} when applied to \(\phi\). These homomorphisms,  satisfy $$\mass(\Lambda_m(\llbracket \sigma\rrbracket))\leq C\varepsilon^m,\quad\quad \mass(\Gamma_m(\llbracket\sigma\rrbracket))\leq C\varepsilon^{m+1}$$ and $\spt\Lambda_m(\llbracket \sigma\rrbracket)$, $\spt\Gamma_m(\llbracket\sigma\rrbracket)\subset B\bigl(\varphi(\sigma^{(0)}),  C\varepsilon\bigr)$ for every $m=0,\dots, k+1$ and each $m$-simplex $\sigma$ in $\Sigma$. The proof of Theorem~\ref{thm:Def-thm-finite-Nagata-dim} relies on the following two auxiliary lemmas.

\begin{lemma}\label{lem:spt-phi-hull-Lambda-Gamma}
 For every subset $A\subset X$ we have $\varphi\bigl(\Hull(\psi(A))\bigr)\subset B(A, C\varepsilon)$. Moreover, if $Q\in\poly_m(\Sigma)$ has support in $\Hull(\psi(A))$, then 
 $$
 \spt(\Lambda_m(Q)),\, \spt(\Gamma_m(Q))\subset B(A, C \varepsilon),
 $$
 where $C$ is a constant only depending on the data of $X$.
\end{lemma}

\begin{proof}
 Let $A\subset X$ be a nonempty set and let $z\in\Hull(\psi(A))$. Then there exists $x\in A$ such that $\psi(x)$ and $z$ lie in  a common simplex in $\Hull(\psi(A))$. In particular, we have $d(z,\psi(x))\leq \varepsilon$ and hence $$d(\varphi(z),x)\leq d(\varphi(z), \varphi(\psi(x)))+ d(\varphi(\psi(x)),x)\leq 2 C \varepsilon.$$ This implies that $\varphi\bigl(\Hull(\psi(A))\bigr)\subset B(A,2C\varepsilon)$, as claimed. In order to prove the second statement, let $\sigma\subset \Hull(\psi(A))$ be an $m$-simplex.
Since $\varphi(\sigma^{(0)})\subset B(A,2C\varepsilon )$ it follows from Proposition~\ref{prop:homomorphisms-transporting} that 
$$
\spt(\Lambda_m(\llbracket \sigma\rrbracket)),\, \spt(\Gamma_m(\llbracket \sigma\rrbracket))\subset B(A,3 C\varepsilon).
$$ 
This implies the second statement.
\end{proof}

\begin{lemma}
 The triple $(\Sigma, \varphi|_{\Sigma^{(0)}}, \Lambda_*)$, where \(\Lambda_\ast\coloneqq\{\Lambda_0, \Lambda_1, \dots, \Lambda_k\}\), defines a $(k,\varepsilon)$-polyhedral structure with a constant only depending on the data of $X$.
\end{lemma}

\begin{proof}
We only need to show that \(\varphi|_{\Sigma^{(0)}}\) has the desired properties. Let $z,w\in \Sigma^{(0)}$. Then there exist $x,y\in X$ such that $z$ and $\psi(x)$ lie in a common simplex and $w$ and $\psi(y)$ lie in a common simplex. In particular, we have 
 \begin{equation*}
     \begin{split}
         d(\varphi(z), \varphi(w))&\geq d(\varphi(\psi(x)),\varphi(\psi(y))) - d(\varphi(z), \varphi(\psi(x))) - d(\varphi(w),\varphi(\psi(y)))\\
         &\geq d(x,y) - 4 C \varepsilon.
     \end{split}
 \end{equation*}
Moreover, for every $x\in X$ we have $$d(x,\varphi(\Sigma^{(0)}))\leq C\varepsilon+ d(\varphi(\psi(x)), \varphi(\Sigma^{(0)})) \leq  C\varepsilon + Cd(\psi(x),\Sigma^{(0)})\leq 2 C \varepsilon.$$ 
This completes the proof. 
\end{proof}

Now we are in a position to prove Theorem~\ref{thm:Def-thm-finite-Nagata-dim}.

\begin{proof}[Proof of Theorem~\ref{thm:Def-thm-finite-Nagata-dim}] In what follows, $C$ and $C'$ will denote constants depending only on $k$ and the data of $X$ and they may change from one appearance to another.
Let $1\leq m\leq k$ and let $T\in\bI_m(X)$. Define $T'\coloneqq \psi_\#T$ and write $T' = P'+R'+\partial S'$, where $P'$, $R'$, $S'$ are as in Theorem~\ref{thm:Federer-Fleming}. If $m$ is strictly larger than the dimension of $\Sigma$ then $T'=0$ by the comment after Theorem~\ref{thm:Federer-Fleming} and in this case we let $P'$, $R'$, and $S'$ be the zero currents. Define $P\coloneqq \Lambda_m(P')$. By Proposition~\ref{prop:homomorphisms-transporting} and Theorem~\ref{thm:Federer-Fleming} 
we have that 
\begin{align*}
    &\mass(P) \leq C \mass(T), && \mass(\partial P)\leq C\mass(\partial T).
\end{align*}
Moreover, $\spt P'\subset\Hull(\psi(\spt T))$ and $\spt \partial P'\subset\Hull(\psi(\spt \partial T))$ and thus Lemma~\ref{lem:spt-phi-hull-Lambda-Gamma} implies that 
\begin{align*}
    &\spt P\subset B(\spt T, C \varepsilon),&& \spt \partial P \subset B(\spt \partial T, C \varepsilon).
\end{align*}
Now, let $H\colon [0,1]\times X\to Y$ be the straight line homotopy 
$$
H(t,x)\coloneqq (1-t)x+t\varphi(\psi(x)).
$$ The integral currents in $Y$ defined by $U\coloneqq H_\#(\bb{0,1}\times \partial T)$ and $V\coloneqq H_\#(\bb{0,1}\times T)$ satisfy $$T-\varphi_\#T' = - U - \partial V$$ and, by Lemma~\ref{lem:mass-of-pushforward-of-product}, $\mass(U)\leq \varepsilon C\mass(\partial T)$ and $\mass(V)\leq \varepsilon C\mass(T)$. Moreover, we have $$\spt U\subset H([0,1]\times\spt \partial T)\subset B(\spt \partial T, C \varepsilon)$$ and similarly $\spt V\subset H([0,1]\times \spt T)\subset B(\spt T, C \varepsilon)$. Proposition~\ref{prop:homomorphisms-transporting} yields $$\varphi_\#P' = P - \partial\Gamma_m(P') - \Gamma_{m-1}(\partial P')$$ and hence 
\begin{equation*}
        T = T-\varphi_{\#}T' + \varphi_{\#}T'
        = -U -\partial V + \varphi_{\#}P' + \varphi_{\#}R' + \partial \varphi_{\#}S'
        = P + \hat{R} + \partial \hat{S},
\end{equation*}
where we have set $\hat{R}\coloneqq \varphi_\#R' - U - \Gamma_{m-1}(\partial P')$ and $\hat{S}\coloneqq \varphi_\#S' - V - \Gamma_m(P')$. A direct calculation shows that 
\begin{align*}
&\mass(\hat{R})\leq \varepsilon \,C\mass(\partial T),&& \mass(\hat{S})\leq \varepsilon \,C\mass(T).    
\end{align*}
Moreover, Lemma~\ref{lem:spt-phi-hull-Lambda-Gamma} implies that \(\spt \hat{R}\subset B(\spt \partial T, C \varepsilon)\) and \(\spt \hat{S}\subset B(\spt T, C \varepsilon)\). 

Since $\partial \hat{R} = \partial (T-P)$ is supported in $X$,  Theorem~\ref{thm:undistorted-unbounded-set} implies the following two statements. There exists $R\in\bI_m(X)$ with $\partial R = \partial \hat{R}$ and such that $\mass(R)\leq C\mass(\hat{R})\leq \varepsilon \,C'\mass(\partial T)$ as well as $\spt R\subset B(\spt \partial T, C \varepsilon)$. Moreover, there is $\bar{S}\in\bI_{m+1}(Y)$ satisfying $\partial\bar{S} = \hat{R}-R$ and 
$$
\mass(\bar{S})\leq \varepsilon C\mass(\hat{R})\leq \varepsilon^2C'\mass(\partial T)
$$ 
as well as 
$\spt\bar{S}\subset B\bigl(\spt\partial T, C'\varepsilon\bigr)$.
By construction, $T= P + R+ \partial(\bar{S} + \hat{S})$ and hence $\partial(\bar{S}+\hat{S})$ is supported in $X$ and satisfies 
$$
\spt (\bar{S}+\hat{S})\subset B\bigl(\spt T, C\varepsilon\bigr).
$$ 
Thus, by Theorem~\ref{thm:undistorted-unbounded-set}, there exists $S\in\bI_{m+1}(X)$ with $\partial S = \partial(\bar{S}+\hat{S})$ and 
$$
\mass(S)\leq C\mass(\bar{S}+\hat{S})\leq  \varepsilon C'\mass(T) + \varepsilon^2C'\mass(\partial T)
$$
as well as 
$$
\spt S\subset B\bigl(\spt T, C'\varepsilon\bigr).
$$ 
Since $T= P+\hat{R}+\partial \hat{S} = P+ R + \partial \bar{S} + \partial \hat{S} = P + R+ \partial S$ we have a decomposition as in the statement of the theorem. 
\end{proof}

Theorem~\ref{thm:Def-thm-finite-Nagata-dim} implies the following approximation result mentioned in the introduction. 

\begin{corollary}\label{cor:flat-approximation}
 Suppose $X$ is a complete metric space of finite Nagata dimension which has $(\LC_k)$ for some $k\geq 1$, and let $T\in\bI_k(X)$. Then there exists a sequence of Lipschitz $k$-chains in $X$ such that the induced integral currents $T_i$ satisfy $\flatnorm_X(T-T_i) \to 0$ as $i\to\infty$ and 
 $$
\sup_{i}\,\Bigl[\mass(T_i) + \mass(\partial T_i)\Bigr]<\infty.
 $$
 If $T$ is a cycle, then the $T_i$ are cycles and $\Fillvol_X(T-T_i)\to 0$ as $i\to \infty$.
\end{corollary}

In the above, $\flatnorm_X(T)$ is the flat norm in $X$ of a current $T\in\bI_k(X)$ and is defined by $$\flatnorm_X(T)\coloneqq \inf\{\mass(U)+\mass(V): U\in\bI_k(X), V\in\bI_{k+1}(X), T=U+\partial V\}.$$  In any complete metric space, convergence with respect to the flat norm implies weak convergence. Moreover, in metric spaces with local coning inequalities the two notions of convergence are equivalent for sequences with uniformly bounded mass and boundary mass (see \cite{MR2284563} for more information).

\begin{proof}[Proof of Corollary~\ref{cor:flat-approximation}]
 Since $X$ has $(\LC_k)$ the map $\varphi$ constructed above satisfies $\varphi(\Sigma^{(k+1)})\subset X$; see Theorem~\ref{thm:factorization-side-length-r}. Moreover, the homomorphisms $\Lambda_m$ in the polyhedral structure $(\Sigma, \varphi|_{\Sigma^{(0)}}, \Lambda_*)$ can be chosen such that $\Lambda(\llbracket\sigma\rrbracket)= \varphi_\#\llbracket \sigma\rrbracket$ for every $m$-simplex $\sigma$. See the begining of the proof of Proposition~\ref{prop:homomorphisms-transporting}. It thus follows that $\poly_k(X) = \{\Lambda_k(Q): Q\in\poly_k(\Sigma)\}$ consists of integral currents induced by Lipschitz chains in $X$. Now, the corollary follows from Theorem~\ref{thm:Def-thm-finite-Nagata-dim}.
\end{proof}

We finally turn to the proof of Corollary~\ref{cor:strong-approximation}. Given \(k\geq 0\) we denote the space of Lipschitz \(k\)-chains in $X$ by \(C_k^{\Lip}(X)\). Precisely, \(C_k^{\Lip}(X)\) is the free abelian group generated by all Lipschitz maps \(\phi\colon\Delta^k \to X\), where \(\Delta^k\) is the Euclidean standard $k$-simplex. We view $C_k^{\Lip}(X)$ as a subcomplex of the singular chain complex of $X$ and equip it with the singular boundary operator $\partial$ (see, for instance, \cite[Ch.\ 2.1]{HatcherAT}).

As already mentioned in Section~\ref{sec:currents}, Lipschitz chains induce integral currents. Indeed, if \(\alpha=\sum_{i=1}^N \theta_i \phi_i\) is a Lipschitz \(k\)-chain in $X$ then
\[
\bb{\alpha}\coloneqq \sum_{i=1}^N \theta_i\, \phi_{i\#} \bb{\Delta^k}
\]
is a \(k\)-dimensional integral current in \(X\). Moreover, one has that
$\partial \bb{\alpha}=\bb{\partial \alpha}$
for all \(\alpha\in C_{k}^{\Lip}(X)\). This can readily be verified by combining Stokes' theorem for chains with \cite[Theorem 11.1]{MR1794185}.
The following lemma proves Corollary~\ref{cor:strong-approximation} in the special case when \(\partial T\) is induced by a Lipschitz cycle.

\begin{lemma}\label{lem:moreover}
Let \(X\) be a complete metric space of finite Nagata dimension and let \(T\in \bI_{k}(X)\) for some \(k\geq 1\). Suppose that there is a cycle \(\delta\in C_{k-1}^{\Lip}(X)\) such that \(\partial T=\bb{\delta}\). If \(X\) has \((\LC_{k-1})\) then for every \(\varepsilon>0\) there exists \(\rho\in C_{k}^{\Lip}(X)\) such that
\[
\partial \rho=\delta \quad\text{ and }\quad \mass(T-\bb{\rho})\leq \varepsilon.
\]
\end{lemma}

We proceed with the proof of Corollary~\ref{cor:strong-approximation}.

\begin{proof}[Proof of Corollary~\ref{cor:strong-approximation}]
Let \(\varepsilon >0\) and \(k\geq 2\). 
By Lemma~\ref{lem:moreover}, there exists \(\rho'\in C_{k-1}^{\Lip}(X)\) with \(\partial \rho'=0\) and \(\mass(\partial T-\bb{\rho'})\leq \varepsilon\). Using Corollary~\ref{cor:lip-implies-EI} we find that \(X\) has \((\EI_{k-1})\), and thus there is \(S'\in \bI_{k}(X)\) satisfying \(\partial S'=\partial T-\bb{\rho'}\) and \(\mass(S) \leq D\, \varepsilon^{k/(k-1)}\). Let \(T'=T-S'\). Then \(\partial T'= \bb{\rho'}\), and so by Lemma~\ref{lem:moreover} there exists \(\rho\in C_{k}^{\Lip}(X)\) with \(\partial \rho=\rho'\) and \(\mass(T'-\bb{\rho})\leq \varepsilon\). Therefore,
\begin{align*}
\mass(T-\bb{\rho})&\leq \varepsilon+\mass(S')\leq 2 D \varepsilon \\
\mass(\partial T-\partial\bb{ \rho})&=\mass(\partial T-\bb{\rho'})\leq \varepsilon.
\end{align*}
This completes the proof in the case when \(k\geq 2\). If \(k=1\), then \(\partial T\in \bI_0(X)\) and thus there are finitely many \(x_i\in X\) and \(\theta_i\in \Z\) such that \(\partial T=\sum \theta_i \bb{x_i}\). In particular, \(\partial T= \bb{\delta}\) for some cycle \(\delta\in C_0^{\Lip}(X)\). Hence, by Lemma~\ref{lem:moreover}, for every \(\varepsilon >0\) there exists \(\rho\in C_1^{\Lip}(X)\) with \(\partial \rho=\delta\) and \(\mass(T-\bb{\rho})\leq \varepsilon\), as desired.
\end{proof}

Thus, we are left to prove Lemma~\ref{lem:moreover}. The main components in its proof are Theorem~\ref{thm:factorization-side-length-r} and \cite[Lemma~3]{MR3548470}. Moreover, we will need the following straightforward approximation result.

\begin{lemma}\label{lem:m-dense}
Let \(X\) be a complete metric space and let \(T\in \bI_k(X)\) for some \(k\geq 1\). If \(X\) has \((\LC_{k-1})\) then for every \(\varepsilon >0\) there exists \(\alpha\in C_{k}^{\Lip}(X)\) such that \(\mass(T-\bb{\alpha})\leq \varepsilon\). 
\end{lemma}
\begin{proof}
Let \(K\subset \R^k\) be a compact subset and \(\phi \colon K\to X\) a Lipschitz map. Due to \cite[Theorem 4.5]{MR1794185} and the inner regularity of the Lebesgue measure on \(\R^k\), to prove the lemma it suffices to show the following statement: for every \(\varepsilon >0\) there exists \(\alpha\in C_k^{\Lip}(X)\) such that \(\mass(\phi_\# \bb{K}-\bb{\alpha})\leq \varepsilon\). Fix \(\varepsilon >0\). By a straightforward construction, we find finitely many \(k\)-cubes \(Q_i\subset \R^k\) such that \(K\subset \bigcup Q_i\), \(\interior Q_i \cap \interior Q_j=\varnothing\) whenever \(i\neq j\), and \(\mathscr{H}^k\bigl(\bigcup Q_i\bigr) \leq \mathscr{H}^k(K)+\varepsilon\). 
 Clearly, there exists \(\gamma\in C_{k}^{\Lip}(\R^k)\) with \(\bb{\gamma}=\sum \bb{Q_i}\). Now, since \(X\) has \((\LC_{k-1})\) there is a \(C\)-Lipschitz map \(\hat{\phi}\colon \R^k\to X\) extending \(\phi\), and so by letting \(\alpha\coloneqq \hat{\phi}_{\#} \gamma\), we obtain
 \[
 \mass\bigl( \phi_\# \bb{K}-\bb{\alpha}\bigr)\leq C^k \mass\bigl( \bb{K}- \sum \bb{Q_i}\bigr)\leq C^k \,\varepsilon,\]
 where in the first inequality we used that \(\phi_\# \bb{K}=\hat{\phi}_{\#} \bb{K}\). 
\end{proof}

The proof of Lemma~\ref{lem:moreover}  will use the following construction. Let $\Sigma$ be a simplicial complex. By choosing a total ordering on the vertices of $\Sigma$, we can construct an isomorphism from $\poly_*(\Sigma)$ to a subcomplex of $C_*^{\Lip}(\Sigma)$. Namely, for every $k$ and every $k$-simplex $\sigma$ of $\Sigma$, there is a unique isometry $\phi_\sigma\colon \Delta^{k}\to \Sigma$ that preserves the ordering of the vertices. The ordering also lets us fix an orientation for $\sigma$ so that $\bb{\sigma}=(\phi_\sigma)_\#\bb{\Delta^k}$. Let $q\colon \poly_k(\Sigma) \to C_k^{\Lip}(\Sigma)$ be defined by
\begin{equation}\label{eq:def-q}
  q\bigg(\sum_{\sigma} \theta_\sigma \bb{\sigma}\bigg) = \sum_{\sigma} \theta_\sigma \phi_\sigma.
\end{equation}
Then $q$ is a homomorphism of chain complexes and $\bb{q(P)} = P$ for any $P\in \poly_*(\Sigma)$.

\begin{proof}[Proof of Lemma~\ref{lem:moreover}]
By Theorem~\ref{thm:factorization-side-length-r} there exists \(C>0\) such that for every \(\varepsilon >0\) there are a metric simplicial complex \(\Sigma=\Sigma(\varepsilon)\) and \(C\)-Lipschitz maps \(\psi\colon X\to \Sigma\) and \(\phi\colon \Sigma^{(k)} \to X\) satisfying the properties listed in Theorem~\ref{thm:factorization-side-length-r}. 
Fix \(\varepsilon_0>0\). By Lemma~\ref{lem:m-dense} there exists \(\tilde{\alpha}\in C_{k}^{\Lip}(X)\) with \(\mass(T-\bb{\tilde{\alpha}})\leq \varepsilon_0\). Let
\(\alpha\coloneqq \partial\tilde{\alpha}-\delta\). Then \(\partial \alpha=0\), and thus by \cite[Lemma 3]{MR3548470} there is a \(c_\alpha>0\) such that for every \(\varepsilon>0\) there are \(\gamma\in C_{k}^{\Lip}(\Sigma)\) and \(\lambda\in C_{k}^{\Lip}(X)\), and a \((k-1)\)-cycle \(\alpha'\in C_{k-1}^{\Lip}(\Sigma)\), such that \(\bb{\alpha'}\in \poly_{k-1}(\Sigma)\) and 
\begin{align*}
\partial \gamma&=\psi_\#(\alpha)-\alpha',& \mass(\bb{\gamma})&\leq c_\alpha \varepsilon,    \\
\partial \lambda&=\alpha-\phi_\#(\alpha'),&\mass(\bb{\lambda})&\leq c_\alpha\varepsilon.
\end{align*}
After fixing a total ordering on the vertices of $\Sigma$, we may suppose that $\alpha'=q(\bb{\alpha'})$, where $q$ is as in \eqref{eq:def-q}.

 Fix \(\varepsilon>0\) small enough such that \(c_\alpha \varepsilon \leq \varepsilon_0\). The boundary of \(T'\coloneqq \psi_\#\bigl(\bb{\tilde{\alpha}}-T\bigr)-\bb{\gamma}\) is equal to \(\bb{\alpha'}\) and one has \(\mass(T')\leq C^k \varepsilon_0+\varepsilon_0\). Now, on account of Theorem~\ref{thm:Federer-Fleming}, there exist \(P'\in \poly_k(\Sigma)\) and \(S'\in \bI_{k+1}(\Sigma)\) with \(T'=P'+\partial S'\) and such that \(\mass(P')\leq K \mass(T')\), where \(K\) is a constant depending only on the Nagata dimension of \(X\). Let $\beta\coloneqq q(P')\in C_{k}^{\Lip}(\Sigma)$ and \(\rho\coloneqq \tilde{\alpha}-\lambda-\phi_\#(\beta)\). Then $\partial \beta = q(\partial P') = q(\bb{\alpha'}) =\alpha'$, and 
 so \(\partial \rho=\delta\) and 
    \[
    \mass(T-\bb{\rho})\leq \varepsilon_0+c_\alpha \varepsilon+C^k\mass(P')\leq 2\varepsilon_0+(C^k K)\cdot \mass(T')\leq \bar{C} \varepsilon_0,
    \]
    where \(\bar{C}\) is a constant which is independent of \(\varepsilon_0\). 
    \end{proof}

\appendix

\section{Deformation theorem in bilipschitz triangulated metric spaces}\label{sec:appendix}

The aim of this appendix is to provide a detailed proof of the classical Federer-Fleming deformation theorem in the setting of metric spaces admitting a bilipschitz triangulation in the sense below. In fact, we establish a version of the deformation theorem which moreover includes local mass estimates. Such estimates are of crucial importance in the proofs of Theorems~\ref{thm:intro-undistorted-main} and \ref{thm:Def-thm-finite-Nagata-dim}.

\begin{definition}\label{def:triangulation} A metric space $X$ is said to admit an \((n,D, \varepsilon)\)-triangulation, where $D\geq 1$, \(n\in \N\), \(\varepsilon >0\), if there exist an \(n\)-dimensional simplicial complex \(\Sigma\) and a homeomorphism \(p\colon \Sigma\to X\) which is \(D\)-bilipschitz on every \(\sigma\in \mathcal{F}\) if \(\Sigma\) is equipped with the scaled \(\ell_2\)-metric such that each simplex has side length \(\varepsilon\).
\end{definition}

Notice that in a path-connected simplicial complex the $\ell_2$-metric and the length metric induce the same topology and they agree on each simplex. Thus, for path connected spaces $X$ the definition above could equivalently be stated with the scaled $\ell_2$-metric replaced by the scaled length metric on $\Sigma$. If $X$ is quasiconvex then the homeomorphism $p\colon \Sigma\to X$ in the definition above is a bilipschitz homeomorphism when $\Sigma$ is equipped with the scaled length metric; see Corollary~\ref{cor:triangulation-bilip-length}.

We will use the following notation for simplices in \((n,D, \varepsilon)\)-triangulated spaces. A subset \(\sigma\subset X\) of a metric space \(X\) admitting a triangulation \(p\colon \Sigma \to X\) in the sense above is called \(k\)-simplex in \(X\) if \(p^{-1}(\sigma)\subset \Sigma\) is a \(k\)-simplex. We denote by $\mathcal{F}$ the family of simplices in $X$ and by $\mathcal{F}_k$ the collection of $k$-simplices in $X$. For every \(\sigma\in \mathcal{F}\) we write \(\interior \sigma\coloneqq p\bigl(\interior p^{-1}(\sigma)\bigr)\) to denote the relative interior of \(\sigma\) in \(X\). 
Given \(\sigma\in\mathcal{F}\), the set
\[
\st \sigma\coloneqq \bigcup \Bigl\{ \interior \tau : \tau\in \mathcal{F} \text{ and } \sigma \subset \tau \Bigr\}
\]
is called the \textit{open star} of \(\sigma\). Further, the \textit{hull} of a subset $B\subset X$ is the union of all simplices $\sigma\in \mathcal{F}$ such that $\interior \sigma\cap B\not=\varnothing$ and is denoted $\Hull(B)$. For $k\geq 0$ we denote by $\poly_k(X)$ the abelian subgroup of $\bI_k(X)$ generated by currents of the form $p_\#\llbracket\sigma\rrbracket$, where $\sigma\subset\Sigma$ is a $k$-simplex.

\begin{theorem}\label{thm:Federer-Fleming} 
Let \(X\) be a \(c\)-quasiconvex \((n,D, \varepsilon)\)-triangulated metric space. Then there is \(C= C(c, n, D)>0\) such that for every $k=1, \dots, n$ and every \(T\in \bI_k(X)\) there exist \(P\in\poly_k(X)\), \(R\in\bI_k(X)\), and \(S\in\bI_{k+1}(X)\) with $T=P+R+\partial S$ and
\begin{alignat}{2}\label{eq:massEstimates}
 \mass(P) &\leq C \mass(T), &\quad \mass(\partial P) &\leq C \mass(\partial T),  \\	
 \mass(S)					&\leq \varepsilon \, C \mass(T), 								&\quad \mass(R) &\leq \varepsilon\, C \mass(\partial T), \nonumber
\end{alignat}
as well as $\spt R, \, \spt \partial P\subset \Hull(\spt\partial T)$ and $\spt P,\, \spt S\subset \Hull(\spt T)$. Moreover, 
\begin{alignat}{2}\label{eq:MassOfFace}
 \lVert P \rVert(\interior \sigma) &\leq C\,\lVert T \rVert(\st \sigma), &\quad \lVert \partial P\rVert(\interior \sigma) &\leq C\, \norm{\partial T}(\st \sigma),  \\	
 \lVert S\rVert(\interior \sigma)					&\leq \varepsilon \, C\, \norm{T}(\st \sigma), 								&\quad \norm{R}(\interior \sigma) &\leq \varepsilon \, C\,\norm{\partial T}(\st \sigma) \nonumber
\end{alignat}
for all \(\sigma\in \mathcal{F}\). 
\end{theorem}
If $X$ is as in the theorem and $T\in\bI_k(X)$ for some $k>n$ then $T=0$; see Remark~\ref{rem:zero-current-above-dim}. Hence, the statement of the theorem holds also in this case with $P$, $R$, $S$ equal to zero.

If \(M\) is a smooth Riemannian manifold of dimension \(n\) admitting a geometric group action by isometries, then \(M\) admits an \((n, D, \varepsilon)\)-triangulation for every \(\varepsilon>0\), where \(D\geq 1\) depends only on \(M\). For such manifolds a variant of Theorem~\ref{thm:Federer-Fleming} without the local mass estimates was established in \cite[Theorem 10.3.3]{MR1161694}. 
We furthermore remark that
any \(n\)-dimensional smooth Riemannian manifold of bounded geometry admits an \((n, D, \varepsilon)\)-triangulation for every \(\varepsilon>0\), for some \(D\geq 1\) depending only on \(n\) and the parameters of bounded geometry (see \cite{Bowditch-bilipschitz-2020} and the references therein).

We prove Theorem~\ref{thm:Federer-Fleming} by analyzing radial projections from simplices of $X$ to their boundaries. 
Let \(\Delta^m\) be the Euclidean standard \(m\)-simplex and denote by $o\in \interior{\Delta^m}$ and $r_m>0$ the incenter and inradius of $\Delta^m$, respectively. There exists $K>0$ only depending on $m$ with the following property. If $b\in \Delta^m\cap U(o,r_m/2)$ then the radial projection $\rho_b\colon \Delta^m\setminus\{b\}\to\partial \Delta^m$ with center \(b\) is $Kr^{-1}$-Lipschitz on $\Delta^m\setminus U(b,r)$ for every $r>0$. The radial projection $\rho_b$ is homotopic to the identity map on $\Delta^m$ by a straight-line homotopy $h$. Like the radial projection itself, the Lipschitz constant of $h$ depends on the distance from $b$, but Proposition~\ref{prop:HomotopyFormula-new} below gives conditions under which we can define the pushforward of $T$ under $\rho_b$ and the pushforward of $[0,1]\times T$ under $h$.

\subsection{Limit homotopy formula}\label{sec:limit-homotopy}
Let $X$ be a complete metric space and let $L,\varepsilon>0$. Suppose that $u\colon X\to[0,1]$ is $1$-Lipschitz and $h\colon [0,1]\times \{u\not=0\}\to X$ is a continuous map which is Lipschitz on $[0,1]\times \{u>r\}$ for every $r>0$. Suppose furthermore that $\lip h_t(x)\leq Lu(x)^{-1}$ and $\lip h_x(t)\leq L\varepsilon$ for all $t\in[0,1]$ and all $x\in \{u\not=0\}$, where we have abbreviated $h_t=h(t,\cdot)$ and $h_x=h(\cdot, x)$.

Let $k\geq 1$ and $T\in\bI_k(X)$. Then $T\on\{u>r\}$ and $(\partial T)\on\{u>r\}$ are integral currents for almost every $r>0$ and thus $T^i_r\coloneqq h_{i\,\#}(T\on\{u>r\})$ for $i=0,1$ and 
$$
S_r\coloneqq h_\#\bigl(\bb{0,1}\times (T\on\{u>r\})\bigr),\quad R_r\coloneqq h_\#\bigl(\bb{0,1}\times ((\partial T)\on\{u>r\})\bigr)
$$ 
are also integral currents. We will call such $r>0$ admissible.

\begin{proposition}\label{prop:HomotopyFormula-new}
 Suppose $\bigl\{B_l\bigr\}_{l\in\N}$ is a Borel partition of $X$ (for example, the partition of a countable simplicial complex into the interiors of its faces) such that for every \(l\geq 1\),
\[
\int_{B_l} u(x)^{-k} \, d\lVert T \rVert (x) \leq L \,\lVert T\rVert(B_l), \quad \quad 
\int_{B_l} u(x)^{1-k} \, d\lVert \partial T \rVert (x) \leq L \,\lVert\partial T\rVert(B_l).
\]
Then there exist admissible $r_m\searrow 0$ such that $T^i_{r_m}\to T^i$, $R_{r_m} \to R$, $S_{r_m}\to S$ for some $T^0,T^1,R\in \bI_k(X)$ and  $S\in\bI_{k+1}(X)$ with
\begin{equation}\label{eq:homotopyFormula}
\partial S+R=T^1-T^0,
\end{equation}
and the following property holds for some $C=C(L,k)$. If \(B\subset X\) is Borel then
\begin{align}\label{eq:localMassEstimate-1}
\lVert T^i\rVert(B) &\leq C \lVert T\rVert (A), &\lVert\partial T^i\rVert(B) \leq C \lVert\partial T\rVert (A),
\end{align}
where \(A\subset X\) is the union of all \(B_l\) such that \(B_l\cap h_i^{-1}(B)\neq \varnothing\), and 
\begin{align}\label{eq:localMassEstimate-2}
\lVert S \rVert (B) &\leq \varepsilon \, C \lVert T\rVert (A'), &\lVert R\rVert (B) \leq \varepsilon \, C \lVert\partial T\rVert (A'),
\end{align}
where \(A'\subset X\) is the union of all \(B_l\) such that \(B_l \cap h_t^{-1}(B)\neq \varnothing\) for some \(t\in [0,1]\).
\end{proposition}

\begin{proof}
 By the slicing theorem we have for every $r>0$ that 
 \begin{equation*}
  \frac{1}{r}\int_r^{2r} t^{1-k}\mass(\langle T,u,t\rangle)\,dt\leq 
         r^{-k}\lVert T\rVert (\{r\leq u\leq 2r\})
         \leq 2^k\int_{\{u\leq 2r\}}u(x)^{-k}\,d\lVert T\rVert (x).
 \end{equation*}
 The last integral tends to zero with $r\to 0$, so there exist admissible $r_m\searrow 0$ such that 
 \begin{equation*}
 \lim_{m\to\infty} r_m^{1-k}\mass(\langle T,u,r_m\rangle) = 0.
 \end{equation*}
 Set $T^i_m\coloneqq T^i_{r_m}$ and observe that if $l>m$ then $T^i_l - T^i_m = h_{i\,\#}(T\on\{r_l<u\leq r_m\})$ and $$\partial T^i_l - \partial T^i_m = h_{i\,\#}((\partial T)\on\{r_l<u\leq r_m\}) + h_{i\,\#}\langle T,u,r_m\rangle - h_{i\,\#}\langle T,u,r_l\rangle,$$ so Lemma~\ref{lem:rajala-wenger} implies $$\mass(T^i_l - T^i_m)\leq L^k\int_{\{u\leq r_m\}}u(x)^{-k}\,d\lVert T\rVert(x)$$ and \begin{equation*}
     \begin{split}
         \mass(\partial T^i_l - \partial T^i_m)\leq&\, L^{k-1}\int_{\{u\leq r_m\}}u(x)^{1-k}\,d\lVert \partial T\rVert(x)\\
         & + L^{k-1}\Bigl[ r_m^{1-k}\mass(\langle T, u, r_m\rangle) + r_l^{1-k}\mass(\langle T,u,r_l\rangle)\Bigr].
     \end{split}
 \end{equation*}
The right hand side in both inequalities tends to zero as both $m,l\to\infty$, thus $(T^i_m)$ is a Cauchy sequence in $\bN_k(X)$ and hence converges to some $T^i\in\bN_k(X)$.
Next, set $S_m\coloneqq S_{r_m}$ and notice that $$S_l - S_m = h_\#\bigl(\bb{0,1}\times (T\on\{r_l<u\leq r_m\})\bigr)$$ for $l>m$ and thus, by Lemma~\ref{lem:mass-of-pushforward-of-product},
$$
\mass(S_l-S_m)\leq (k+1)L^{k+1}\varepsilon \int_{\{u\leq r_m\}} u(x)^{-k}\,d\lVert T\rVert(x).
$$ 
Since the integral on the right tends to $0$ as $m\to\infty$ it follows that $(S_m)$ is a Cauchy sequence in $\bM_{k+1}(X)$ and converges to some $S\in\bM_{k+1}(X)$. One shows analogously that $R_m\coloneqq R_{r_m}$ converges to some $R\in\bM_k(X)$ as $m\to\infty$. Since 
$$
\partial S_m = T^1_m - T^0_m - R_m + h_\#(\bb{0,1}\times \langle T,u,r_m\rangle)
$$
by the homotopy formula and since 
$$
\mass\bigl(h_\#(\bb{0,1}\times \langle T,u,r_m\rangle)\bigr)\leq k\varepsilon L^{k} r_m^{1-k}\mass(\langle T,u,r_m\rangle)\to 0
$$ 
as $m\to\infty$, it follows that $\partial S +R = T^1-T^0$. From this we conclude that $S\in\bN_{k+1}(X)$ and $R\in\bN_k(X)$. As $T_m^i$, $S_m$, $R_m$ are integral currents for all $m$ and the convergence is in mass it follows that $T^i$, $S$, $R$ are also integral currents.

Finally, if $B\subset X$ is Borel and $A\subset X$ is the union of all $B_l$ such that $B_l\cap h_i^{-1}(B)\neq\varnothing$ then, by Lemma~\ref{lem:rajala-wenger},
\begin{equation*}
    \lVert T^i_m\rVert(B)
    \leq L^k\int_{h_i^{-1}(B)}u(x)^{-k}\,d\lVert T\rVert(x)\leq L^{k+1}\lVert T\rVert(A)
\end{equation*}
and
\begin{equation*}
    \begin{split}
    \lVert \partial T^i_m\rVert(B)&\leq L^{k-1}\int_{h_i^{-1}(B)}u(x)^{1-k}\,d\lVert \partial T\rVert(x) + L^{k-1}r_m^{1-k}\mass(\langle T,u,r_m\rangle),
    \end{split}
\end{equation*}
which implies \eqref{eq:localMassEstimate-1} upon taking $m\to\infty$. The inequalities in \eqref{eq:localMassEstimate-2} are proved analogously using Lemma~\ref{lem:mass-of-pushforward-of-product}.
\end{proof}

\subsection{Special case of the deformation theorem}

We will prove the deformation theorem by successive applications of the following proposition, which we prove by applying Proposition~\ref{prop:HomotopyFormula-new} to a map from $X$ to $X$ which restricts to a radial projection on each $m$-simplex.

\begin{proposition}\label{prop:inductive-step}
Given $m>k\geq 1$ and $D,c\geq 1$ there is a positive constant $C$ with the following property. Let $X$ be a complete $c$-quasiconvex metric space admitting an $(m,D,\varepsilon)$-triangulation and $T\in\bI_k(X)$. Then there exist $P,R\in\bI_k(X)$ and $S\in\bI_{k+1}(X)$ with $\spt P\subset X^{(m-1)}$ and such that $T=P+R+\partial S$ as well as 
 \begin{alignat*}{2}\label{eq:mass-B2}
 \mass(P)&\leq C\mass(T),&\quad \mass(\partial P)&\leq C \mass(\partial T),\\
 \mass(S)&\leq \varepsilon\, C \mass(T), &\quad \mass(R) &\leq \varepsilon\, C\mass(\partial T)
 \end{alignat*}
and $\spt R, \, \spt \partial P\subset \Hull(\spt\partial T)$, and $\spt P,\, \spt S\subset \Hull(\spt T)$.
Moreover,
\begin{alignat*}{2}
\lVert P\rVert(\interior{\sigma})& \leq C\,\lVert T\rVert(\st \sigma),& \quad  \lVert \partial P\rVert (\interior{\sigma}) &\leq C\,\lVert \partial T\rVert(\st \sigma), \\ \lVert S\rVert(\interior{\sigma}) &\leq \varepsilon\, C\, \lVert T\rVert( \st \sigma),& \quad \lVert R\rVert(\interior{\sigma}) &\leq \varepsilon \, C \,\lVert \partial T\rVert(\st \sigma),
\end{alignat*}
for every \(\sigma\in \mathcal{F}\).
\end{proposition}

 Let $o\in \interior{\Delta^m}$ and $r_m>0$ be the incenter and inradius of $\Delta^m$ as above. As in \cite{MR123260}, one can establish that there exists $K>0$ only depending on $m$ such that if $\mu_1,\mu_2$ are finite Borel measures on $\Delta^m$ then there exists $b\in \Delta^m\cap U(o,r_m/2)$ such that for $i=1,2$, $$\int_{\interior{\Delta^m}}|y-b|^{1-m}\,d\mu_i(y)\leq K\mu_i(\interior{\Delta^m}).$$ 
 \begin{proof}[Proof of Proposition~\ref{prop:inductive-step}]
 Denote the metric on $X$ by $d$. By possibly considering $X$ equipped with the metric $\varepsilon^{-1}d$ we may assume that $\varepsilon=1$. Let $q\colon X\to \Sigma$ be the inverse of an $(m,D,1)$-triangulation of $X$, and let $T\in\bI_k(X)$. By the discussion preceding the proof there exists, for every $\sigma\in \mathcal{F}_m$, some point $a_\sigma\in\interior{\sigma}$ such that $$\int_{\interior{\sigma}}d(x,a_\sigma)^{-k}\,d\lVert T\rVert (x)\leq KD^k\lVert T\rVert(\interior{\sigma})$$ and $$\int_{\interior{\sigma}}d(x,a_\sigma)^{1-k}\,d\lVert \partial T\rVert (x)\leq KD^k\lVert \partial T\rVert(\interior{\sigma}).$$
 Set $Z\coloneqq \{a_\sigma: \sigma\in \mathcal{F}_m\}$ and define a $1$-Lipschitz function $u\colon X\to[0,1]$ by $$u(x)\coloneqq \min\{1, d(x,Z)\}.$$ Set $r'\coloneqq \frac{r_m}{2cD}$ and observe that $u(x)\geq r'$ for every $x\in X^{(m-1)}$ and if $\sigma\in \mathcal{F}_m$ then $$\min\{d(x,a_\sigma), r'\}\leq u(x)\leq d(x,a_\sigma)$$ for every $x\in\sigma$. With the above it follows that for every \(\tau\in \mathcal{F}\) we have $$\int_{\interior{\tau}}u(x)^{-k}\,d\lVert T\rVert (x)\leq K'\lVert T\rVert(\interior{\tau}),\quad \int_{\interior{\tau}}u(x)^{1-k}\,d\lVert \partial T\rVert (x)\leq K'\lVert \partial T\rVert(\interior{\tau})$$ for some $K'=K'(D,m,c)$.
 
Let $h\colon [0,1]\times \{u\not=0\}\to X$ be the unique map such that for each \(\sigma\in \mathcal{F}_m\), 
 $$
 q(h(t,x))=(1-t)q(x) + t\rho_{q(a_\sigma)}(q(x))
 $$ 
 for all $t\in[0,1]$ and all $x\in \sigma$, where we have naturally identified $q(\sigma)$ with $\Delta^m$. We have $h(t,x) =x$ for all $t\in[0,1]$ and all $x\in X^{(m-1)}$, so this is well-defined. Then $h_0(x) = x$ and $h_1(x)\in X^{(m-1)}$ for all $x\in\{u\not=0\}$. Moreover, if $A\subset\{u\not=0\}$ then 
 $$
 h([0,1]\times A)\subset \Hull(A)
 $$ 
 and if $\sigma,\sigma'\in \mathcal{F}$ are simplices such that $h([0,1]\times \interior{\sigma'}) \cap \interior{\sigma}\not=\varnothing$ then $\sigma'=\sigma$ or $\sigma'$ is $m$-dimensional and contains $\sigma$.
 
 Let $t\in[0,1]$ and $x\in\bigl\{u\not=0\bigr\}$. Then clearly we have $\lip h_x(t)\leq D^2$.
 We next show that $\lip h_t(x)\leq K''u(x)^{-1}$ for some $K''=K''(D,m,c)$. For this, given $y\in X$ let $\sigma(y)\subset X$ be the unique simplex containing $y$ in its interior. Let $x\in\bigl\{u\neq 0\bigr\}$. For all $y\in X$ sufficiently close to $x$ we have $\sigma(x)\subset\sigma(y)$. If $\sigma(y)\subset X^{(m-1)}$ then $h_t(x)=x$ and $h_t(y)=y$ and therefore $d(h_t(x),h_t(y)) = d(x,y)\leq r'u(x)^{-1}\cdot d(x,y)$. If $\sigma(y)$ is $m$-dimensional then, letting $b= q(a_{\sigma(y)})$ and identifying $q(\sigma(y))$ with $\Delta^m$, we have 
 \begin{equation*}
 \begin{split}
  d(h_t(x),h_t(y))&\leq D\Bigl((1-t)\,\abs{q(x)-q(y)} + t\,\abs{\rho_b(q(x)) -\rho_b(q(y))}\Bigr)\\
         &\leq D^2\biggl(1+\frac{tKD}{\min\bigl\{d(x,a_\sigma),d(y,a_\sigma)\bigr\}}\biggr)\, d(x,y),
 \end{split}
 \end{equation*}
 which implies \(d(h_t(x),h_t(y))\leq D^2(1 + tKD \min\{u(x),u(y)\}^{-1}) d(x,y)\).
As a result, $\lip h_t(x)\leq K''(1 + t u(x)^{-1}) \leq 2K''u(x)^{-1}$ for some $K''=K''(D,m,c)$.
 Next, since $\{u\leq r\} = \bigcup B(a_\sigma, r)$ for sufficiently small $r>0$, it follows that $\{u>r\}$ is quasiconvex and hence, by Lemma~\ref{lem:littleLip}, $h$ is Lipschitz on $[0,1]\times\{u>r\}$ for sufficiently small $r>0$. 
 
 Finally, let $T^i_r$, $S_r$, $R_r$ be defined as at the beginning of Section~\ref{sec:limit-homotopy}. By the properties of $h$ listed above, $T^0_r = T\on\{u>r\}$, $\spt T^1_r\subset X^{(m-1)}$, and $\spt T^1_r$, $\spt S_r\subset \Hull(\spt T)$ as well as $\spt R_r\subset \Hull(\spt \partial T)$ for almost every $r>0$.
By Lemma~\ref{lem:GeometryOfDeltaN}, it follows that \(q\bigl(\Hull(\spt T) \bigr)\) is a countable simplicial complex and thus \(\mathcal{A}\coloneqq \bigl\{ \interior \sigma : \sigma \in \mathcal{F} \text{ and } \sigma\subset \Hull(\spt T) \bigr\}\) is countable. Applying Proposition~\ref{prop:HomotopyFormula-new} with the countable Borel partition \(\mathcal{A} \cup \{X\setminus \Hull(\spt T)\}\) shows that there exists a sequence of admissible $r_s\searrow 0$ such that the currents $T^i_{r_s}$, $S_{r_s}$, $R_{r_s}$ converge to currents $T^i$, $S$, $R$ satisfying the properties listed in Proposition~\ref{prop:HomotopyFormula-new}. 

Set $P\coloneqq T^1$. It follows that $T^0=T$, $\spt P\subset X^{(m-1)}$, and \(\spt P\), \(\spt S \subset \Hull(\spt T)\) as well as $ \spt R\subset \Hull(\spt \partial T)$. Since $T = P - R-\partial S$ it moreover follows that $\spt \partial P\subset \spt\partial T\cup \spt\partial R\subset \Hull(\spt\partial T)$. The remaining properties of $P$, $S$, $R$ listed in the statement of the proposition easily follow from Proposition~~\ref{prop:HomotopyFormula-new} and the properties of $h$ listed above.
\end{proof}

When $k=m$, we can use similar techniques to approximate $T$ by $P\in \poly_m(X)$.
\begin{proposition}\label{prop:final-inductive-step}
Given $m\geq 1$ and $D,c\geq 1$ there is a positive constant $C$ with the following property. Let $X$ be a complete $c$-quasiconvex metric space admitting an $(m,D,\varepsilon)$-triangulation and $T\in\bI_m(X)$. Then there exist $P,R\in\bI_m(X)$ with $P\in \poly_m(X)$ such that $T=P+R$ as well as 
 \begin{align*}
 \mass(P)&\leq C\mass(T),&
 \mass(\partial P)&\leq C \mass(\partial T),&
 \mass(R) &\leq \varepsilon\, C\mass(\partial T)
 \end{align*}
 and $\spt R, \, \spt \partial P\subset \Hull(\spt\partial T)$, and $\spt P \subset \Hull(\spt T)$.
 Moreover, 
 \begin{align*}
   \lVert P\rVert(\sigma) & \leq C\,\lVert T\rVert(\sigma),&
   \lVert \partial P\rVert (\tau) & \leq C\,\lVert \partial T\rVert(\st\tau), &
   \lVert R\rVert(\sigma) &\leq \varepsilon \, C \,\lVert \partial T\rVert(\interior{\sigma}) &
 \end{align*}
 for every \(\sigma\in \mathcal{F}_m\) and every \(\tau\in \mathcal{F}_{m-1}\).
\end{proposition}
\begin{proof}
  We may suppose that $\varepsilon=1$. After possibly changing the metric on $X$ in a bilipschitz way we may assume by Corollary~\ref{cor:triangulation-bilip-length} that $X$ is isometric to $\Sigma$ equipped with the length metric. Let $q\colon X\to\Sigma$ be as above and let $\sigma\in \mathcal{F}_m$. Then $q|_\sigma$ is an isometry and we may view it as a map from $\sigma$ to $\Delta^m$. There exists a function $\theta_\sigma \in L^1(\sigma,\Z)$ such that $T\on \interior{\sigma} = \llbracket \theta_\sigma \rrbracket$ and $\|\theta_\sigma\|_1 \le K_1 \lVert T\rVert(\sigma)$ for some $K_1=K_1(m)$, where $\llbracket \theta_\sigma \rrbracket$ is the $m$-current in $X$ given by integrating $\theta_\sigma$ along $\sigma$.
  
  By the discussion preceding the proof of Proposition~\ref{prop:inductive-step}, there exists $K_2=K_2(m)$ such that for every $\sigma\in \mathcal{F}_m$ there are a point $a_\sigma\in\interior{\sigma}$ such that 
  \begin{equation}\label{eq:singular-T}
    \int_{\interior{\sigma}}d(x,a_\sigma)^{1-m}\,d\lVert \partial T\rVert (x)\leq K_2 \lVert \partial T\rVert(\interior{\sigma})
  \end{equation}
  and an integer $z_\sigma$ such that $|z_\sigma| \le K_2 \lVert T\rVert(\sigma)$ and
  \begin{equation}\label{eq:density-point-T}
    \lim_{r\to 0} \frac{1}{r^m} \int_{U(a_\sigma,r)} |\theta_\sigma- z_\sigma|\,d\hspace{-0.2em}\Haus^m =0.
  \end{equation}
 In particular, $z_\sigma= 0$ for all but finitely many $\sigma\in\mathcal{F}_m$.
  We claim that 
  $$P=\sum_{\sigma\in \mathcal{F}_m} z_\sigma \llbracket\sigma\rrbracket$$
  and $R=T-P$ satisfy the proposition. By construction, for every $\sigma\in \mathcal{F}_m$, we have $\lVert P\rVert(\sigma) \leq C\,\lVert T\rVert(\sigma)$, so $\spt P \subset \Hull(\spt T)$ and $\mass(P)\leq C\mass(T)$ for some $C=C(m)$. 
  
  Let $u$ and $h$ be as above.   
  We bound $\partial P$ and $R$ by applying Proposition~\ref{prop:HomotopyFormula-new} to $\partial T$. For every admissible $r>0$, let
  \begin{align*}
    (\partial T)_r & \coloneqq h_{1\,\#}(\partial T\on\{u>r\}), &
    S_r&\coloneqq h_\#\bigl(\bb{0,1}\times (\partial T\on\{u>r\})\bigr).
  \end{align*}
  As above, there is a $K_3=K_3(m)$ such that for every $\tau\in \mathcal{F}_m$,
  $$\int_{\interior{\tau}}u(x)^{1-m}\,d\lVert \partial T\rVert (x)\leq K_3\lVert \partial T\rVert(\interior{\tau}).$$
  Therefore, by Proposition~\ref{prop:HomotopyFormula-new} applied to $\partial T$, there exist admissible $r_i\searrow 0$, $Q\in \bI_{m-1}(X)$, and $S\in \bI_{m}(X)$ with $\spt Q\subset X^{(m-1)}$ and such that $(\partial T)_{r_i}\to Q$, $S_{r_i}\to S$, $\partial S=Q-\partial T$, and
  \begin{align}\label{eq:local-mass-final-ind}
    \lVert S \rVert (\interior{\sigma}) & \leq C \lVert \partial T\rVert (\interior{\sigma}) & 
    \lVert Q \rVert (\interior{\tau}) &\leq C \lVert \partial T\rVert (\st\tau)
  \end{align}
  for every $\sigma\in \mathcal{F}_{m}$ and every $\tau\in \mathcal{F}_{m-1}$. We let $R=S$ and claim that $S = P-T$; it follows that $Q=\partial P$, and \eqref{eq:local-mass-final-ind} implies the desired bounds on $\lVert R\rVert$ and $\lVert \partial P\rVert$.
  
  For sufficiently small $r>0$, let $g_r \colon X\to X$ be the continuous map such that $$q(g_r(x))=(1-t_r(x))q(a_\sigma) + t_r(x)\varrho_{q(a_\sigma)}(q(x))$$ for all $x\in\sigma\setminus\{a_\sigma\}$, where $t_r(x) = \min\{1,r^{-1} u(x)\}$. Let moreover $h^r\colon [0,1]\times X \to X$ be the straight-line homotopy from $\id_X$ to $g_r$, i.e., $q(h^r(t,x))=(1-t)q(x) + tq(g_r(x))$. Then for $\sigma\in \mathcal{F}_m$, we have $g_r(\sigma\cap \{u\ge r\}) = h_1(\sigma\cap \{u\ge r\}) \subset \partial \sigma$ and $g_r(\sigma\cap \{u\le r\}) = \sigma$. Furthermore, it follows as in the proof of Proposition~\ref{prop:inductive-step} that there exists $K_4=K_4(m)$ such that $\lip g_r(x)\leq K_4r^{-1}$ and $\lip h^r_t(x) \le K_4 r^{-1}$ as well as $\lip h^r_x(t) \le K_4$, where we have used the notation $h^r_t(x)=h^r(t,x)=h^r_x(t)$ as above. In particular, $g_r$ and $h^r$ are Lipschitz. 
  
  The currents $U_i = g_{r_i\,\#}(T)$ and $V_i= h^{r_i}_{\#}(\bb{0,1}\times \partial T)$ satisfy $$U_i = T + V_i$$ because $W_i=h^{r_i}_\#(\bb{0,1}\times T)$ is an $(m+1)$-chain in $X$, hence $W_i=0$ and therefore
  $$0 = \partial W_i = h^{r_i}_{1\,\#}(T) - h^{r_i}_{0\,\#}(T) - h^{r_i}_\#(\bb{0,1}\times \partial T) = U_i - T - V_i.$$
  Since $h^r(t,x)=h(t,x)$ whenever $u(x)>r$, we have $V_i - S_{r_i} = h^{r_i}_{\#}(\bb{0,1}\times (\partial T \on \{u\le r\}))$. Therefore, letting \(K_5\coloneqq m K_4^m\),
  \begin{multline*}
    \mass(V_i-S_{r_i}) \le K_5 r_i^{1-m}\mass(\partial T \on \{u\le r_i\}) \\
    \le K_5 \sum_{\sigma \in \mathcal{F}_m} \int_{\sigma\cap \{u\le r_i\}} d(x,a_\sigma)^{1-m}\,d\lVert \partial T\rVert (x) \le K_5 K_2\mass(\partial T),
  \end{multline*}
  where we use \eqref{eq:singular-T} in the last inequality. As $i\to \infty$, each term in the sum goes to zero, so by dominated convergence, $\mass(S_{r_i} - V_i)\to 0$. Thus $V_i\to S$ in mass and $(U_i)$ converges in mass too.
 Finally, since $g_r$ is injective on $U(a_\sigma, r)$, we have for each $\sigma\in \mathcal{F}_m$ that
  $U_i\on \interior{\sigma} = \llbracket \theta_\sigma \circ g_{r_i}^{-1}|_{\interior{\sigma}} \rrbracket.$
  By \eqref{eq:density-point-T}, for every $\sigma\in\mathcal{F}_m$ we have that $\|\theta_\sigma \circ g_{r_i}^{-1}|_{\interior{\sigma}} - z_\sigma\|_1\to 0$ as $i\to \infty$. Since $(U_i)$ converges in mass this implies that $U_i\to P$ in mass. This shows that $S = P - T$, as desired.
\end{proof}

\begin{remark}\label{rem:zero-current-above-dim}
 It is not difficult to see that if $X$ is an $(n,D,\varepsilon)$-triangulated metric space and $T\in\bI_k(X)$ for some $k>n$ then $T=0$. Indeed, for every simplex $\sigma$ in $X$ we have $\lVert T\rVert(\sigma)=0$ because the Hausdorff dimension of $\sigma$ is strictly smaller than $k$. Now, since $\spt T$ is separable it follows from Lemma~\ref{lem:GeometryOfDeltaN} that $\Hull(\spt T)$ is a countable simplicial complex and hence $\mass(T)=\lVert T\rVert(\Hull(\spt T)) =0$. Thus, \(T=0\), as desired.
\end{remark}

\subsection{Proof of the deformation theorem}

Let $p\colon\Sigma\to X$ be an $(n,D,\varepsilon)$-triangulation of $X$ and  $T\in\bI_k(X)$ for some $1\leq k\leq n$. It is not difficult to show that for each $i=1,\dots, n-k$ the restriction $p|_{\Sigma^{(n-i)}}$ is an $(n-i,D,\varepsilon)$-triangulation of $X^{(n-i)}$ and that $X^{(n-i)}$ is quasiconvex with constant only depending on $c$, $n$, and $D$.

In the following, $C$ and $C_i$ will denote a constant depending on $D$, $c$, $n$ and the constant may change from one occurrence to another. Set $P_0\coloneqq T$. By Proposition~\ref{prop:inductive-step}, applied to $X^{(n-i+1)}$ and $P_{i-1}$ for $i=1,\dots, n-k$, we have $P_{i-1} = P_i + R_i + \partial S_i$  for some $P_i, R_i\in\bI_k(X^{(n-i+1)})$, $S_i\in\bI_{k+1}(X^{(n-i+1)})$ such that $\spt P_i\subset X^{(n-i)}$ and 
\begin{align}
\label{eq:intermediate-global-estimates}\mass(P_i)&\leq C_i\mass(T),&\quad \mass(\partial P_i)&\leq C_i\mass(\partial T),\;\\ \mass(S_i)&\leq \varepsilon \,C_i\mass(T),&\quad \mass(R_i)&\leq \varepsilon \,C_i\mass(\partial T).\notag
\end{align}
Moreover, $\spt P_i, \spt S_i\subset \Hull(\spt T)$ and $\spt \partial P_i, \spt R_i\subset \Hull(\spt\partial T)$ and the following property holds. For each simplex $\sigma\in\mathcal{F}$ we have 
\begin{align}\label{eq:intermediate-local-estimates}
\lVert P_i\rVert(\interior{\sigma}) &\leq C \lVert P_{i-1}\rVert(\st{\sigma}),&\quad \ \ \lVert \partial P_i\rVert (\interior{\sigma}) &\leq C \lVert \partial P_{i-1}\rVert(\st{\sigma}), \\
\lVert S_i\rVert(\interior{\sigma}) &\leq \varepsilon C \lVert P_{i-1}\rVert(\st{\sigma}),& \quad \lVert R_i\rVert(\interior{\sigma}) &\leq \varepsilon C \lVert \partial P_{i-1}\rVert(\st{\sigma}). \nonumber
\end{align}
By Proposition~\ref{prop:final-inductive-step}, we may write $P_{n-k} = P_{n-k+1} + R_{n-k+1}$ with $P_{n-k+1} \in \poly_k(X)$ and $R_{n-k+1}\in\bI_k(X)$ with $\spt R_{n-k+1}\subset X^{(k)}$. These likewise satisfy $\spt P_{n-k+1} \subset \Hull(\spt T)$ and $\spt \partial P_{n-k+1}, \, \spt R_{n-k+1} \subset \Hull(\spt\partial T)$, and their masses satisfy \eqref{eq:intermediate-global-estimates} and \eqref{eq:intermediate-local-estimates}. Let $P\coloneqq P_{n-k+1}$, $R\coloneqq R_1+\dots+R_{n-k+1}$, and $S\coloneqq S_1+\dots+S_{n-k}$. Then $T = P+R+\partial S$ and \eqref{eq:massEstimates} as well as \(\spt P, \, \spt S\subset \Hull(\spt T)\) and \(\spt \partial P,\, \spt R\subset\Hull(\spt\partial T)\) are satisfied. Finally, we show the local mass estimates in \eqref{eq:MassOfFace}.
For any $i\ge 1$ and any $\sigma\in \mathcal{F}$,
$$\lVert P_i\rVert(\interior{\sigma}) \le C \lVert P_{i-1}\rVert(\st{\sigma}) = C \sum_{\sigma\subset \sigma_1} \lVert P_{i-1}\rVert(\interior{\sigma_1}),$$
where $\sigma_1\in \mathcal{F}$. By induction,
$$\lVert P_i\rVert(\interior{\sigma}) \le C^i \sum_{\sigma\subset \sigma_1\subset \dots \subset \sigma_i} \lVert T\rVert(\interior{\sigma_i}),$$
where the sum is taken over all chains of simplices $\sigma_i\in \mathcal{F}$. Let $N$ be the number of chains of simplices $\sigma_0 \subset \sigma_1\subset \dots \subset \sigma_{n-k} \subset \Delta^n$. Then for any $\sigma\in \mathcal{F}$, 
$$\lVert P \rVert(\interior{\sigma}) \le C^{n-k+1} N \lVert T \rVert(\st{\sigma}).$$
The same argument with $P_i$ replaced by $\partial P_i$ gives
$$\lVert \partial P \rVert(\interior{\sigma}) \le C^{n-k+1} N \lVert \partial T \rVert(\st{\sigma}).$$
The estimates of \(\norm{S}\) and \(\norm{R}\) then follow from \eqref{eq:intermediate-local-estimates}.

\bibliographystyle{plain}

\end{document}